\documentclass[draft]{amsart}

\usepackage{amsmath}
\usepackage{amssymb}
\usepackage{array}
\usepackage{amscd}
\usepackage{graphicx}
\input xy
\xyoption{all}   
\usepackage{amsfonts}
\usepackage{a4}
\usepackage{latexsym, amsthm, mathrsfs}
\usepackage[latin1]{inputenc}
\usepackage{fancybox}
\usepackage{amscd}
\usepackage[german, english]{babel}
\input{xy}
\usepackage[all]{xy}

\usepackage{amssymb}
\usepackage{url} 

\newcommand{\al}{\alpha}
\newcommand{\om}{\omega}

\newcommand{\sse}{\subseteq}
\newcommand{\contains}{\supseteq}

\DeclareMathOperator{\dom}{dom}

\DeclareMathOperator{\FIN}{FIN}

\newcommand{\rgl}{\rangle}
\newcommand{\lgl}{\langle}

\newcommand{\re}{\restriction}

\newcommand{\ra}{\rightarrow}

\theoremstyle{plain}
\newtheorem{thm}{Theorem}  
\newtheorem{prop}[thm]{Proposition} 
\newtheorem{lem}[thm]{Lemma} 
\newtheorem{cor}[thm]{Corollary} 
\newtheorem{fact}[thm]{Fact}

\theoremstyle{definition}   
\newtheorem{defn}[thm]{Definition}

\newtheorem{question}[thm]{Question}
\newtheorem{problem}[thm]{Problem}

\newtheorem*{notn}{Notation}

\theoremstyle{remark} 
\newtheorem*{rem}{Remark}

\newtheorem*{claim}{Claim}

\newtheorem*{claim1}{Claim 1}
\newtheorem*{claim2}{Claim 2}

\newtheorem*{subclaimn}{Subclaim}

\begin{document}





\title{Tukey types  of ultrafilters}


%
\author{Natasha Dobrinen} 
\address{Department of Mathematics\\
University of Denver\\
2360 S Gaylord St, Denver, CO 80208\\ 
USA} 
\email{natasha.dobrinen@du.edu} 
\thanks{The first author was supported by her University of Denver, Department of Mathematics start-up fund  and by an NSERC grant of the second author.}

\author{Stevo Todorcevic}
\address{Department of Mathematics\\
University of Toronto\\
40 St George St, Toronto, ON\\
 CANADA M5S 2E4}
\email{stevo@utoronto.edu}
\thanks{The second author was supported by grants from NSERC and CNRS}



%


%

\begin{abstract}
We investigate the structure of the Tukey types of ultrafilters on countable sets partially ordered by reverse inclusion.
A canonization of cofinal maps from a p-point into another ultrafilter is obtained. 
This is used in particular to study the Tukey types of p-points and selective ultrafilters.
Results fall into three main categories:  comparison to a basis element for selective ultrafilters, embeddings of chains and antichains into the Tukey types,
and Tukey types generated by block-basic ultrafilters on $\FIN$.
\end{abstract}

\maketitle


\section{Introduction}\label{sec.intro}

Let $D$ and $E$ be partial orderings.
We say that a function $f:E\ra D$ is {\em cofinal} if the image of each cofinal subset of $E$ is cofinal in $D$.
We say that $D$ is {\em Tukey reducible} to $E$, and write $D\le_T E$, if there is a cofinal map from $E$ to $D$.
An equivalent formulation of Tukey reducibility was noticed by Schmidt in \cite{Schmidt55}.
Given partial orderings $D$ and $E$, a map
$g:D\ra E$ such that the image of each unbounded subset of $D$ is an unbounded subset of $E$ is called a {\em Tukey map} or an {\em unbounded map}.
$E\ge_T D$ iff there is a Tukey map from $D$ into $E$.
If both $D\le_T E$ and $E\le_T D$, then we write $D\equiv_T E$ and say that $D$ and $E$ are Tukey equivalent.
$\equiv_T$ is an equivalence relation, and $\le_T$ on the equivalence classes forms a partial ordering. 
The equivalence classes can be called {\em Tukey types} or  {\em Tukey degrees}.

In \cite{Tukey40}, Tukey introduced the Tukey ordering  to 
develop the notion of Moore-Smith convergence in topology to the more general setting of   directed partial orderings.
The study of cofinal types and Tukey types of partial orderings often reveals useful information for the comparison of different partial orderings.
For example, Tukey reducibility downward preserves calibre-like properties, such as c.c.c., property K, precalibre $\aleph_1$, $\sigma$-linked, and $\sigma$-centered (see \cite{Todorcevic96}).

Satisfactory classification theories of Tukey degrees have been developed for several classes of ordered sets.
The  cofinal types of countable directed systems are $1$ and $\om$ (see \cite{Tukey40}).
Day found a classification of
countable oriented systems (partially ordered sets) in \cite{Day44} in terms of a three element basis.
Assuming PFA, Todorcevic in \cite{TodorcevicDirSets85} 
 classified the Tukey degrees of 
directed partial orderings of cardinality $\aleph_1$ by showing that there are exactly five cofinal types, and in \cite{Todorcevic96} classified the Tukey degrees of oriented systems (partially ordered sets) of size $\aleph_1$
in terms of a basis consisting of five forms of partial orderings.
However, he also showed in \cite{Todorcevic96} that there are at least $2^{\aleph_1}$ many Tukey incomparable  separative $\sigma$-centered partial orderings of size $\mathfrak{c}$.
This would  preclude a satisfactory classification theory of all partial orderings of size continuum.

However, the structure of the Tukey types
of   particular classes of partial orderings of size continuum
can yield useful information.  
This has been fully stressed first in the paper \cite{Fremlin91} by Fremlin who considered partially ordered sets occurring in analysis.
After this, several papers appeared dealing with different classes of posets such as, for example, the paper \cite{Solecki/Todorcevic04} of Solecki and Todorcevic which makes
a  systematic study of the  structure of the Tukey degrees of topological directed sets.
The paper  \cite{Milovich08}  of Milovich is the first paper after Isbell \cite{Isbell65}
to study Tukey degrees of ultrafilters on $\om$.

In this paper, we investigate the  structure  of the Tukey degrees of ultrafilters on $\om$ ordered by reverse inclusion.
For any ultrafilter $\mathcal{U}$  on $\om$,
$(\mathcal{U},\contains)$ is a directed partial ordering.
We remark  that 
for any two directed partial orderings $D$ and $E$, $D\equiv_T E$ iff $D$ and $E$ are {\em cofinally similar}; that is, there is a partial ordering  into which both $D$ and $E$ embed as cofinal subsets \cite{Tukey40}.
So for ultrafilters, Tukey equivalence is the same as cofinal similarity.

Another motivation for 
this study  is that Tukey reducibility is a 
generalization of  Rudin-Keisler reducibility.

\begin{fact}\label{fact.TRK}
Let $\mathcal{U}$ and $\mathcal{V}$ be ultrafilters on $\om$.
If $\mathcal{U}\ge_{RK} \mathcal{V}$, then $\mathcal{U}\ge_T\mathcal{V}$.
\end{fact}

\begin{proof}
Take a function $h:\om\ra\om$ satisfying $\mathcal{V}=h(\mathcal{U}):=\{X\sse\om:h^{-1}(X)\in\mathcal{U}\}$.  
Define $f:\mathcal{U}\ra\mathcal{V}$ by $f(X)=\{h(n):n\in X\}$, for each $X\in\mathcal{U}$.
Then $f$ is a cofinal map. 
\end{proof}

Thus arises the question:  How different are Tukey and Rudin-Keisler reducibility?
We shall study this question particularly for p-points.


\section{Notation and basic facts}
In this section, we fix notation and provide some basic facts.
All ultrafilters in this paper  have a base set which is countable.  The base set  will usually be $\om$, but in Section \ref{sec.FIN} we also investigate ultrafilters on $\FIN$, the family of finite, nonempty subsets of $\om$.

\begin{defn}\label{def.cofinal}
Let $(P,\le)$ be a partial ordering.
We say that a subset $C\sse P$ is {\em cofinal} in $P$ if for each $p\in P$ there is a $c\in C$ such that $p\le c$.
We say that $(P,\le)$ is  {\em directed}  if for any $p,r\in P$, there is an $s\in P$ such that $p\le s$ and $r\le s$.
\end{defn}

\begin{fact}\label{fact.cofinalTukeysame}
If $C$ is a cofinal subset of a partial ordering $(P,\le)$,
then $(C,\le)\equiv_T(P,\le)$.
\end{fact}

\begin{proof}
Let $C$ be a cofinal subset of $P$ and let $id_C:C\ra P$ be the identity map.
Then $id_C$ is both a cofinal map and a Tukey map.
For if $D\sse C$ is cofinal in $(C,\le)$,
then $id_C''D=D$ is also cofinal in $(P,\le)$.
If $B\sse P$ is bounded in $(P,\le)$,
then there is a $p\in P$ bounding each element of $B$ from above. 
Take a $c\in C$ such that $p\le c$.
Then $c$ bounds $id_C^{-1}(B)$.  Thus, $id_C$ maps each unbounded subset of $C$ to an unbounded subset of $P$, hence is a Tukey map.
\end{proof}

The partial ordering $\le$ on an ultrafilter $\mathcal{U}$  is $\contains$;
that is, for $X,Y\in\mathcal{U}$,
 $X\le Y$ iff $X\contains Y$.
Note that $(\mathcal{U},\contains)$ is a directed partial ordering.

We now show that, for ultrafilters, there is a nice subclass of cofinal maps, namely the monotone cofinal maps, to which we may restrict our attention.

\begin{defn}\label{defn.monotonemap}
Let $(P,\le_P)$ and $(Q,\le_Q)$ be partial orderings.
A  map $f:P\ra Q$ is {\em monotone}
if whenever $p,r$ are in $P$ and $p\le_P r$, then $f(p)\le_Q f(r)$.
For the special case of ultrafilters $\mathcal{U},\mathcal{V}$, this translates to the following:
a map $f:\mathcal{U}\ra\mathcal{V}$ is {\em monotone} if 
whenever $W,X\in\mathcal{U}$ and $W\contains X$, then $f(W)\contains f(X)$.  
\end{defn}

\begin{fact}\label{fact.monotonemap}
Let $(P,\le_P)$ and $(Q,\le_Q)$ be partial orderings.
A monotone map $f:P\ra Q$ is a cofinal map if and only if its image $f''P$ is a cofinal subset of $Q$.
\end{fact}

\begin{proof}
Let $f:P\ra Q$ be a monotone map.
If $f$ is a cofinal map,
then certainly the image of $P$ under $f$ is a cofinal subset of $Q$.

Conversely, suppose the image $f''P$ is cofinal in $Q$.
Let $C\sse P$ be a cofinal subset of $P$ and let $q\in Q$ be given.
Since $f''P$ is cofinal in $Q$, 
there is a $p\in P$ such that $q\le_Q f(p)$.
Since $C$ is cofinal in $P$, there is a $c\in C$ such that $p\le_P c$.
Since $f$ is monotone,
$q\le_Q f(p)\le_Q f(c)$.
Therefore, $f'' C$ is cofinal in $Q$.
\end{proof}

\begin{fact}\label{fact.monotone}
Let  $\mathcal{U}$ and $\mathcal{V}$ be ultrafilters.
If $\mathcal{U}\ge_T\mathcal{V}$, then this is witnessed by a monotone cofinal map. 
\end{fact}

\begin{proof}
Suppose $\mathcal{U}\ge_T\mathcal{V}$.  Then there is a Tukey map $g:\mathcal{V}\ra\mathcal{U}$ witnessing this.
Define $f:\mathcal{U}\ra\mathcal{V}$ by
$f(U)=\bigcap\{V\in\mathcal{V}:g(V)\contains U\}$.

First we check that $f$ is a function from $\mathcal{U}$ into $\mathcal{V}$.
Let $U\in\mathcal{U}$.
Note that $\{V\in\mathcal{V}:g(V)\contains U\}=
g^{-1}(\{U'\in\mathcal{U}:U'\contains U\})$.
Since the set $\{U'\in\mathcal{U}:U'\contains U\}$
is bounded in $\mathcal{U}$ and $g$ is a Tukey map,
it follows that $\{V\in\mathcal{V}:g(V)\contains U\}$ is bounded in $\mathcal{V}$.
Thus, $\bigcap \{V\in\mathcal{V}:g(V)\contains U\}$ is a member of $\mathcal{V}$.

Next we check that $f$ is monotone.
Let $U\contains U'$ be elements of $\mathcal{U}$.
Then it is the case that 
$\{V\in\mathcal{V}:g(V)\contains U\}\sse
\{V\in\mathcal{V}:g(V)\contains U'\}$.
Thus, 
$f(U)=\bigcap\{V\in\mathcal{V}:g(V)\contains U\}\contains
\bigcap\{V\in\mathcal{V}:g(V)\contains U'\}=f(U')$.

Finally, we show that $f''\mathcal{U}$ is cofinal in $\mathcal{V}$.
Let $V'\in\mathcal{V}$.
Then $g(V')$ is in $\mathcal{U}$; let $U$ denote $g(V')$.
By definition, $f(U)
=\bigcap\{V\in\mathcal{V}:g(V)\contains g(V')\}
\sse V'$.
Thus, by Fact \ref{fact.monotonemap},
$f$ is a monotone cofinal map from $\mathcal{U}$ into $\mathcal{V}$.
\end{proof}

Thus, for ultrafilters, we can restrict ourselves to using monotone cofinal maps.

We now fix some notation for the duration of the paper.
Recall that the partial ordering on a (finite or infinite) cartesian product of partially ordered sets is the coordinate-wise ordering.
Thus, the partial ordering on a cartesian product of directed partial orderings is again a directed partial ordering.

\begin{notn}\label{notation.cross}
Let $\mathcal{U}$, $\mathcal{V}$, and $\mathcal{U}_n$ ($n<\om$) be ultrafilters.
We define the notation for the following ultrafilters. 
\begin{enumerate}
\item
$\mathcal{U}\cdot\mathcal{V}= \{A\sse\om\times\om:\{i\in\om:\{j\in\om:(i,j)\in A\}\in\mathcal{V}\}\in\mathcal{U}\}$.
\item
$\lim_{n\ra\mathcal{U}}\mathcal{U}_n=\{A\sse\om\times\om:\{n\in\om:\{j\in\om:(n,j)\in A\}\in\mathcal{U}_n\}\in\mathcal{U}\}$.
\item
We shall use $\mathcal{U}^2$ to denote $\mathcal{U}\cdot\mathcal{U}$; and more generally, $\mathcal{U}^{n+1}$ shall denote $\mathcal{U}\cdot\mathcal{U}^n$.
We shall use $\mathcal{U}^{\om}$ to denote $\lim_{n\ra\mathcal{U}}\mathcal{U}^{k_n}$,
where $(k_n)_{n<\om}$ is any strictly increasing sequence of natural numbers.
More generally, for any ordinal $\al<\om_1$,
$\mathcal{U}^{\al+1}$ denotes $\lim_{n\ra\mathcal{U}}\mathcal{U}^{\al}$.
For $\al$ a limit ordinal,
$\mathcal{U}^{\al}$ is used to denote any ultrafilter of the form
$\lim_{n\ra\mathcal{U}}\mathcal{U}^{\beta_n}$,
where $(\beta_n)_{n<\om}$ is a strictly increasing sequence of ordinals such that $\sup_{n<\om}\beta_n=\al$.
(So for $\om\le\al<\om_1$, $\mathcal{U}^{\al}$ does not denote a unique ultrafilter, but rather any ultrafilter formed in the way described above.)
\item
$\mathcal{U}\times\mathcal{V}$ is defined to be the ordinary cartesian product of $\mathcal{U}$ and $\mathcal{V}$ with the coordinate-wise ordering $\lgl\contains,\contains\rgl$.
\item
$\Pi_{n<\om}\mathcal{U}_n$ is the cartesian product of the $\mathcal{U}_n$ with its natural coordinate-wise product ordering.
We will let
$\Pi_{n<\om}\mathcal{U}$ denote the cartesian product of $\om$ many copies of $\mathcal{U}$.
\end{enumerate}
\end{notn}

The following basic facts are used  throughout the paper.

\begin{fact}
Let $\mathcal{U},\mathcal{U}_0,\mathcal{U}_1,\mathcal{V},\mathcal{V}_0$, and $\mathcal{V}_1$ be ultrafilters.
\begin{enumerate}
\item
$\mathcal{U}\times\mathcal{U}\equiv_T\mathcal{U}$.
\item
$\mathcal{U}\times \mathcal{V}\ge_T\mathcal{U}$ and 
$\mathcal{U}\times \mathcal{V}\ge_T\mathcal{V}$.
\item
If $\mathcal{U}_1\ge_T \mathcal{U}_0$ and $\mathcal{V}_1\ge_T\mathcal{V}_0$, then $\mathcal{U}_1\times\mathcal{V}_1\ge_T\mathcal{U}_0\times\mathcal{V}_0$.
\item
If $\mathcal{W}\ge_T\mathcal{U}$ and $\mathcal{W}\ge_T\mathcal{V}$,
then $\mathcal{W}\ge_T\mathcal{U}\times\mathcal{V}$.
Thus, $\mathcal{U}\times\mathcal{V}$ is the minimal Tukey type which is Tukey greater than or equal to both $\mathcal{U}$ and $\mathcal{V}$.
\item
 $\mathcal{U}\cdot\mathcal{V}\ge_T \mathcal{U}$ and
 $\mathcal{U}\cdot\mathcal{V}\ge_T \mathcal{V}$,
 and therefore $\mathcal{U}\cdot\mathcal{V}\ge_T\mathcal{U}\times\mathcal{V}$. 
\end{enumerate}
\end{fact}

\begin{proof}
 Let $\pi_1,\pi_2$ denote the projection maps $\pi_i:\om\times\om\ra\om$ ($i=1,2$) given by $\pi_1(m,n)=m$,
 and $\pi_2(m,n)=n$.

(1) $\pi_1$  induces the map $\bar{\pi}_1:\mathcal{U}\times\mathcal{U}\ra\mathcal{U}$, given by $\bar{\pi}_1(U,U')=U$, which is a cofinal map.  Conversely, the map $f(U)=(U,U)$ is a cofinal map from $\mathcal{U}$ into $\mathcal{U}\times\mathcal{U}$.

(2)
Again, the induced map $\bar{\pi}_1:\mathcal{U}\times\mathcal{V}\ra\mathcal{U}$ given by $\bar{\pi}_1(U,V)=U$ is a cofinal map. 
The second part follows since $\mathcal{U}\times\mathcal{V}\equiv_T\mathcal{V}\times\mathcal{U}$.

(3)
Given monotone cofinal maps $f:\mathcal{U}_1\ra\mathcal{U}_0$ and  
$g:\mathcal{V}_1\ra\mathcal{V}_0$,
define the map $h:\mathcal{U}_1\times\mathcal{V}_1\ra\mathcal{U}_0\times\mathcal{V}_0$ by
$h(U,V)=(f(U),g(V))$.
Let $\mathcal{X}$ be a cofinal subset of $\mathcal{U}_1\times\mathcal{V}_1$ 
and let $(A,B)\in\mathcal{U}_0\times\mathcal{V}_0$.
There are $U\in\mathcal{U}_1$ and $V\in\mathcal{V}_1$ such that $f(U)\sse A$ and $g(V)\sse B$.
Since $\mathcal{X}$ is cofinal in $\mathcal{U}_1\times\mathcal{V}_1$, there is some $(U',V')\in\mathcal{X}$ such that $U'\sse U$ and $V'\sse V$.
Since $f$ and $g$ are monotone,
$h(U',V')=(f(U'),g(V'))\ge (f(U),g(V))\ge (A,B)$.
Thus, $h''\mathcal{X}$ is cofinal in $\mathcal{U}_0\times\mathcal{V}_0$.

(4) 
follows immediately from (1) -  (3).

(5)
Define $f:\mathcal{U}\cdot\mathcal{V}\ra\mathcal{U}$
by $f(A)=\{\pi_1(m,n):(m,n)\in A\}$, for each $A\in\mathcal{U}\cdot\mathcal{V}$.
Then $f$ is monotone, and has cofinal range in $\mathcal{U}$. 
Hence, by Fact \ref{fact.monotone},
$\mathcal{U}\cdot\mathcal{V}\ge_T \mathcal{U}$.
(Alternatively, one can just note that the map $\pi_1$ is a Rudin-Keisler map from $\mathcal{U}\cdot\mathcal{V}$ to $\mathcal{U}$; and hence 
$\mathcal{U}\cdot\mathcal{V}\ge_T \mathcal{U}$.)

Let $g:\mathcal{U}\cdot\mathcal{V}\ge_T \mathcal{V}$ be defined by $g(A)=\{\pi_2(m,n):(m,n)\in A\}$, for each $A\in\mathcal{U}\cdot\mathcal{V}$.
Then $g$ is monotone and has cofinal range in $\mathcal{V}$, hence is a cofinal map.
\end{proof}

\begin{rem}\label{rem.dots}
One cannot conclude from the above that $\mathcal{U}\cdot\mathcal{V}\equiv_T\mathcal{U}\times\mathcal{V}$.
Section \ref{sec.omom} contains an investigation into this matter.
\end{rem}

At this point, we recall the definitions of the following special ultrafilters.  
All these definitions can found in \cite{Bartoszynski/JudahBK}.
Recall the standard notation $\sse^*$,
where for $X,Y$ in an ultrafilter $\mathcal{U}$, we write
$X\sse^* Y$ to denote that $|X\setminus Y|<\om$.

\begin{defn}\label{defn.uftypes}
Let $\mathcal{U}$ be an ultrafilter.
\begin{enumerate}
\item
$\mathcal{U}$ is {\em selective} if 
for every function $f:\om\ra\om$,
there is an $X\in\mathcal{U}$ such that either $f\re X$ is constant or $f\re X$ is one-to-one.
\item
$\mathcal{U}$ is a {\em p-point} if for every family $\{X_n:n<\om\}\sse\mathcal{U}$
there is an $X\in\mathcal{U}$ such that $X\sse^* X_n$ for each $n<\om$.
\item
$\mathcal{U}$ is a {\em q-point} if for each partition of $\om$ into finite pieces $\{I_n:n<\om\}$,
there is an $X\in\mathcal{U}$ such that $|X\cap I_n|\le 1$ for each $n<\om$.
\item
$\mathcal{U}$ is {\em rapid} if 
for each function $f:\om\ra\om$,
there exists an $X\in\mathcal{U}$ such that $|X\cap f(n)|\le n$ for each $n<\om$.
\end{enumerate}
\end{defn}

The following well-known implications can  be found in \cite{Bartoszynski/JudahBK}.

\begin{thm}
\begin{enumerate}
\item
An ultrafilter is  selective if and only if it is both a p-point and  a q-point.
\item
Every q-point is rapid.
\end{enumerate}
\end{thm}

We point out that all of these special ultrafilters exist under CH, under  MA, and even under  weaker assumptions  involving cardinal invariants.
However, the existence of selective ultrafilters, p-points, q-points, or even rapid ultrafilters does not follow from ZFC.
We refer the interested reader to \cite{Bartoszynski/JudahBK} for further exposition on these topics.

We point out the next fact, since it is useful to know, especially in Section \ref{sec.omom}.

\begin{fact}\label{dotnotppoint}
For any ultrafilter $\mathcal{U}$,
$\mathcal{U}\cdot\mathcal{U}$ is not a p-point.
\end{fact}

\begin{proof}
If $\mathcal{U}$ is principle, generated by $\{n\}$,
then $\mathcal{U}\cdot\mathcal{U}$ is also principle, 
generated by $\{(n,n)\}$.
If $\mathcal{U}$ is not principle, then it contains the Fr\'{e}chet filter.
For each $n<\om$, let $A_n=[n,\om)\times\om$.
Then each $A_n$ is in $\mathcal{U}$.
However, there is no $B\in\mathcal{U}\cdot\mathcal{U}$ such that $B\sse^* A_n$ for all $n<\om$; for if $B\sse^* A_n$ for all $n<\om$,
then for each $n$ there could only be finitely many $j$ such that $(n,j)\in B$.
\end{proof}

A word about the top Tukey type for ultrafilters.
The directed set $([\mathfrak{c}]^{<\om},\sse)$ is the maximal Tukey type among all directed partial orderings of cardinality $\mathfrak{c}$.

\begin{fact}\label{fact.c<omtop}
Let $(X,\le)$ be any directed partial ordering of cardinality $\mathfrak{c}$.
Then\\
 $(X,\le)\le_T([\mathfrak{c}]^{<\om},\sse)$.
\end{fact}

\begin{proof}
Let $g:X\ra[\mathfrak{c}]^{<\om}$ be any one-to-one function.  
Then $g$ is a Tukey map.
To see this, 
let $W$ be any unbounded subset of $X$. 
Then in particular, $W$ must be infinite, since every finite subset of $X$ is bounded since $X$ is directed.
Since $g$ is one-to-one, the image $g''W$ is also infinite.
Every infinite subset of $[\mathfrak{c}]^{<\om}$ is unbounded, so $g''W$ is unbounded.
\end{proof}

The following combinatorial characterization of when an ultrafilter has top Tukey type is useful.

\begin{fact}\label{fact.Tukeytopchar}
Let $\mathcal{U}$ be an ultrafilter.
$(\mathcal{U},\contains)\equiv_T([\mathfrak{c}]^{<\om},\sse)$ if and only if there is a subset $\mathcal{X}\sse\mathcal{U}$ such that $|\mathcal{X}|=\mathfrak{c}$ and for each infinite $\mathcal{Y}\sse\mathcal{X}$,
$\bigcap\mathcal{Y}\not\in\mathcal{U}$.
\end{fact}

\begin{proof}
We first show the foreword direction by contrapositive.
Suppose that there is no subset 
$\mathcal{X}\sse\mathcal{U}$ such that $|\mathcal{X}|=\mathfrak{c}$ and for each infinite $\mathcal{Y}\sse\mathcal{X}$,
$\bigcap\mathcal{Y}\not\in\mathcal{U}$.
Then for each subset $\mathcal{X}\sse\mathcal{U}$ such that $|\mathcal{X}|=\mathfrak{c}$, there is an infinite $\mathcal{Y}\sse\mathcal{X}$ such that $\bigcap\mathcal{Y}\in\mathcal{U}$.
We shall show that there is no Tukey map from 
$([\mathfrak{c}]^{<\om},\sse)$ into $(\mathcal{U},\contains)$.

Let $g:([\mathfrak{c}]^{<\om},\sse)\ra(\mathcal{U},\contains)$ be given.
If the range of $g$ is countable, then there is an uncountable subset $\mathcal{C}\sse[\mathfrak{c}]^{<\om}$  and a $U\in\mathcal{U}$ such that $g''\mathcal{C}=\{U\}$.
So $g$ maps an unbounded set to a bounded set, hence is not a Tukey map.
Otherwise, the range of $g$ is uncountable.
By our hypothesis, there is an infinite set $\mathcal{Y}\sse g''[\mathfrak{c}]^{<\om}$ such that 
$\bigcap\mathcal{Y}\in\mathcal{U}$.
Letting $\mathcal{C}$ be the $g$-preimage of $\mathcal{Y}$, we see that $\mathcal{C}$ is infinite, hence unbounded.
Thus, $g$ is not a Tukey map.
Therefore, $([\mathfrak{c}]^{<\om},\sse)\not\le_T(\mathcal{U},\contains)$.

Suppose there is a subset $\mathcal{X}\sse\mathcal{U}$ such that $|\mathcal{X}|=\mathfrak{c}$ and for each infinite $\mathcal{Y}\sse\mathcal{X}$,
$\bigcap\mathcal{Y}\not\in\mathcal{U}$.
By Fact \ref{fact.c<omtop}, we know that 
$(\mathcal{U},\contains)\le_T([\mathfrak{c}]^{<\om},\sse)$, so it remains to show that 
$(\mathcal{U},\contains)\ge_T([\mathfrak{c}]^{<\om},\sse)$.
Let $g:[\mathfrak{c}]^{<\om}\ra\mathcal{X}$ be any one-to-one function.
Let $Z\sse[\mathfrak{c}]^{<\om}$ be unbounded.
Then $Z$ must be infinite, since 
$([\mathfrak{c}]^{<\om},\sse)$ is directed.
Since $g$ is one-to-one,
$g''Z$ is an infinite subset of $\mathcal{X}$.
Thus, $\bigcap g'' Z$ is not in $\mathcal{U}$, so $g'' Z$ is unbounded in $(\mathcal{U},\contains)$.
Therefore, $g$ is a Tukey map.
\end{proof}


\section{Basic and basically generated ultrafilters}\label{sec.basic}

The following type of partial ordering was introduced by Solecki and Todorcevic in \cite{Solecki/Todorcevic04}.

\begin{defn}[\cite{Solecki/Todorcevic04}]\label{defn.basic}
Let $D$ be a separable metric space and let $\le$ be a partial ordering on $D$.  We say that $(D,\le)$ is \em basic \rm if
\begin{enumerate}
\item
each pair of elements of $D$ has the least upper bound with respect to $\le$ and the binary operation of least upper bound from $D\times D$ to $D$ is continuous;
\item
each bounded sequence has a converging subsequence;
\item
each converging sequence has a bounded subsequence.
\end{enumerate}
\end{defn}

Each ultrafilter is a separable metric space using the metric inherited from $\mathcal{P}(\om)$ viewed as the Cantor space, and recall that we define $\le$ on an ultrafilter to be $\contains$.
In this context, a sequence $(W_n)_{n<\om}$ of elements of $\mathcal{P}(\om)$ converges to $W\in\mathcal{P}(\om)$ iff for each $m$ there is some $k$ such that for each $n\ge k$, $W_n\cap m=W\cap m$.
It is not hard to see that every bounded subset of an ultrafilter has a convergent subsequence.
Thus, an ultrafilter is basic iff (3) holds.

The next theorem shows that the basic ultrafilters are exactly the p-points. 
We recall the following characterization of non-meager ideals, which can be found in \cite{Jalali-Naini76} or \cite{Talagrand80}.
An ideal $\mathcal{I}\sse\mathcal{P}(\om)$ is called {\em unbounded} if for each strictly increasing sequence of natural numbers $(n_i)_{i<\om}$,
there is an $X\in\mathcal{I}$ such that $[n_i,n_{i+1})\sse X$ for infinitely many $i<\om$.
It was shown in \cite{Jalali-Naini76} that an ideal is unbounded if and only if it is nonmeager (as a subset
 of $\mathcal{P}(\om)$ with the topology inherited from the Cantor space).

\begin{thm}\label{thm.stevosummer}
An ideal $\mathcal{I}$ on $\mathcal{P}(\om)$ containing all finite subsets of $\om$ is basic relative to the Cantor topology iff $\mathcal{I}$ is a non-meager p-ideal.
Hence, an ultrafilter is basic iff it is a p-point.
\end{thm}

\begin{proof}
Let $\mathcal{I}$ be an ideal on $\mathcal{P}(\om)$ containing all finite subsets of $\om$.

Assume $\mathcal{I}$ is basic.
Let $\lgl n_k:k<\om\rgl$ be an increasing sequence of integers.
Note that each $[n_k,n_{k+1})\in\mathcal{I}$, since Fin $\sse \mathcal{I}$.
 $[n_k,n_{k+1})\ra\emptyset$; so by basicness, there is a subsequence whose union is in $\mathcal{I}$.
Hence, $\mathcal{I}$ unbounded, and thus is nonmeager.

Let $\{A_n:n<\om\}\sse\mathcal{I}$.
We can assume that for each $n<\om$,
$A_n\sse A_{n+1}$.
Let $A'_n=A_n\setminus n$.
Then $A'_n\sse A_n$, so $A'_n\in\mathcal{I}$.
$A'_n\ra\emptyset$ in the Cantor topology, so since $\mathcal{I}$ is basic,
there is a subsequence $n_k$ such that $\bigcup_{k<\om}A'_{n_k}\in\mathcal{I}$.
Let $A=\bigcup_{k<\om}A_{n_k}$.
Then for each $n<\om$, 
$A_n\sse^* A$, since for each $n$ there is an $n_k>n$ such that $A_n\sse A_{n_k}\sse^* A'_{n_k}\sse A$.
Thus, $\mathcal{I}$ is a p-ideal.

Now suppose $\mathcal{I}$ is a nonmeager p-ideal.
Suppose $A_n, A\in\mathcal{I}$ and $A_n\ra A$ in the Cantor topology. 
Take $B\in\mathcal{I}$ such that for each $n$, $A_n\sse^* B$.
Let $m_k$ be a strictly increasing sequence such that $m_0=0$ 
and 
\begin{enumerate}
\item
$n\ge m_{k+1}$ implies $A_n\cap m_k=A\cap m_k$, and 
\item
 $n\le m_k$ implies
$A_n\setminus m_{k+1}\sse B$.
\end{enumerate}
Since $\mathcal{I}$ is nonmeager, there is a subsequence $(m_{k_i})_{i<\om}$ of $(m_k)_{k<\om}$ such that 
$C:=\bigcup_{i<\om}[m_{k_i},m_{k_i+2})\in\mathcal{I}$.
Let $X=\bigcup_{i<\om}A_{m_{k_i}+1}$.

We claim that $X\sse A\cup B\cup C$.
Let $i<\om$ be given.
Then $A_{m_{k_i}+1}\cap m_{k_i} = A\cap m_{k_i}$, by (1).
$A_{m_{k_i}+1}\cap[m_{k_i},m_{k_i +2})\sse C$, since $C$ contains the interval $[m_{k_i},m_{k_i +2})$.
Finally, $A_{m_{k_i}+1}\setminus m_{k_i+2}\sse B$, by (2).
Thus, $A_{m_{k_i}+1}\sse A\cup B\cup C$.
Since $i$ was arbitrary, we have the desired conclusion that $X\sse A\cup B\cup C$, and hence $X\in\mathcal{I}$.
Therefore, $\mathcal{I}$ is basic, since every convergent sequence of elements of $\mathcal{I}$ has a bounded subsequence.
\end{proof}

\begin{rem}
From the proof, we can see that an ultrafilter is basic iff every sequence which converges to $\om$ has a bounded subsequence.
\end{rem}

The next definition gives a notion of ultrafilters which is weaker than p-point.

\begin{defn}
We say that an ultrafilter $\mathcal{U}$ on $\mathcal{P}(\om)$ is \em basically generated \rm if it has a filter basis $\mathcal{B}\sse\mathcal{U}$ (i.e.\ $\forall A\in\mathcal{U}$ $\exists B\in\mathcal{B}$ $B\sse A$)
with the property that each sequence 
$\{A_n:n<\om\}\sse\mathcal{B}$ converging to an element of $\mathcal{B}$
has a subsequence $\{A_{n_k}:k<\om\}$ such that $\bigcap_{k<\om}A_{n_k}\in\mathcal{U}$.
\end{defn}

\begin{thm}\label{thm.2}
Suppose that $\mathcal{U}$ and $\mathcal{U}_n$ ($n<\om$) are basically generated ultrafilters on $\mathcal{P}(\om)$ by filter bases which are closed under finite intersections.  
Then $\mathcal{V}=\lim_{n\ra\mathcal{U}}\mathcal{U}_n$ is basically generated by a filter basis which is closed under finite intersections.
It follows that 
the collection of all ultrafilters  basically generated by some filter base  closed under finite intersections
is closed under  Fubini products. 
\end{thm}

\begin{proof}
Let $\mathcal{B},\mathcal{B}_n$  be filter bases of $\mathcal{U}$, $\mathcal{U}_n$ ($n<\om$)
which are closed under finite intersections and which witness the fact that $\mathcal{U}$, $\mathcal{U}_n$ are basically generated, respectively. 
Let $p_1:\om\times\om\ra\om$ be the projection map onto the first coordinate.
For $A\sse\om\times\om$ and $n<\om$,
let $(A)_n$ denote $\{j<\om:(n,j)\in A\}$.
Let $\mathcal{C}=\{A\in\mathcal{V}:p_1[A]\in\mathcal{B}$ and for each $n<\om$, either $(A)_n=\emptyset$ or $(A)_n\in\mathcal{B}_n\}$.
Then $\mathcal{C}$ is a filter basis for $\mathcal{V}$ which is closed under finite intersections.

Consider a converging sequence $A_n\ra B$ in $\mathcal{C}$.
Note that $p_1[A_n]\ra X$ for some $X\in\mathcal{U}$ containing $p_1[B]$.
$X$ might not be in $\mathcal{B}$, but $p_1[B]$ is in $\mathcal{B}$, since $B\in\mathcal{C}$.
So for each $n<\om$, let $A'_n=A_n\cap(p_1[B]\times \om)$, so that $A'_n\in\mathcal{C}$.
Note that $A'_n\ra B$, $p_1[A'_n]\ra p_1[B]$, and all $p_1[A'_n]\in\mathcal{B}$, since $\mathcal{B}$ is closed under finite intersections.
Since $\mathcal{B}$ witnesses that $\mathcal{U}$ is basically generated, 
there is a subsequence 
of $(p_1[A'_n])_{n<\om}$ whose intersection is in $\mathcal{U}$.
Take such a subsequence and reindex it, so that we have 
$\bigcap_{n<\om}p_1[A'_n]\in\mathcal{U}$.
Let $U$ denote $\bigcap_{n<\om}p_1[A'_n]$.
Note that  $U\sse\bigcap_{n<\om}p_1[A_n]$.
Enumerate $U$ as $(n_k)_{k<\om}$.
Then for each $k<\om$ and each $m<\om$,
$(A'_m)_{n_k}=(A_m)_{n_k}$ since $n_k\in U\sse p_1[B]$.
So for each $k<\om$,
we have that $(A_m)_{n_k}\ra(B)_{n_k}$ as $m\ra\infty$.
Take a decreasing sequence $M_0\contains M_1\contains\dots\contains M_k\contains\dots$ 
of infinite subsets of $\om$ such that for each $k$,
$\bigcap_{m\in M_k}(A_m)_{n_k}\in\mathcal{U}_{n_k}$.
We may assume that $m_k=\min M_k$ is a strictly increasing sequence.

Let $C=\bigcap_{l<\om}A_{m_l}$.
We claim that  
$C\in\mathcal{V}$.
Note that  $U=\{n_k:k<\om\}\sse p_1[A_{m_l}]$ for all $l$,
so $U\sse p_1[C]$.
Thus, $p_1[C]\in\mathcal{U}$.
For each $k$, $\bigcap_{l\ge k}(A_{m_l})_{n_k}\contains \bigcap_{m\in M_k}(A_m)_{n_k}$ which is in $\mathcal{U}_{n_k}$.
Hence, intersecting $\bigcap_{l\ge k}(A_{m_l})_{n_k}$ with finitely more members $(A_{m_l})_{n_k}$, $l<k$, of $\mathcal{U}_{n_k}$ still yields a member of $\mathcal{U}_{n_k}$.
Thus, $(C)_{n_k}=\bigcap_{l<\om}(A_{m_l})_{n_k}$, which is in $\mathcal{U}_{n_k}$.
Therefore, $C\in\mathcal{V}$.
\end{proof}

\begin{rem}
For any ultrafilter $\mathcal{U}$, $\mathcal{U}\cdot\mathcal{U}$ is not a p-point.
Thus, there are basically generated ultrafilters which are not p-points.
\end{rem}

Recall Fact \ref{fact.c<omtop}
which says that for every ultrafilter $\mathcal{U}$, $(\mathcal{U},\contains)\le_T([\mathfrak{c}]^{<\om},\sse)$.
We say that an ultrafilter $\mathcal{U}$ {\em has top Tukey type} if 
$(\mathcal{U},\contains)\equiv_T([\mathfrak{c}]^{<\om},\sse)$.
The following theorem of Isbell shows that, in ZFC, there is always an ultrafilter which has top Tukey type.

\begin{thm}[Isbell \cite{Isbell65}]\label{thm.3}
There is an ultrafilter $\mathcal{U}_{\mathrm{top}}$ on $\om$ realizing the maximal cofinal type among all directed sets of cardinality continuum, i.e.\ $\mathcal{U}_{\mathrm{top}}\equiv_T[\mathfrak{c}]^{<\om}$.
\end{thm}

We remark here that the same construction in Isbell's  proof was done independently by  Juh\'{a}sz in \cite{Juhasz66} (stated in \cite{Juhasz67}) in connection with strengthening a theorem of Posp{\'{i}}{\v{s}}il \cite{Pospisil39}, though without the Tukey terminology.

There are in fact $2^{\mathfrak{c}}$ many ultrafilters on $\om$ having Tukey type exactly $([\mathfrak{c}]^{<\om},\sse)$, since any collection of independent sets can be used in  a canonical way to construct an ultrafilter with top Tukey type. 
Thus, already we see that for the case of the top Tukey type, the Rudin-Keisler equivalence relation is strictly finer than the Tukey 
equivalence relation, since every Rudin-Keisler equivalence class has cardinality $\mathfrak{c}$.

Note also that $\mathcal{U}_{\mathrm{top}}$ is not basically representable, or in other words,

\begin{thm}\label{thm.4}
If $\mathcal{U}$ is a basically generated ultrafilter on $\om$, then $\mathcal{U}<_T[\mathfrak{c}]^{<\om}$.
\end{thm}

\begin{proof}
Let $\mathcal{U}$ be basically generated.
Then there is  a filter basis $\mathcal{B}\sse\mathcal{U}$ 
with the property that each sequence 
$(A_n)_{n<\om}\sse\mathcal{B}$ converging to an element of $\mathcal{B}$
has a subsequence $(A_{n_k})_{k<\om}$ such that $\bigcap_{k<\om}A_{n_k}\in\mathcal{U}$.

Let $\mathcal{X}$ be any subset of $\mathcal{U}$ of cardinality $\mathfrak{c}$.
For each $X\in\mathcal{X}$, choose one $B_X\in\mathcal{B}$ such that $B_X\sse X$.
If 
there is an infinite $\mathcal{Y}\sse\mathcal{X}$ and a $B\in\mathcal{B}$ such that 
all $X\in\mathcal{Y}$ have  $B_X=B$, then this $B\sse\bigcap\mathcal{Y}$.
Otherwise, $\{B_X:X\in\mathcal{X}\}$ is uncountable, so there is a sequence $(A_n)_{n<\om}\sse \{B_X:X\in\mathcal{X}\}$ which converges to some $B\in\{B_X:X\in\mathcal{X}\}$, and such that all $A_n$ are distinct.
Since $\mathcal{B}$ witnesses that $\mathcal{U}$ is basically generated,
there is a subsequence $(A_{n_k})_{k<\om}$ such that $\bigcap_{k<\om}A_{n_k}\in\mathcal{U}$.
Taking $\mathcal{Y}$ to be the collection of $X\in\mathcal{X}$ such that $B_X=A_{n_k}$ for some $k$,
we have that $\mathcal{Y}$ is infinite and 
$\bigcap\mathcal{Y}\contains \bigcap_{k<\om}A_{n_k}$ which is in $\mathcal{U}$.
By Fact \ref{fact.Tukeytopchar},
$(\mathcal{B},\contains)<_T([\mathfrak{c}]^{<\om},\sse)$.
\end{proof}

\begin{cor}\label{cor.pptnottop}
Every p-point has Tukey type strictly below the top Tukey type.
\end{cor}

\begin{proof}
Since every basic ultrafilter is basically generated, it follows from Theorems \ref{thm.stevosummer}
and \ref{thm.4}
that every p-point has Tukey type strictly below $[\mathfrak{c}]^{<\om}$.
\end{proof}

The next theorem gives a canonical form for cofinal maps from p-points to any other ultrafilter.
This theorem or similar ideas will be used in the majority of proofs in the rest of this paper.

Recall that any subset of $\mathcal{P}(\om)$ is a   topological space, with the subspace topology inherited from the Cantor space.
Thus, given any
$\mathcal{X},\mathcal{Y}\sse\mathcal{P}(\om)$, 
a function $f:\mathcal{X}\ra\mathcal{Y}$ is continuous if it is continuous with respect to the subspace topologies on $\mathcal{X}$ and $\mathcal{Y}$.
Equivalently,
 a function $f:\mathcal{X}\ra\mathcal{Y}$ is continuous if for each sequence $(X_n)_{n<\om}\sse\mathcal{X}$ which converges to some $X\in\mathcal{X}$,
the sequence $(f(X_n))_{n<\om}$ converges to $f(X)$.

If  $X\in\mathcal{U}$, then we use $\mathcal{U}\re X$ to denote $\{Y\in\mathcal{U}:Y\sse X\}$.
Note that $\mathcal{U}\re X$ is a filter base for $\mathcal{U}$, and hence $(\mathcal{U},\contains)\equiv_T(\mathcal{U}\re X,\contains)$.

\begin{thm}\label{thm.5}
Suppose $\mathcal{U}$ is a p-point on $\om$ and that $\mathcal{V}$ is an arbitrary ultrafilter on $\om$ such that $\mathcal{U}\ge_T \mathcal{V}$.
Then there is a continuous monotone map
$f^*:\mathcal{P}(\om)\ra\mathcal{P}(\om)$ 
 whose restriction to $\mathcal{U}$ is continuous
and has cofinal range in $\mathcal{V}$.
Hence,  there is a continuous monotone cofinal map from $\mathcal{U}$ into $\mathcal{V}$ witnessing that $\mathcal{U}\ge_T\mathcal{V}$.
\end{thm}

\begin{proof}
Let $\mathcal{U}$ be a p-point, $\mathcal{V}$ be an ultrafilter, and suppose that $\mathcal{U}\ge_T\mathcal{V}$.
By Fact \ref{fact.monotone}, there is a monotone cofinal map $f:\mathcal{U}\ra\mathcal{V}$.
We claim that there is an $\tilde{X}\in\mathcal{U}$ such that $f:\mathcal{U}\re \tilde{X}\ra\mathcal{V}$ is continuous.

Construct a decreasing sequence $X_0\contains X_1\contains\dots$ of elements of $\mathcal{U}$ as follows.
Let $X_0=\om$.
For $1\le n<\om$, given $X_{n-1}$, we take an $X_n\in\mathcal{U}$ with the following properties:
\begin{enumerate}
\item
$X_n\sse X_{n-1}$;
\item
 $X_n\cap n=\emptyset$;
\item  for each $s\sse n$, for each $k\le n$, if there is a $Y'\in\mathcal{U}$ such that $s=Y'\cap(n+1)$ and $k\not\in f(Y')$, then $k\not\in f(s\cup X_n)$.
\end{enumerate}

That there is such a sequence of $X_n$ follows from $f$ being monotone, as we shall see now.
Suppose we already have $X_{n-1}$.
Fix a $W_0\in\mathcal{U}$ such that $W_0\sse X_{n-1}$ and $W_0\cap n=\emptyset$.
List out all subsets of $n$ as $s_1,\dots,s_{2^n}$.
Suppose there is a $Y'\in\mathcal{U}$ such that $s_1=Y'\cap (n+1)$ and $k\not\in f(Y')$.
Then take some such $Y'_1$ and let $W_1= W_0\cap Y_1'$. 
If there is no such $Y'\in\mathcal{U}$,
then let $W_1=W_0$.
For $1\le l<2^n$, given $W_0\contains\dots\contains W_l$,
if there is a 
$Y'\in\mathcal{U}$ such that $s_{l+1}=Y'\cap (n+1)$ and $k\not\in f(Y')$,
then take some such $Y'_{l+1}$ and let $W_{l+1}= W_l\cap Y'_{l+1}$. 
If there is no such $Y'\in\mathcal{U}$,
then let $W_{l+1}=W_l$.
After the $2^n$ many steps of this process, we
let $X_n=W_{2^n}$.

Note the following for each $1\le l\le 2^n$. 
If there is a 
$Y'\in\mathcal{U}$ such that $s_{l}=Y'\cap (n+1)$ and $k\not\in f(Y')$, then
$W_l$ was taken to be $W_{l-1}\cap Y'_l$.
So
for any $U\in\mathcal{U}\re X_n$, we have $s_{l}\cup U\sse Y'_l$.
Since $f$ is monotone, we have $f(s_{l}\cup U)\sse f(Y'_l)$.
Thus, $k\not\in f(s_{l}\cup U)$, since $k\not\in f(Y'_l)$.

We check that $X_n$ has the desired properties.
By construction, (1) holds.
Since $X_n\sse W_0$, we have that $X_n\cap n=\emptyset$, so (2) holds.
Let $s$ be any subset of $n$.  Then there is some $1\le l\le 2^n$ such that $s=s_l$.
Suppose there is a $Y'\in\mathcal{U}$ such that $s_l=Y'\cap(n+1)$ and $k\not\in f(Y')$.
Then by the preceding paragraph, $k$ is not in $f(s\cup X_n)$.

Since $\mathcal{U}$ is a p-point, fix some 
 $Y\in\mathcal{U}$ be such that for each $n<\om$, $Y\sse^* X_n$.
Let $0=n_0<n_1<\dots$ be such that for each $i<\om$, for each $n\le n_i$, $Y\setminus n_{i+1}\sse X_n$.
Let $Z=\bigcup_{i=0}^{\infty}[n_{2i+1},n_{2i+2})$.
Without loss of generality, assume that $Z\not\in\mathcal{U}$.
(If $Z$ is in $\mathcal{U}$, then  let $\tilde{X}$ be $Y\cap Z$.  The proof for this case goes through exactly as the one we give below, with the minor modification of readjusting the indexes by 1 at the outset.)
Let $\tilde{X}=Y\setminus Z$.
We show that  $f\re(\mathcal{U}\re \tilde{X})$ is continuous.
Precisely, we shall show that 
there is a  non-decreasing sequence $(m_k)_{k<\om}$ such that 
for each $W\in\mathcal{U}\re \tilde{X}$,
the initial segment $f(W)\cap (k+1)$ of $f(W)$ is determined by $W\cap m_k$.

Given $k<\om$,  let $i_k$ denote the least $i$ for which $n_{2i_k +1}\ge k$.
Let $W\in\mathcal{U}\re \tilde{X}$ be given and let $s= W\cap n_{2i_k +1}$.
Recalling that $\tilde{X}\cap[n_{2i_k+1},n_{2i_k+2})=\emptyset$, we have that
$W\setminus n_{2i_k +1}\sse\tilde{X}\setminus n_{2i_k +1}
=\tilde{X}\setminus n_{2i_k+2}
\sse Y\setminus n_{2i_k +2}
\sse X_{n_{2i_k +1}}$.
Therefore,
$k\in f(W)$ iff 
$k\in f(s\cup X_{n_{2i_k +1}})$ iff
$k\in f(s\cup (\tilde{X}\setminus n_{2i_k+1}))$.
Letting $m_k=n_{2i_k+1}$, we see that $f\re(\mathcal{U}\re\tilde{X})$ is continuous,
since the question of whether or not $k\in f(W)$  is determined by the finite initial segment $W\cap m_k$ along with $\tilde{X}\setminus m_k$.

Next, we extend $f$ on $\mathcal{U}\re \tilde{X}$ to all of $\mathcal{U}$ by defining $f'(X)=f(X\cap \tilde{X})$, for $X\in\mathcal{U}$. 
Then $f':\mathcal{U}\ra\mathcal{V}$ is again monotone.
Moreover, for each $X\in\mathcal{U}$ and $k<\om$,
$k\in f'(X)$ iff $k\in f(X\cap\tilde{X})$
iff $k\in f(s\cup (\tilde{X}\setminus m_k))$, where $s= X\cap\tilde{X}\cap m_k$.
So whether or not $k$ is in $f'(X)$ is determined by the initial segment $X\cap\tilde{X}\cap m_k$ of $X\cap\tilde{X}$;
hence $f'$ is continuous.

Finally, we extend $f'$ to a monotone continuous map $f^*$ defined on all of $\mathcal{P}(\om)$.
For an arbitrary $Z\sse\omega$ set
\begin{equation}\label{eq.+}
f^*(Z)=\bigcap\{f'(Z): X\contains Z\text{ and } X\text{ is  cofinite}\}.
\end{equation}
Note that since $f'$ is monotone, 
$f^*(Z)$ is exactly $\bigcap\{f'((Z\cap n)\cup[n,\om)):n<\om\}$, since every cofinite $X$ containing $Z$ contains $(Z\cap n)\cup[n,\om)$ for some $n$.
From the definition  of $f^*$ and the fact that $f'$ is monotone, it follows that $f^*$ is monotone.

First, we show that $f^*\re\mathcal{U}=f'$.
Let $Z\in\mathcal{U}$ be given.
Let $Z_n=(Z\cap n)\cup[n,\om)$, for each $n<\om$.
Then $f^*(Z)=\bigcap\{f'(Z_n):n<\om\}$.
Since  $Z_n\ra Z$ and $f'$ is continuous on $\mathcal{U}$,
it follows that $f'(Z_n)\ra f'(Z)$.
This, along with the fact that each $f'(Z_n)\contains f'(Z)$ imply that $\bigcap_{n<\om}f'(Z_n)$ equals $f'(Z)$.
Hence,
 $f^*(Z)=\bigcap_{n<\om} f'(Z_n)=f'(Z)$.
Thus,  $f^*\re\mathcal{U}=f'$.

To see that $f^*$ is continuous, 
we show that for each $k<\om$ and $Z\sse\om$, whether or not $k$ is in $f^*(Z)$ is determined by the initial segment $Z\cap\tilde{X}\cap m_k$ of $Z\cap\tilde{X}$, along with $\tilde{X}\setminus m_k$.
Let $Z\sse\om$ and $k<\om$, and let $Z_n=(Z\cap n)\cup[n,\om)$ for each $n<\om$.
Then 
$k\in f^*(Z)$ iff
for each $n<\om$, $k\in f'(Z_n)$
iff for each $n\ge m_k$, $k\in f'(Z_n)$
iff for each $n\ge m_k$, $k\in f'(Z_n\cap \tilde{X})$
iff $k\in f(s\cup (\tilde{X}\setminus m_k))$, where $s=Z\cap\tilde{X}\cap m_k$.
\end{proof}

\begin{rem}
Note that Theorem \ref{thm.5}
 gives the canonical form of cofinal maps that is
likely going to be the main object of study in this area from now on:
Every Tukey reduction $\mathcal{U}\ge_T\mathcal{V}$ for $\mathcal{U}$ a p-point 
is witnessed by some monotone continuous $f^*:\mathcal{P}(\om)\ra\mathcal{P}(\om)$ such that 
$f^*\re\mathcal{U}$ is a cofinal map from $\mathcal{U}$ into $\mathcal{V}$.
Moreover,  for
any monotone cofinal map $f:\mathcal{U}\ra\mathcal{V}$, (where $\mathcal{U}$ is a p-point),
there is a a cofinal subset of the form $\mathcal{U}\re\tilde{X}$ for some $\tilde{X}\in\mathcal{U}$ such that $f\re(\mathcal{U}\re\tilde{X})$ is continuous.
Note that the restriction of $f$ to any cofinal subset of $\mathcal{U}\re\tilde{X}$ retains continuity, justifying the use of the word {\em canonical}.
\end{rem}

\begin{rem}
Whereas the top Tukey type has cardinality $2^{\mathfrak{c}}$, the previous theorem implies that the Tukey type of any p-point has cardinality $\mathfrak{c}$.
\end{rem}

\begin{cor}\label{cor.5.1.5}
Every $\le_T$-chain of p-points on $\om$ has cardinality $\le\mathfrak{c}^+$.
\end{cor}

\begin{proof}
Theorem \ref{thm.5} 
 shows that every Tukey chain $\mathcal{F}\sse\{$p-points$\}$ is $\mathfrak{c}^+$-like, that is, $|\{\mathcal{V}\in\mathcal{F}:\mathcal{V}\le_T\mathcal{U}\}|\le\mathfrak{c}$ for all $\mathcal{U}\in\mathcal{F}$.
\end{proof}

Recall the Free Set Lemma of Hajnal.

\begin{lem}[Free Set Lemma of Hajnal \cite{Juhasz75}]\label{lem.freeset}
If $|X|=\kappa$ and $\lambda<\kappa$ and 
$F:X\ra\mathcal{P}(X)$ satisfies $x\not\in F(x)$ and $|F(x)|<\lambda$, for all $x\in X$, 
then there is a $Y\sse X$ with $x\not\in F(y)$ and $y\not\in F(x)$ for all $x,y\in Y$ and $|Y|=\kappa$.
\end{lem}

\begin{cor}\label{cor.5.5}
Every family $\mathcal{X}$ of p-points on $\om$ of cardinality $>\mathfrak{c}^+$ contains a subfamily $\mathcal{Z}\sse\mathcal{X}$ of equal size such that $\mathcal{U}\not\le_T\mathcal{V}$ whenever $\mathcal{U}\ne\mathcal{V}$ are in $\mathcal{Z}$.
\end{cor}

\begin{proof}
Let $\mathcal{X}$ be a family of p-points   such that  $\kappa:=|\mathcal{X}|>\mathfrak{c}^+$.
Define $F:\mathcal{X}\ra\mathcal{P}(\mathcal{X})$ by
$F(\mathcal{U})=\{\mathcal{V}\in\mathcal{X}:\mathcal{V}<_T\mathcal{U}\}$.
By Theorem \ref{thm.5},
for each $\mathcal{U}\in \mathcal{X}$,
$|F(\mathcal{U})|< \mathfrak{c}^+$.
So, by the Free Set Lemma \ref{lem.freeset}, 
there is a family $\mathcal{Y}\sse\mathcal{X}$ such that $|\mathcal{Y}|=\kappa$ and for each $\mathcal{U},\mathcal{V}\in\mathcal{Y}$,
$\mathcal{U}\not\in F(\mathcal{V})$ and $\mathcal{V}\not\in F(\mathcal{U})$; that is, $\mathcal{U}\not<_T\mathcal{V}$ and $\mathcal{V}\not<_T\mathcal{U}$.
By Theorem \ref{thm.5}, there are at most $\mathfrak{c}$ many ultrafilters Tukey equivalent to any given p-point.
Thus, there is a subfamily $\mathcal{Z}\sse\mathcal{Y}$ also of cardinality $\al$ such that every two p-points in $\mathcal{Z}$ are Tukey incomparable.
\end{proof}

\begin{rem}
A similar trick was used by Rudin and Shelah in \cite{Shelah/Rudin78} in part of their proof that there are always $2^{\mathfrak{c}}$ many Rudin-Keisler incomparable ultrafilters.
\end{rem}


Next, we use Theorem \ref{thm.5} to  see that some strength of selective ultrafilters is preserved  downward in the Tukey ordering.

\begin{thm}\label{thm.14}
Suppose $\mathcal{U}$ is selective and $\mathcal{U}\ge_T\mathcal{V}$.
Then $\mathcal{V}$ is basically generated.
\end{thm}

\begin{proof}
By Theorem \ref{thm.5}, there is a continuous monotone map $f:\mathcal{P}(\om)\ra\mathcal{P}(\om)$
such that $f''\mathcal{U}\sse\mathcal{V}$ and $f''\mathcal{U}$ generates $\mathcal{V}$.
By  the selective version of the Pr\"{o}mel-Voight canonical form of the Galvin-Prikry Theorem,
there is an
 $M\in\mathcal{U}$,
 a   Lipschitz map $\varphi:[\om]^{\om}\ra\mathcal{P}(\om)$ such that $\varphi(X)\sse X$ for each $X\in[\om]^{\om}$,
 and a 1-1 homeomorphism $\psi:$ range$(\varphi)\ra\mathcal{P}(\om)$ such that $f=\psi\circ\varphi$.

Let $\mathcal{B}=f''\mathcal{U}\re M$.
Note that $\mathcal{B}$ is a cofinal subset of $\mathcal{V}$.
We claim that 
every converging sequence $X_n\ra X$ of elements of $\mathcal{B}$ has a subsequence $X_{n_k}$ such that $\bigcap_{k<\om}X_{n_k}\in\mathcal{V}$.
Let $X_n$, $n<\om$, and $X$ be elements of $\mathcal{B}$ such that $X_n\ra X$.
Let $Y=\psi^{-1}(X)$ and $Y_n=\psi^{-1}(X_n)$.
Then $Y_n\ra Y$, since $\psi$ is a 1-1 homeomorphism.
Let $K=\{A\in\mathcal{U}:\varphi(A)=Y\}$ and 
$K_n=\{A\in\mathcal{U}:\varphi(A)=Y_n\}$ $(n\in\om)$.
Then $K$ and $K_n$ are compact subsets of $\mathcal{U}$ such that $K_n\ra K$. 
So in particular for an arbitrary choice $A_n\in K_n$, $(n\in\om)$
we can find a subsequence $A_{n_k}$ converging to a member $B$ in $K$.
Note that $A_{n_k}$ is a  sequence in $\mathcal{U}$ converging to the member $B$, which is in $\mathcal{U}$.
Since $\mathcal{U}$ is basic there is a further subsequence $A_{n_{k_i}}$ such that 
$$
A=\bigcap_{i<\om}A_{n_{k_i}}\in\mathcal{U}.
$$
It follows that $X_{n_{k_i}}=f(A_{n_{k_i}})\contains f(A)$ for all $i<\om$ and so in particular, 
$f(A) \in\mathcal{V}$ and $f(A)\sse\bigcap_{i<\om}X_{n_{k_i}}$.
Thus, $\mathcal{B}$ witnesses that $\mathcal{V}$ is basically generated.
\end{proof}

It will be shown in Section \ref{sec.omom}
that  for each selective ultrafilter $\mathcal{U}$, 
$\mathcal{U}\cdot\mathcal{U}\equiv_T\mathcal{U}$;
hence $\mathcal{U}\equiv_T\mathcal{V}$ does not imply that $\mathcal{V}$ is selective.

\begin{question}\label{q.p-pointdown}
If $\mathcal{U}$ is a p-point and $\mathcal{U}\ge_T\mathcal{V}$, does it follow that $\mathcal{V}$ is basically generated?
\end{question}

\begin{question}
From Theorem \ref{thm.2},
we know that every iteration of Fubini products of p-points is basically generated.
Is there 
an ultrafilter which is basically generated but is not a Fubini limit  of p-points?
\end{question}

\begin{question}
Can Theorem \ref{thm.5} be improved to show that if $\mathcal{U}$ is basically generated and $\mathcal{U}\ge_T\mathcal{V}$,
then there is a continuous (or definable) monotone cofinal map $f:\mathcal{U}\ra\mathcal{V}$ witnessing this?
\end{question}

More generally,
\begin{question}
If $\mathcal{V}\le_T\mathcal{U}<_T[\mathfrak{c}]^{<\om}$,
then is there a continuous (or definable) monotone cofinal map $f:\mathcal{U}\ra\mathcal{V}$ witnessing this?
\end{question}

One might first try to show that
 the existence of a continuous cofinal map propagates Tukey downwards, or in other words,

\begin{question}
Suppose  that  $\mathcal{U}$ is such that whenever $\mathcal{U}\ge_T\mathcal{V}$ then there is a continuous monotone cofinal map from $\mathcal{U}$ to $\mathcal{V}$.
If $\mathcal{U}\ge_T\mathcal{W}$, then does it follow that for each $\mathcal{V}\le_T\mathcal{W}$ there is a continuous monotone cofinal map from $\mathcal{W}$ into $\mathcal{V}$?
\end{question}




\section{Comparing Tukey types of ultrafilters with $(\om^{\om},\le)$}\label{sec.omom}

In this section we investigate which ultrafilters are above $(\om^{\om},\le)$, where  $h\le g$ iff for each $n<\om$, $h(n)\le g(n)$.

\begin{fact}\label{fact.rapid>omom}
If $\mathcal{U}$ is a rapid ultrafilter, then $\mathcal{U}\ge_T \om^{\om}$.
\end{fact}

\begin{proof}
Define $f:\mathcal{U}\ra\om^{\om}$ by letting $f(X)$ be the function which enumerates all but the least element of $X$ in strictly increasing order. It is not hard to check that  $f$ is a cofinal map.
\end{proof}

Hence each selective ultrafilter and each q-point is Tukey above  $\om^{\om}$.

\begin{fact}\label{fact.Uoo}
For each ultrafilter $\mathcal{U}$,
$\mathcal{U}\cdot\mathcal{U}\ge_T \om^{\om}$.
\end{fact}

\begin{proof} 
Define $f:\mathcal{U}\cdot\mathcal{U}\ra\om^{\om}$ by letting 
$f(A)$ be the function $g_A:\om\ra\om$ defined by 
$g_A(n)=\min(A)_{n_k}$,
where $(n_k)_{k<\om}$ enumerates those $n$ for which $(A)_n\in\mathcal{U}$.
We shall show that $f$ is a cofinal map.

Let $\mathcal{X}$ consist of those $A\in\mathcal{U}\cdot\mathcal{U}$ with the properties that (a) whenever $(A)_n\ne\emptyset$, then $(A)_n\in\mathcal{U}$, and  (b)
whenever $m<n$ and $(A)_m,(A)_n\in\mathcal{U}$, then
$\min(A)_m\le\min(A)_n$.
Note that $\mathcal{X}$ is a base for $\mathcal{U}\cdot\mathcal{U}$, so it suffices to show that $f\re\mathcal{X}$ is a cofinal map from $\mathcal{X}$ into $\om^{\om}$.
We show that $f\re\mathcal{X}$ is monotone and has range which is cofinal in $\om^{\om}$,
hence by Fact \ref{fact.monotonemap},
$f\re\mathcal{X}$ is a cofinal map from $\mathcal{X}$ into $\om^{\om}$.

Let $A,B\in\mathcal{X}$ be given such that $A\contains B$.
Then the sequence $(i_k)_{k<\om}$ enumerating those $n$ for which $(B)_n\in\mathcal{U}$ is a subsequence of the sequence $(n_k)_{k<\om}$ enumerating those $n$ for which $(A)_n\in\mathcal{U}$.
Hence, for each $k$, $n_k\le i_k$. Since $A,B$ are in $\mathcal{X}$ and $A\contains B$,
we have  that $\min(A)_{n_k}\le \min(A)_{i_k}\le\min(B)_{i_k}$;
hence $g_A(k)\le g_B(k)$ for all $k<\om$.
Therefore, $f\re\mathcal{X}$ is monotone.

Next, let $h:\om\ra\om$ be given.
Define $A$ to be the collection of pairs $(n,l)$ such that $l>\max\{h(i):i\le n\}$.  Then $A\in\mathcal{X}$ and $g_A(n)\ge h(n)$ for all $n<\om$.
Thus, $f\re\mathcal{X}$ has cofinal range in $\om^{\om}$.
\end{proof}

\begin{thm}\label{thm.10}
For any ultrafilters $\mathcal{U},\mathcal{U}_n$ ($n<\om$), $\lim_{n\ra\mathcal{U}}\mathcal{U}_n
\le_T\mathcal{U}\times\Pi_{n<\om}\mathcal{U}_n$,
where $\mathcal{U}\times\Pi_{n<\om}\mathcal{U}_n$ is given its natural product ordering.
In particular,  $\mathcal{U}\cdot\mathcal{U}\le_T\Pi_{n<\om}\mathcal{U}$.
\end{thm}

\begin{proof}
Let $\mathcal{V}$ denote $\lim_{n\ra\mathcal{U}}\mathcal{U}_n$.
Let  $\mathcal{B}=\{A\in \mathcal{V}:$ for each $n<\om$, either $(A)_n=\emptyset$ or $(A)_n\in \mathcal{U}_n\}$. 
Note that $\mathcal{B}$ is a basis for 
 $\mathcal{V}$; hence it suffices to construct a Tukey map $g:\mathcal{B}\ra\mathcal{U}\times\Pi_{n<\om}\mathcal{U}_n$.
Given $A\in\mathcal{B}$ let $g(A)=(p_1[A],(q_n(A))_{n<\om})$,
where $q_n(A)=(A)_n$ if $n\in p_1[A]$ and $q_n(A)=\om$ otherwise.

To verify $g$ is a Tukey map 
let $\mathcal{Y}$ be a bounded subset of $\mathcal{V}$.
Then there is some $(C,(D_n:n<\om))\in\mathcal{U}\times\Pi_{n<\om}\mathcal{U}_n$ 
which bounds $\mathcal{Y}$.
Let
$\mathcal{X}=\{A\in\mathcal{B}: p_1[A]\contains C$ and $\forall n<\om,\ q_n(A)\contains D_n\}$.
Note that $\mathcal{X}$ contains the $g$-preimage of $\mathcal{Y}$.
Let $B=\bigcap\mathcal{X}$.
Then $p_1[B]\contains C$ and for each $n\in C$,
$(B)_n\contains D_n$, so $B\in\mathcal{V}$.
Moreover, by its definition, $B$ bounds $\mathcal{X}$.
Hence $B$ also bounds the $g$-preimage of $\mathcal{Y}$.
\end{proof}

\begin{thm}\label{thm.8}
If $\mathcal{U}$ is a p-point, then $\Pi_{n<\om}\mathcal{U}\equiv_T\mathcal{U}\times\om^{\om}$
and therefore $\Pi_{n<\om}\mathcal{U}\equiv_T\mathcal{U}\cdot\mathcal{U}\cdot\mathcal{U}$.
\end{thm}

\begin{proof}
First, we show that $\Pi_{n<\om}\mathcal{U}\le_T\mathcal{U}\times\om^{\om}$.
Given a sequence $(A_n)_{n<\om}\in\Pi_{n<\om}\mathcal{U}$,
choose a $B\in\mathcal{U}$ and an $h:\om\ra\om$ such that 
$B\setminus h(n)\sse A_n$ for each $n$.
(Since $\mathcal{U}$  is a p-point, there is a $B\in\mathcal{U}$ such that $B\sse^* A_n$ for each $n$.  Let $h(n)$ be the least $m$ such that $B\setminus m\sse A_n$.)
Set $g((A_n)_{n<\om})=(B,h)$.

$g$ is a Tukey map.
To see this, let 
$\mathcal{Y}$ be a bounded subset of $\mathcal{U}\times\om^{\om}$.
Then there is some
$(B_*,h_*)\in\mathcal{U}\times\om^{\om}$ which bounds $\mathcal{Y}$.
Let $\mathcal{X}=\{(A_n)_{n<\om}:g((A_n)_{n<\om})\le (B_*,h_*)\}$.
Note that $\mathcal{X}$ set contains the $g$-preimage of $\mathcal{Y}$.
We claim that
$\mathcal{X}$ is bounded by $(B_*\setminus h_*(n))_{n<\om}$.
For given any $(A_n)_{n<\om}\in\mathcal{X}$,
letting $(B,h)$ denote $g((A_n)_{n<\om})$,
we have that $(B,h)\le (B_*,h_*)$, 
which means that $B\contains B_*$ and $h(n)\le h_*(n)$ for all $n$.
So for each $n$, $B_*\setminus h_*(n)\sse B\setminus h(n)\sse A_n$.
Thus, $(B_*\setminus h_*(n))_{n<\om}$ is a bound for $\mathcal{X}$.

On the other hand, $\om^{\om}\le_T\mathcal{U}\cdot\mathcal{U}\le_T\Pi_{n<\om}\mathcal{U}$, by Fact \ref{fact.Uoo} and Theorem \ref{thm.10}.
So $\mathcal{U}\times\om^{\om}\le_T\mathcal{U}\times\Pi_{n<\om}\mathcal{U}=\Pi_{n<\om}\mathcal{U}$.

Finally,
$\mathcal{U}\cdot\mathcal{U}\le_T\mathcal{U}\cdot\mathcal{U}\cdot\mathcal{U}$ and Fact \ref{fact.Uoo} imply that 
$\mathcal{U}\times\om^{\om}
\le_T\mathcal{U}\times(\mathcal{U}\cdot\mathcal{U})
\equiv_T\mathcal{U}\cdot\mathcal{U}
\le_T\mathcal{U}\cdot\mathcal{U}\cdot\mathcal{U}$.
On the other hand, applying Theorem \ref{thm.10} twice,
we have 
$\mathcal{U}\cdot\mathcal{U}\cdot\mathcal{U}\equiv_T\lim_{n\ra\mathcal{U}\cdot\mathcal{U}}\mathcal{U}
\le_T(\mathcal{U}\cdot\mathcal{U})\times\Pi_{n<\om}\mathcal{U}
\le_T\Pi_{n<\om}\mathcal{U}\times\Pi_{n<\om}\mathcal{U}=\Pi_{n<\om}\mathcal{U}$.
Thus, $\mathcal{U}\cdot\mathcal{U}\cdot\mathcal{U}\equiv_T\Pi_{n<\om}\mathcal{U}$.
\end{proof}

\begin{cor}\label{cor.dot}
If $\mathcal{V}$ is a p-point, $\mathcal{V}\ge_T\om^{\om}$, and $\mathcal{U}$ is any ultrafilter,
then $\mathcal{U}\cdot\mathcal{V}\equiv_T\mathcal{U}\times\mathcal{V}$.
\end{cor}

\begin{proof}
By Theorem \ref{thm.10},
$\mathcal{U}\cdot\mathcal{V}
\le_T\mathcal{U}\times\Pi_{n<\om}\mathcal{V}$.
Since $\mathcal{V}$ is a p-point, $\Pi_{n<\om}\mathcal{V}
\equiv_T\mathcal{V}\times\om^{\om}$,
by Theorem \ref{thm.8}.
$\mathcal{V}\ge_T\om^{\om}$ implies that $\mathcal{V}\times\om^{\om}\equiv_T\mathcal{V}$.
Therefore,
$\mathcal{U}\cdot\mathcal{V}
\le_T\mathcal{U}\times\Pi_{n<\om}\mathcal{V}
\equiv_T\mathcal{U}\times\mathcal{V}
\le_T\mathcal{U}\cdot\mathcal{V}$.
\end{proof}

\begin{thm}\label{thm.7}
The following are equivalent for a p-point  $\mathcal{U}$.
\begin{enumerate}
\item
$\mathcal{U}\ge_T\om^{\om}$;
\item
$\mathcal{U}\equiv_T\Pi_{n<\om}\mathcal{U}$;
\item
$\mathcal{U}\equiv_T\mathcal{U}\cdot\mathcal{U}$.
\end{enumerate}
\end{thm}

\begin{proof}
Suppose $\mathcal{U}\ge_T\om^{\om}$.
By Theorem \ref{thm.8}, $\Pi_{n<\om}\mathcal{U}\equiv_T\mathcal{U}\times\om^{\om}\equiv_T\mathcal{U}\le_T\Pi_{n<\om}\mathcal{U}$.
Suppose $\mathcal{U}\equiv_T\Pi_{n<\om}\mathcal{U}$.
Since always $\mathcal{U}\le_T\mathcal{U}\cdot\mathcal{U}$, and $\mathcal{U}\cdot\mathcal{U}\le_T\Pi_{n<\om}\mathcal{U}$ by Theorem \ref{thm.10},
we have that $\mathcal{U}\equiv_T\mathcal{U}\cdot\mathcal{U}$.
If $\mathcal{U}\equiv_T\mathcal{U}\cdot\mathcal{U}$, then since $\mathcal{U}\cdot\mathcal{U}\ge_T\om^{\om}$, 
we have that  $\mathcal{U}\ge_T\om^{\om}$.
\end{proof}

\begin{cor}
If $\mathcal{U}$ is a p-point of cofinality $<\mathfrak{d}$, then $\mathcal{U}\not\ge_T \om^{\om}$ and therefore $\mathcal{U}<_T\mathcal{U}\cdot\mathcal{U}$.
\end{cor}

\begin{rem}
Such an ultrafilter $\mathcal{U}$ exists in any extension of a model of CH by a countable support iteration of length $\om_2$ of superperfect-set forcing since by a result of Shelah such an iteration preserves p-points.
\end{rem}

\begin{cor}\label{cor.12}
If $\mathcal{U}$ is a rapid p-point then $\Pi_{n<\om}\mathcal{U}\equiv_T\mathcal{U}\cdot\mathcal{U}\equiv_T\mathcal{U}$.
\end{cor}

\begin{rem}
By Corollary \ref{cor.12},
 for each  selective ultrafilter $\mathcal{U}$, the Tukey type of $\mathcal{U}$ is strictly coarser than the Rudin-Keisler type of $\mathcal{U}$, even though they both have cardinality $\mathfrak{c}$.
That is, if  $\mathcal{U}$ is selective,
then
 $\mathcal{U}\cdot\mathcal{U}$ is not a p-point yet $\mathcal{U}\equiv_T\mathcal{U}\cdot\mathcal{U}$.
However, if $\mathcal{U}\equiv_{RK}\mathcal{V}$ then $\mathcal{V}$ is selective.
We remark here that Todorcevic has more recently shown that if $\mathcal{U}$ is a p-point, $\mathcal{V}$ is selective and $\mathcal{U}\ge_T\mathcal{V}$, 
then $\mathcal{U}\ge_{RK}\mathcal{V}$,
and hence, $\mathcal{V}\equiv_{RK}\mathcal{U}$. 
\footnote{By the time of printing, this result  has been  extended by Raghavan in \cite{Raghavan/Todorcevic11} to the following more general context: For any ultrafilters $\mathcal{U}\ge_T\mathcal{V}$, if  $\mathcal{V}$ is a q-point and there is a continuous cofinal map from $\mathcal{U}$ into $\mathcal{V}$,
then $\mathcal{U}\ge_{RB}\mathcal{V}$.}
Hence, although the Tukey type of a selective ultrafilter includes non-p-points, any two selective ultrafilters with the same Tukey type are isomorphic.
\end{rem}

\begin{thm}\label{thm.18}
Assuming $\mathfrak{p}=\mathfrak{c}$, there is a p-point $\mathcal{U}$ such that $\mathcal{U}\not\ge_T \om^{\om}$ and therefore $\mathcal{U}<_T\mathcal{U}\cdot\mathcal{U}<_T\mathcal{U}_{\text{top}}$.
\end{thm}

\begin{proof}
Let $\{f_{\al}:0<\al<\mathfrak{c}\}$ be an enumeration of all Souslin-measurable mappings from $\om^{\om}$ into $[\om]^{\om}$,
and  
let $\{X_{\al}:\al<\mathfrak{c}\}$ be an enumeration of $\mathcal{P}(\om)$.
We build an ultrafilter $\mathcal{U}$  to be generated by a $\contains^*$ chain $\lgl A_{\al}:\al<\mathfrak{c}\rgl$ 
of infinite subsets of $\om$, while diagonalizing over all Souslin-measurable mappings of the form $f_{\al}:\om^{\om}\ra[\om]^{\om}$ ($\al<\mathfrak{c}$).

Let $A_0=\om$.
Given $\al<\mathfrak{c}$ and $\{A_{\xi}:\xi<\al\}$,
using the fact that $\mathfrak{p}=\mathfrak{c}$,
there is an $A'_{\al}\in[\om]^{\om}$ such that $A'_{\al}\sse^* A_{\xi}$ for all $\xi<\al$.
Let $A''_{\al}=A'_{\al}\cap X_{\al}$ if this is infinite, otherwise, let $A''_{\al}=A'_{\al}\setminus X_{\al}$.
If there is an $x\in\om^{\om}$ such that $A''_{\al}\setminus f_{\al}(x)$ is infinite,
then let
$A_{\al}=A''_{\al}\setminus f_{\al}(x)$.
Otherwise, 
we let $A_{\al}=A''_{\al}$.

Let $\mathcal{U}$ be the p-point generated by the tower $\{A_{\al}:\al<\mathfrak{c}\}$.
We need to show that $\mathcal{U}\not\ge_T\om^{\om}$.
Suppose toward a contradiction 
that $\mathcal{U}\ge_T\om^{\om}$.
Then applying [Theorem 5.3 (i),
\cite{Solecki/Todorcevic04}],
there is a Souslin measurable map $f:\om^{\om}\ra\mathcal{U}$ such that $f$ is a Tukey map.
Since we listed all Souslin measurable maps from $\om^{\om}$ into $[\om]^{\om}$, there is an
 $\al<\mathfrak{c}$ such that $f_{\al}=f$.
Since the range of $f$ is contained in $\mathcal{U}$,
$A_{\al}$ is not $A'_{\al}\setminus f_{\al}(x)$ for any $x\in\om^{\om}$.
Hence, $A_{\al}=A''_{\al}$ 
and 
$A_{\al}\sse^*f_{\al}(x)$ for all $x\in\om^{\om}$.

Define $P_n$ to be $\{x\in\om^{\om}:A_{\al}\setminus n\sse f_{\al}(x)\}$.
There is  an $n_0\in\om$ such that
$P_{n_0}$,
 is not bounded in $\om^{\om}$ relative to the ordering of eventual domination.
(For if not, then for each $n$,
there is some $g_n\in\om^{\om}$ which eventually dominates every element of
$P_n$.
Let $g$ be a function which eventually dominates each $g_n$.
Then $g\ge^* x$ for each $x$ such that  for some $n$,
$A_{\al}\setminus n \sse f_{\al}(x)$.
But  $A_{\al}\sse^* f_{\al}(x)$ for all $x\in\om^{\om}$, 
and hence $g$ eventually dominates every member of $\om^{\om}$, contradiction.)
In particular, there is a $k\in\om$ and an infinite subset $\{x_i:i<\om\}\sse P_{n_0}$ such that $x_i(k)\ge i$ for all $i<\om$.
It follows that $\{x_i:i<\om\}$ is unbounded in $(\om^{\om},\le)$ but its image $\{f_{\al}(x_i):i<\om\}$ is bounded by $A_{\al}\setminus n$, which is in $\mathcal{U}$.
Thus, $f_{\al}$ is not a Tukey map from $\om^{\om}$ into $\mathcal{U}$.
\end{proof}

\begin{question}\label{question.21}
Is there an ultrafilter $\mathcal{U}$ on $\om$ such that $\mathcal{U}<_T\mathcal{U}\cdot\mathcal{U}<_T\mathcal{U} \cdot\mathcal{U}\cdot\mathcal{U}<_T\mathcal{U}_{top}$?\footnote{After  this article was first circulated, Blass and Milovich have independently shown that for all ultrafilters $\mathcal{U}$, $\mathcal{U}\cdot\mathcal{U}\equiv_T\mathcal{U}\cdot\mathcal{U}\cdot\mathcal{U}$.}
\end{question}


\begin{rem}\label{rem.22}
Using some assumptions like $\mathfrak{p}=\mathfrak{c}$, it seems possible to get Tukey chains of p-points of order-type $\mathfrak{c}^+$ which is, as we know, maximal possible. By Corollary \ref{cor.cheap?} below,  CH implies there are Tukey chains of p-points of length $\mathfrak{c}$.
 Dilip Raghavan has shown that, assuming CH, there is a Tukey chain of p-points isomorphic to the reals.
\end{rem}

\begin{question}
 Is there an ultrafilter $\mathcal{U}<_T\mathcal{U}_{\mathrm{top}}$ which is not Tukey reducible to any p-point? 
\end{question}

\begin{question}
Is every basically generated ultrafilter Tukey reducible to a  p-point? 
\end{question}

Both of the preceding two questions are answered using the assumption $\mathcal{U}\not\ge_T\om^{\om}$ for any p-point $\mathcal{U}$ (which is true in the iterated superperfect extension).
Namely, then $\mathcal{U}\cdot\mathcal{U}\not\le_T\mathcal{V}$  for every ultrafilter $\mathcal{U}$ and every p-point $\mathcal{V}$.

\begin{question}\label{question.19}
Is there a p-ideal $I$ on $\om$ which is not countably generated but $I\not\ge_T\om^{\om}$?
\end{question}

\begin{rem}\label{rem.20}
If $\mathfrak{b}\ne\mathfrak{d}$ there is such a p-ideal, 
so the question is whether we can get one with no extra set-theoretic assumptions.
\end{rem}

\begin{question}
Does $\mathcal{U}\cdot\mathcal{U}\equiv_T\mathcal{U}<_T\mathcal{U}_{\mathrm{top}}$ imply $\mathcal{U}$ is basically generated? 
\end{question}


\section{Antichains, chains, and incomparable predecessors}\label{sec.chainantichain}

We now investigate the structure of the Tukey types of p-points and selective ultrafilters in terms of which chains, antichains, and incomparable ultrafilters with a common upper bound embed into the Tukey types.

\begin{thm}\label{thm.selective}
\begin{enumerate}
\item
Assume cov$(\mathscr{M})=\mathfrak{c}$. 
Then there are $2^{\mathfrak{c}}$ pairwise Tukey incomparable selective ultrafilters.
\item
Assume  $\mathfrak{d}=\mathfrak{u}=\mathfrak{c}$.
Then there are $2^{\mathfrak{c}}$ pairwise Tukey incomparable p-points.
\end{enumerate}
\end{thm}

We prove Theorem \ref{thm.selective} by proving it first in the case that $2^{\mathfrak{c}}>\mathfrak{c}^+$ (see Theorem \ref{thm.selective2c>c+}),
and then proving it in the case that $2^{\mathfrak{c}}=\mathfrak{c}^+$  (see Theorem \ref{thm.selective2c=c+}).
Of use will be  two propositions of Ketonen.
Recall  [Theorem 1.7, \cite{Ketonen76}] of Ketonen: If  cov$(\mathscr{M})=\mathfrak{c}$ then every filter with a filter  base of size less than $\mathfrak{c}$ can be extended to a selective ultrafilter.
The key part of his proof uses the following proposition.

\begin{prop}[Ketonen, Proposition 1.8 \cite{Ketonen76}]\label{prop.Ketonen}
If  cov$(\mathscr{M})=\mathfrak{c}$ and $\mathcal{F}$ is a filter generated by less than $\mathfrak{c}$ many sets, and $\{P_i:i<\om\}$ is a partition of $\om$ so that for each $i<\om$, $\bigcup\{P_j:j>i\}\in\mathcal{F}$,
then there exists a set $X\sse\om$ such that $\{X\}\cup\mathcal{F}$ has the  finite intersection property, and for every $i<\om$,
$|X\cap P_i|\le 1$.
\end{prop}

The following proposition of Ketonen was used in his proof of [Theorem 1.2, \cite{Ketonen76}]:  $\mathfrak{d}=\mathfrak{c}$ if and only if
any filter generated by a base of cardinality less than $\mathfrak{c}$ can be extended to a p-point.

\begin{prop}[Ketonen, Proposition 1.3 \cite{Ketonen76}]\label{prop.Ketonen1.3}
If $\mathfrak{d}=\mathfrak{c}$,
then given any filter $\mathcal{F}$ generated by less than $\mathfrak{c}$ elements and a sequence $\lgl A_i:i<\om\rgl$ of elements of $\mathcal{F}$,
there exists a set $A\sse\om$ so that $\mathcal{F}\cup\{A\}$ has the finite intersection property, and for each $i<\om$, $A\sse^* A_i$.
\end{prop}

We are now equipped to prove Theorem \ref{thm.selective}
in the case that $2^{\mathfrak{c}}>\mathfrak{c}^+$.

\begin{thm}\label{thm.selective2c>c+}
Assume $2^{\mathfrak{c}}>\mathfrak{c}^+$.
\begin{enumerate}
\item
Assume cov$(\mathscr{M})=\mathfrak{c}$. 
Then there are $2^{\mathfrak{c}}$ pairwise Tukey incomparable selective ultrafilters.
\item
Assume  $\mathfrak{d}=\mathfrak{u}=\mathfrak{c}$.
Then there are $2^{\mathfrak{c}}$ pairwise Tukey incomparable p-points.
\end{enumerate}
\end{thm}

\begin{proof}
We prove (1) first.
Recall that cov$(\mathscr{M})=\mathfrak{c}$ implies $\mathfrak{u}=\mathfrak{c}$, so every 
filter base of cardinality less than $\mathfrak{c}$ does not generate an ultrafilter.
We fix some notation used throughout the proof.
Fix a listing $\lgl D_{\al}:\al<\mathfrak{c}\rgl$ of all the infinite subsets of $\om$.
There are $\mathfrak{c}$ many partitions of $\om$, so we fix a sequence $\lgl \vec{P}_{\al}:\al<\mathfrak{c}\rgl$ such that each $\vec{P}_{\al}=\lgl P_{\al}^n:n<\om\rgl$ is a partition of $\om$ 
(that is, $\bigcup_{n<\om}P_{\al}^n=\om$ and for each $m\ne n$, $P_{\al}^m\cap P_{\al}^n=\emptyset$)
 and each partition of $\om$ appears  in the listing.
We shall say that a filter $\mathcal{U}$  is \em selective for the partition $\vec{P}_{\al}$ \rm if
either there is some $n<\om$ such that $P_{\al}^n\in\mathcal{U}$ or else there is some $X\in\mathcal{U}$ such that $|X\cap P_{\al}^n|\le 1$ for each $n<\om$.

We now begin the construction.
In a very  similar manner to the proof of [Theorem 2, \cite{Blass73}] of Blass,
 we will construct selective ultrafilters $\mathcal{U}_x$, $x\in 2^{\mathfrak{c}}$, such that for $x\ne y$, $\mathcal{U}_x\ne\mathcal{U}_y$.
Let $\mathcal{U}_{\lgl\rgl}$ be the Fr\'{e}chet filter.
If there is an $i<\om$ such that $P_0^i$ is infinite, then let $\mathcal{U}_{\lgl\rgl}'$ be the filter generated by  $\mathcal{U}_{\lgl\rgl}\cup\{P_0^i\}$.
Otherwise, for each $i<\om$, $P_0^i$ is finite. 
Then take some infinite  $X$ such that for each $i$, $|X\cap P_0^i|\le 1$ and let $\mathcal{U}_{\lgl\rgl}'$ be the filter generated by  $\mathcal{U}_{\lgl\rgl}\cup\{X\}$.
Take $\al_0$ minimal such that both $D_{\al_0}$ and $D_{\al_0}^c$ are in $(\mathcal{U}_{\lgl\rgl}')^+$.
Let $\mathcal{U}_{\lgl 0\rgl}$ be the filter generated by $\mathcal{U}_{\lgl\rgl}'\cup\{D_{\al_0}\}$ and let $\mathcal{U}_{\lgl 1\rgl}$ be the filter generated by $\mathcal{U}_{\lgl\rgl}'\cup\{D_{\al_0}^c\}$.
Note that both $\mathcal{U}_{\lgl 0\rgl}$ and $\mathcal{U}_{\lgl 1\rgl}$  have countable filter bases,  are selective for $\vec{P}_0$,
and any ultrafilter extending $\mathcal{U}_{\lgl i\rgl}$ does not extend $\mathcal{U}_{\lgl 1-i\rgl}$,
for each $i\le 1$.

Suppose for  $t\in 2^{<\mathfrak{c}}$,
 the filter  $\mathcal{U}_t$ has been constructed and has a filter base of cardinality less than $\mathfrak{c}$.
Let $\beta$ be the length of $t$. 
The partition of $\om$ under consideration is 
$\vec{P}_{\beta}=\lgl P_{\beta}^n:n<\om\rgl$.
If there is an $n<\om$ such that 
$P_{\beta}^n\in \mathcal{U}_t$,
then let $\mathcal{U}_t'=\mathcal{U}_t$.
Otherwise, for each $n<\om$,
$\bigcup_{j>n} P_{\beta}^j\in \mathcal{U}_t$.
Apply Proposition  \ref{prop.Ketonen}  to find an $X\in[\om]^{\om}$ such that 
$\{X\}\cup \mathcal{U}_t$ has the finite intersection property, and such that for each $n<\om$,
$|X\cap P_{\beta}^n|\le 1$.
Let $\mathcal{U}'_t$ be the filter generated by $\{X\}\cup \mathcal{U}_t$.
Take $\al_{\beta}$ minimal such that both $D_{\al_{\beta}}$ and $D_{\al_{\beta}}^c$ are in $(\mathcal{U}_t')^+$.
(Note that $\al_{\beta}\ge\beta$.)
Let $\mathcal{U}_{t^{\frown}0}$ be the filter generated by $\mathcal{U}_t'\cup\{D_{\al_{\beta}}\}$ and let $\mathcal{U}_{t^{\frown}1}$ be the filter generated by $\mathcal{U}_t'\cup\{D_{\al_{\beta}}^c\}$.
Note that for each $i\le 1$, both $\mathcal{U}_{t^{\frown}i}$  have  filter bases of cardinality less than $\mathfrak{c}$,  are selective for $\vec{P}_{\beta}$,
and any ultrafilter extending $\mathcal{U}_{t^{\frown} i}$ does not extend $\mathcal{U}_{t^{\frown}(1-i)}$.

For $t\in 2^{<\mathfrak{c}}$ with length of $t$ some limit ordinal $\gamma$,
if for all $\beta<\gamma$, $\mathcal{U}_{t\re\beta}$ has been constructed,
then we let $\mathcal{U}=\bigcup_{\beta<\gamma}\mathcal{U}_{t\re\beta}$.

This constructs filters $\mathcal{U}_t$, $t\in 2^{<\mathfrak{c}}$, satisfying the following.
For each $t\in 2^{<\mathfrak{c}}$,
\begin{enumerate}
\item
$\mathcal{U}_t$ is a filter with a filter base of cardinality less than $\mathfrak{c}$;
\item
If $s$  is an initial segment of  $t$,
then $\mathcal{U}_s\sse\mathcal{U}_t$;
\item
If the length of $t$ is $\al+1$ for some $\al<\mathfrak{c}$,
then 
for all $\beta\le\al$,
$\mathcal{U}_t$ is selective for $\vec{P}_{\beta}$,
 and either $D_{\beta}$ or $D_{\beta}^c$ is in $\mathcal{U}_t$;
\item
No ultrafilter can extend both $\mathcal{U}_{t^{\frown}0}$ and $\mathcal{U}_{t^{\frown}1}$.
\end{enumerate}

For each $x\in 2^{\mathfrak{c}}$,
let $\mathcal{U}_x=\bigcup_{\beta<\mathfrak{c}}\mathcal{U}_{x\re\beta}$.
Then by (1) - (3), each $\mathcal{U}_x$ is a selective  ultrafilter.
Furthermore, (4) implies that for $x,y\in 2^{\mathfrak{c}}$,
if $x\ne y$,
then  $\mathcal{U}_x\ne\mathcal{U}_y$.
Thus, we have $2^{\mathfrak{c}}$ selective ultrafilters. 
By Theorem \ref{thm.5},
each $\mathcal{U}_x$ has Tukey type of cardinality at most $\mathfrak{c}$.
Thus, there are $2^{\mathfrak{c}}$ Tukey types among the collection of Tukey types of the $\mathcal{U}_x$,
$x\in 2^{\mathfrak{c}}$.
Since $2^{\mathfrak{c}}>\mathfrak{c}^+$, Corollary \ref{cor.5.5} yields $2^{\mathfrak{c}}$  Tukey incomparable selective ultrafilters.

The proof of  (2) of the Theorem follows exactly the same steps as for  (1)
with only the following modification which ensures that we build p-points (instead of selective ultrafilters).
Before starting the construction, 
fix an enumeration
$\lgl \vec{A}_{\al}:\al<\mathfrak{c}\rgl$,
where $\vec{A}_{\al}=\lgl A_{\al}^n:n<\om\rgl$,
such that for each countable collection $\vec{B}=\lgl B_n:n<\om\rgl$ of infinite subsets of $\om$,
$\vec{B}=\vec{A}_{\al}$ for cofinally many $\al<\mathfrak{c}$.


We now begin the construction for (2).
Let $\mathcal{U}_{\lgl\rgl}$ be the Fr\'{e}chet filter.
If 
the sequence $\lgl A_0^n:n<\om\rgl$ is contained in $\mathcal{U}_{\lgl\rgl}$,
then
apply Proposition \ref{prop.Ketonen1.3}
to obtain a  set $B$ such that $B\sse^* A_0^n$ for each $n<\om$
and such that $\{B\}\cup\mathcal{U}_{\lgl\rgl}$  has the finite intersection property.
In this case, let $\mathcal{U}_{\lgl\rgl}'$ denote the filter generated by $\{B\}\cup\mathcal{U}_{\lgl\rgl}$.
If the sequence $\lgl A_0^n:n<\om\rgl$ is not contained in $\mathcal{U}_{\lgl\rgl}$,
then let $\mathcal{U}_{\lgl\rgl}'=\mathcal{U}_{\lgl\rgl}$.
Take $\al_0$ minimal such that both $D_{\al_0}$ and $D_{\al_0}^c$ are in $(\mathcal{U}_{\lgl\rgl}')^+$.
Let $\mathcal{U}_{\lgl 0\rgl}$ be the filter generated by $\mathcal{U}_{\lgl\rgl}'\cup\{D_{\al_0}\}$ and let $\mathcal{U}_{\lgl 1\rgl}$ be the filter generated by $\mathcal{U}_{\lgl\rgl}'\cup\{D_{\al_0}^c\}$.

Suppose for  $t\in 2^{<\mathfrak{c}}$,
 the filter  $\mathcal{U}_t$ has been constructed and has a filter base of size less than $\mathfrak{c}$.
Let $\beta$ be the length of $t$. 
If 
the sequence $\lgl A_{\beta}^n:n<\om\rgl$ is contained in $\mathcal{U}_t$,
then
apply Proposition \ref{prop.Ketonen1.3}
to obtain a  set $B$ such that $B\sse^* A_{\beta}^n$ for each $n<\om$
and such that $\{B\}\cup\mathcal{U}_t$  
has the finite intersection property.
In this case, 
let $\mathcal{U}_t'$ denote the filter generated by $\{B\}\cup\mathcal{U}_t$.
If the sequence $\lgl A_{\beta}^n:n<\om\rgl$ is not contained in $\mathcal{U}_t$,
then let $\mathcal{U}_t'=\mathcal{U}_t$.
Take $\al_{\beta}$ minimal such that both $D_{\al_{\beta}}$ and $D_{\al_{\beta}}^c$ are in $(\mathcal{U}_t')^+$.
Let $\mathcal{U}_{t^{\frown}0}$ be the filter generated by $\mathcal{U}_t'\cup\{D_{\al_{\beta}}\}$ and let $\mathcal{U}_{t^{\frown}1}$ be the filter generated by $\mathcal{U}_t'\cup\{D_{\al_{\beta}}^c\}$.
For $t\in 2^{<\mathfrak{c}}$ such that length of $t$ is some limit ordinal $\gamma$,
if for all $\beta<\gamma$, $\mathcal{U}_{t\re\beta}$ has been constructed,
then we let $\mathcal{U}_t=\bigcup_{\beta<\gamma}\mathcal{U}_{t\re\beta}$.

For each $x\in 2^{\mathfrak{c}}$,
let $\mathcal{U}_x=\bigcup_{\beta<\mathfrak{c}}\mathcal{U}_{x\re\beta}$.
By similar arguments as for (1),
 each $\mathcal{U}_x$ is  an ultrafilter
and for $x\ne y$,
$\mathcal{U}_x\ne\mathcal{U}_y$.
Moreover, $\mathfrak{d}=\mathfrak{c}$ implies that the cofinality of $\mathfrak{c}$ is uncountable.
Thus, any countable collection of elements of $\mathcal{U}_x$ appears in $\mathcal{U}_t$ for some $t\in 2^{<\mathfrak{c}}$ such that $t\sqsubseteq x$ and hence is considered at some stage in the construction of $\mathcal{U}_x$.
Thus, $\mathcal{U}_x$ is 
a p-point.
By Theorem \ref{thm.5} and Corollary \ref{cor.5.5},
we obtain $2^{\mathfrak{c}}$  Tukey incomparable p-points.
\end{proof}

Next we take care of the case when $2^{\mathfrak{c}}=\mathfrak{c}^+$.
In this case, Corollary \ref{cor.5.5}
does not apply,
so we present a new way of constructing $\mathfrak{c}^+$ Tukey incomparable selective ultrafilters (or p-points).
To do so we shall use the following notion.

Given a continuous monotone function $f:\mathcal{P}(\om)\ra\mathcal{P}(\om)$,
define $\hat{f}:2^{<\om}\ra\mathcal{P}(\om)$ by letting
$\hat{f}(s)=\bigcap_{n\ge m}f(\tilde{s}\cup[n,\om))$,
for each $m<\om$ and each $s\in 2^m$, 
where $\tilde{s}$ denotes $\{i<m:s(i)=1\}$.
We shall say that  $f$ is {\em presented} by the function $\hat{f}$ if  the following hold:
\begin{enumerate}
\item 
For each $Z\sse\om$,
$f(Z)=\bigcup\{\hat{f}(s):s\sqsubseteq Z\}$,
where $Z$ is identified with its characteristic function;
\item
For any $X\sse\om$ and any $l<\om$,  $l\in f(X)$ iff $l\in \hat{f}(X\cap (l+1))$, where $X\cap (l+1)$ is identifed  with its characteristic function of length $l+1$.
\end{enumerate}

In the proof of Theorem \ref{thm.5}, it was shown that 
for any p-point $\mathcal{Z}$, any ultrafilter $\mathcal{U}$, and any monotone cofinal map $f:\mathcal{Z}\ra\mathcal{U}$,
there is a continuous monotone  map $f^*:\mathcal{P}(\om)\ra\mathcal{P}(\om)$ and a cofinal subset $\mathcal{Z}\re\tilde{X}$ of $\mathcal{Z}$ such that $f^*\re(\mathcal{Z}\re\tilde{X})$ equals $f\re(\mathcal{Z}\re\tilde{X})$.
Moreover, the proof of Theorem \ref{thm.5}  shows that this $f^*$ is presented by  $\hat{f^*}$.
Thus, it suffices to consider only continuous monotone maps $f:\mathcal{P}(\om)\ra\mathcal{P}(\om)$ which are presented  by  $\hat{f}$.

\begin{lem}\label{lem.fix}
Let $f:\mathcal{P}(\om)\ra\mathcal{P}(\om)$ be a continuous monotone map presented by a map $\hat{f}:2^{<\om}\ra\mathcal{P}(\om)$, let $\mathcal{U}$ be a non-principal ultrafilter, and let $\mathcal{Y}$ be a filter containing the Fr\'{e}chet filter with a filter base of size less than $\mathfrak{u}$.
Then there is a $Y\in\mathcal{Y}^+$ such that for any ultrafilter $\mathcal{Z}$ which extends $\mathcal{Y}\cup\{Y\}$,
$f\re\mathcal{Z}$ is not a cofinal map from $\mathcal{Z}$ into $\mathcal{U}$.
\end{lem}

\begin{proof}
Let $f$, $\mathcal{U}$, and $\mathcal{Y}$ satisfy the hypotheses.
If there is a $Y\in\mathcal{Y}^+$ such that $f(Y)\not\in\mathcal{U}$,
then we are done.
So now suppose that for each $Y\in\mathcal{Y}^+$,
$f(Y)\in\mathcal{U}$.
If there is a $U\in\mathcal{U}$ such that for each $Y\in\mathcal{Y}^+$, $f(Y)\not\sse U$,
then for every ultrafilter $\mathcal{Z}$ extending $\mathcal{Y}$,
$f''\mathcal{Z}$ is not cofinal in $\mathcal{U}$.

Thus, the remaining case is that $f''\mathcal{Y}^+$ is cofinal in $\mathcal{U}$, which we assume throughout the rest of the proof of the lemma.
Let $\hat{f}:2^{<\om}\ra\mathcal{P}(\om)$ be given such that $f$ is presented by $\hat{f}$.
Recall that for each $s\in 2^{<\om}$, 
$\hat{f}(s)$ is the set of all $k$ which must be in $f(X)$ for every extension $X$ of $s$, and 
 $\hat{f}$ has the property that for any $X\sse\om$ and any $l<\om$,  $l\in f(X)$ iff $l\in \hat{f}(X\cap (l+1))$.
For a filter $\mathcal{W}$,  the dual ideal is denoted by $\mathcal{W}^*$.

\begin{claim1}\label{claim.1fix}
For any ultrafilter $\mathcal{U}$,
given any collection $\{C_i:i<\om\}\sse\mathcal{U}^*$ such that each $C_i$ is infinite,
there is a $U\in\mathcal{U}$ such that for each $i<\om$,
$C_i\not\sse U$.
\end{claim1}

\begin{proof}
Let $\{C_i:i<\om\}$ be a collection of infinite sets such that each $C_i\in\mathcal{U}^*$.
Let $a_0=\min(C_0)$ and $b_0=\min(C_0\setminus\{a_0\})$.
Let $I_0=\{i<\om:\{a_0,b_0\}\sse C_i\}$ and let 
$I_1=\{i<\om:\{a_0,b_0\}\not\sse C_i\}$.
Let $i_1=\min(I_1)$.
Let $a_1=\min(C_{i_1}\setminus\{a_0,b_0\})$ and let 
$b_1=\min(C_{i_1}\setminus\{a_0,b_0,a_1\})$.
Let $I_2=\{i\in I_1:\{a_1,b_1\}\not\sse C_i\}$.
For general $m$, given $I_m$,
let $i_m=\min(I_m)$,
let $a_m=\min(C_{i_m}\setminus(\{a_j:j<m\}\cup\{b_j:j<m\}))$ and let 
$b_m=\min(C_{i_m}\setminus(\{a_j:j\le m\}\cup\{b_j:j<m\}))$.
Let $I_{m+1}=\{i\in I_m:\{a_m,b_m\}\not\sse C_i\}$.


Let $A=\{a_m:m<\om\}$ and $B=\{b_m:m<\om\}$.
Then $A$ and $B$ are infinite and $A\cap B=\emptyset$.
Moreover, for each $j<\om$,
there is an $m$ such that $j< i_{m+1}$,
so $\{a_{m'},b_{m'}\}\sse C_j$ for some $m'<m+1$.
Hence, $A\cap C_j\ne\emptyset$ and $B\cap C_j\ne\emptyset$.
Therefore, for each $j<\om$,
$C_j\not\sse A$ and $C_j\not\sse B$.
Note that one of $A$ and $B\cup(\om\setminus A)$ must be in $\mathcal{U}$.
However, neither $A$ nor $B\cup(\om\setminus A)$ contains $C_j$ for any $j<\om$.
\end{proof}

Now we exhaust the possible  cases regarding $\hat{f}$.
\vskip.1in

\it Case 1. \rm
For each $X\in\mathcal{Y}^+$, identifying $X$  with its characteristic function, 
there is a finite initial segment $s\sqsubseteq X$  such that $\hat{f}(s)\in\mathcal{U}$.
Let $\mathcal{S}$ be the collection of $s\in 2^{<\om}$ such that $\hat{f}(s)\in\mathcal{U}$.
Then for each $X\in\mathcal{Y}^+$,
$f(X)$ is the union of the $\hat{f}(s)$, where $s\in\mathcal{S}$ and $s\sqsubseteq X$.
Since there are only countably many $\hat{f}(s)$, $s\in\mathcal{S}$,
they cannot generate the ultrafilter $\mathcal{U}$.
Hence, for any ultrafilter extension $\mathcal{Z}$ of $\mathcal{Y}$,
$f''\mathcal{Z}$ is not cofinal in $\mathcal{U}$.
\vskip.1in

\it Case 2.  \rm
Not Case 1.
Then there is an $X_0\in\mathcal{Y}^+$ such that for each finite initial segment $s\sqsubseteq X_0$, $\hat{f}(s)$ is not in $\mathcal{U}$.
\vskip.1in

\it Subcase 2(a). \rm
There is an $X_1\sse X_0$ in $\mathcal{Y}^+$ such that for each $Y\in\mathcal{Y}^+$ with $Y\sse X_1$,
there is a finite initial segment $s$ of $Y$ such that $\hat{f}(s)$ is infinite.
Let $\mathcal{S}$ be the collection of finite initial segments $s$ of members $Y\sse X_1$ in $\mathcal{Y}^+$ such that $\hat{f}(s)$ is infinite.
Then $\{\hat{f}(s):s\in\mathcal{S}\}$ satisfies the hypotheses of Claim 1.
Thus, there is a $U\in\mathcal{U}$ such that for each $s\in\mathcal{S}$, 
$\hat{f}(s)\not\sse U$.
Therefore, for any ultrafilter $\mathcal{Z}$ extending $\mathcal{Y}\cup\{X_1\}$,
$f''\mathcal{Z}$ is not cofinal in $\mathcal{U}$,
since for any $Z\in\mathcal{Z}$, $f(Z)=\bigcup\{\hat{f}(s):s\sqsubseteq Z\}$.
\vskip.1in

\it Subcase 2(b). \rm
For each $X_1\sse X_0$ in $\mathcal{Y}^+$,
there is some $X_2\sse X_1$ also in $\mathcal{Y}^+$ such that for each finite initial segment $s\sqsubseteq X_2$,
$\hat{f}(s)$ is finite.
Fix some such $X_2$.
Then note that for each $s\in 2^{<\om}$ such that $\tilde{s}\in[X_2]^{<\om}$,
$\hat{f}(s)$ is finite.
(Recall that $\tilde{s}$ denotes $\{i\in\dom(s):s(i)=1\}$.)
Let $\mathcal{S}_2$ denote $\{s\in 2^{<\om}:\tilde{s}\sse X_2\}$.

\begin{claim2}
There is a $Y\in\mathcal{Y}^+$ such that $Y\sse X_2$ and $f(Y)\not\in\mathcal{U}$.
\end{claim2}

\begin{proof}
Since each $\hat{f}(s)$ is finite for $s\in \mathcal{S}_2$,
for each $k$ there is an $m$
such that for each $s\in 2^k\cap\mathcal{S}_2$,
$\max(\hat{f}(s))<m$.
Let $j_0=0$.
Given $j_i$,
choose $j_{i+1}$ to be the least  $m>j_i$ such that for each $s\in 2^{j_i}\cap\mathcal{S}_2$,
$\max(\hat{f}(s))<m$.
Notice that for each $i<\om$ and each $s\in 2^{j_i}\cap\mathcal{S}_2$,
we have that $\max(\hat{f}(s))<j_{i+1}$.

Let $\mathcal{W}$ be the filter generated by $\mathcal{Y}\cup\{X_2\}$.
Then $\mathcal{W}$ has a base of size less than $\mathfrak{u}$ (since $\mathcal{Y}$ does), so $\mathcal{W}$ is not an ultrafilter.
Let $H=\bigcup_{i<\om}[j_{2i},j_{2i+1})$.
Then $H$ and $H^c$ cannot both be in $\mathcal{W}^*$, since $\mathcal{W}^*$ is a proper ideal.
Without loss of generality, assume that $H\not\in\mathcal{W}^*$.
Then $H\in\mathcal{W}^+$.
(If $H$ is in $\mathcal{W}^*$, then use $H^c$ and modify the indexes in the following argument.)

\begin{subclaimn}
There is an infinite, co-infinite set $K\sse\om$ such that both $\bigcup_{i\in K}[j_{2i},j_{2i+1})$ and $\bigcup_{i\in K^c}[j_{2i},j_{2i+1})$ are in $\mathcal{W}^+$.
\end{subclaimn}

\begin{proof}
For each $i<\om$, let $\bar{j}_i$ denote the interval $[j_{2i},j_{2i+1})$.
Let $\mathcal{K}=\{K\sse\om:\exists W\in\mathcal{W}\ \forall i<\om\ (W\cap\bar{j}_i\ne\emptyset\ra i\in K)\}$.
Note that $\mathcal{K}$ is a filter:
By its definition, $\mathcal{K}$ is closed under supersets and contains the Fr\'{e}chet filter since $\mathcal{W}\contains\mathcal{Y}$ contains the Fr\'{e}chet filter.
Also if $K$ and $K'$ are in $\mathcal{K}$ as witnessed by $W,W'\in\mathcal{W}$, respectively,
then $W\cap W'\in\mathcal{W}$, and 
$W\cap W'$ witnesses that $K\cap K'\in\mathcal{K}$.

Let $\mathcal{C}$ be a base of size less than $\mathfrak{u}$ for the filter $\mathcal{W}$.
For each $W\in\mathcal{C}$,
define $K_W=\{i\in\om:W\cap\bar{j}_i\ne\emptyset\}$.
Let $\mathcal{B}=\{K_W:W\in\mathcal{C}\}$.
Note that $\mathcal{B}$ is a base for the filter $\mathcal{K}$.
Also, $|\mathcal{B}|\le|\mathcal{C}|<\mathfrak{u}$,
so $\mathcal{K}$ is not an ultrafilter.
Thus, we can fix a $K\in\mathcal{K}^+\setminus\mathcal{K}$.
Then also $K^c\in\mathcal{K}^+\setminus\mathcal{K}$;
 so $K$ and $K^c$ are both infinite.
Define $A$ to be $\bigcup_{i\in K}[j_{2i},j_{2i+1})$
and $B$ to be $\bigcup_{i\in K^c}[j_{2i},j_{2i+1})$.
Note that both $A$ and $B$ are subsets of $H$,  $A\cap B=\emptyset$, and $A\cup B=H$.

We claim that both $A$ and $B$ are in $\mathcal{W}^+$.
Since $K\in\mathcal{K}^+$, it follows that for each $J\in\mathcal{K}$, $|K\cap J|=\om$.
Since  $\mathcal{B}$ generates $\mathcal{K}$,
we have that for each $W\in\mathcal{C}$,
$|K\cap K_W|=\om$.
Therefore, for each $W\in \mathcal{C}$,
$\{i\in K:W\cap \bar{j}_i\ne\emptyset\}$ is infinite.
Thus,
$A\cap W=(\bigcup_{i\in K}\bar{j}_i)\cap W$ is infinite for each $W\in \mathcal{C}$. 
Hence,   $A\cap W$ is infinite for each $W\in\mathcal{W}$.
Thus, $A\in\mathcal{W}^+$.
Likewise, since $K^c$ is in $\mathcal{K}^+$,
we have that $B\in\mathcal{W}^+$.
This finishes the proof of the Subclaim.
\end{proof}

We claim that $f(A)\cap f(B)=\emptyset$.
We shall prove more: For any $I\sse\om$,
$f(\bigcup_{i\in I}[j_{2i},j_{2i+1}))\sse\bigcup_{i\in I}[j_{2i},j_{2i+2})$.
It suffices to prove this for all finite $I\sse\om$ since for any $I\sse\om$,
$f(\bigcup_{i\in I}[j_{2i},j_{2i+1}))=\bigcup_{k<\om}\hat{f}(\bigcup_{i\in I\cap k}[j_{2i},j_{2i+1}))$.

$\hat{f}(\emptyset)$ must be the emptyset, (for if not, then $f$ would not map $\mathcal{Y}^+$ cofinally into $\mathcal{U}$).
$\hat{f}([j_0,j_1))\sse[j_0,j_2)$, by definition of 
$j_2$.
Suppose that $k\ge 1$ and given any  finite $I\sse k$,
$\hat{f}(\bigcup_{i\in I}[j_{2i},j_{2i+1}))\sse\bigcup_{i\in I}[j_{2i},j_{2i+2})$.
Let $I'\sse k+1$ be given and let $I$ denote $I'\cap k$.
By the induction hypothesis, 
$\hat{f}(\bigcup_{i\in I}[j_{2i},j_{2i+1}))\sse\bigcup_{i\in I}[j_{2i},j_{2i+2})$.
If $I=I'$, we are done.
If $I\ne I'$, then $k\in I'$.
Recall the fact that  $\hat{f}$ has the property that for any $X\sse\om$ and any $l<\om$,  $l\in f(X)$ iff $l\in \hat{f}(X\cap (l+1))$.
Hence, by our choice of the $j_i$,
we have that $\hat{f}(\bigcup_{i\in I'}[j_{2i},j_{2i+1}))\cap j_{2k}=
\hat{f}(\bigcup_{i\in I}[j_{2i},j_{2i+1}))$.
Thus,
$\hat{f}(\bigcup_{i\in I'}[j_{2i},j_{2i+1}))\sse\bigcup_{i\in I'}[j_{2i},j_{2i+2})$.

Thus, $f(A)\cap f(B)=\emptyset$.
This implies that at least one 
of them is not in $\mathcal{U}$.
Thus, Claim 2 holds.
\end{proof}

Taking a $Y\in\mathcal{Y}^+$ satisfying Claim 2 contradicts the hypothesis that $f''\mathcal{Y}^+\sse\mathcal{U}$.
Thus, the Lemma holds.
\end{proof}

\begin{thm}\label{thm.selective2c=c+}
\begin{enumerate}
\item
Assume cov$(\mathscr{M})=\mathfrak{c}$. 
Then there are $\mathfrak{c}^+$ pairwise Tukey incomparable selective ultrafilters.
\item
Assume  $\mathfrak{d}=\mathfrak{u}=\mathfrak{c}$.
Then there are $\mathfrak{c}^+$ pairwise Tukey incomparable p-points.
\end{enumerate}
\end{thm}

\begin{proof}
Proof of (1).
Assume cov$(\mathscr{M})=\mathfrak{c}$. 
To show that there are $\mathfrak{c}^+$ Tukey incomparable selective ultrafilters,
we shall show that given $\le \mathfrak{c}$ selective ultrafilters,
there is another selective ultrafilter Tukey incomparable with each of them.

Let $\mathcal{U}_{\gamma}$, $\gamma<\kappa$, where $\kappa\le\mathfrak{c}$, be a collection of selective ultrafilters.
Fix a listing $\lgl D_{\al}:\al<\mathfrak{c}\rgl$ of all the infinite subsets of $\om$.
Fix a sequence $\lgl \vec{P}_{\al}:\al<\mathfrak{c}\rgl$ such that each $\vec{P}_{\al}=\lgl P_{\al}^n:n<\om\rgl$ is a partition of $\om$ 
 and each partition of $\om$ appears  in the listing.
Fix a listing $\lgl f_{\beta}:\beta<\mathfrak{c}\rgl$ of all continuous monotone maps $f:\mathcal{P}(\om)\ra\mathcal{P}(\om)$ which is represented by  $\hat{f}:2^{<\om}\ra\mathcal{P}(\om)$.
Finally, fix an onto function $\theta:\mathfrak{c}\ra\{\mathcal{U}_{\gamma}:\gamma<\kappa\}\times\{f_{\beta}:\beta<\mathfrak{c}\}$.

We will construct filters $\mathcal{Y}_{\al}$,  $\al<\mathfrak{c}$, satisfying the following:
\begin{enumerate}
\item
For $\al<\al'<\mathfrak{c}$,
$\mathcal{Y}_{\al}\sse \mathcal{Y}_{\al'}$;
\item
$\mathcal{Y}_{\al}$
has a base of cardinality less than $\mathfrak{c}$;
\item
$\mathcal{Y}_{\al+1}$ is selective for $\vec{P}_{\al}$;
\item
Either $D_{\al}$ or $D_{\al}^c$ is in $\mathcal{Y}_{\al+1}$;
\item
If $\theta(\al)$ is the pair $\lgl \mathcal{U}_{\gamma_{\al}},f_{\beta_{\al}}\rgl$,
then
for each ultrafilter $\mathcal{Z}$ extending $\mathcal{Y}_{\al+1}$,
$f_{\beta_{\al}}\re\mathcal{U}_{\gamma_{\al}}$ does not map $\mathcal{U}_{\gamma_{\al}}$ cofinally into $\mathcal{Z}$,
and $f_{\beta_{\al}}\re\mathcal{Z}$ does not map $\mathcal{Z}$ cofinally into $\mathcal{U}_{\gamma_{\al}}$.
\end{enumerate}

We now begin the construction.
Let $\mathcal{Y}_0$ be the Fr\'{e}chet filter.
Suppose 
 the filter  $\mathcal{Y}_{\al}$ has been constructed.
The partition of $\om$ under consideration is 
$\vec{P}_{\al}=\lgl P_{\al}^n:n<\om\rgl$.
If there is an $n<\om$ such that 
$P_{\al}^n\in \mathcal{Y}_{\al}$,
then let $\mathcal{Y}_{\al+1}^{(0)}=\mathcal{Y}_{\al}$.
Otherwise, for each $n<\om$,
$\bigcup_{j>n} P_{\al}^j\in \mathcal{Y}_{\al}$.
Apply Proposition  \ref{prop.Ketonen}  to find an $X\in[\om]^{\om}$ such that 
$\{X\}\cup \mathcal{Y}_{\al}$ has the finite intersection property, and such that for each $n<\om$,
$|X\cap P_{\al}^n|\le 1$.
Then let $\mathcal{Y}_{\al+1}^{(0)}$ be the filter generated by $\{X\}\cup \mathcal{Y}_{\al}$.
If $D_{\al}\in(\mathcal{Y}_{\al+1}^{(0)})^+$,
then let $\mathcal{Y}_{\al+1}^{(1)}$ be the filter generated by 
$\{D_{\al}\}\cup \mathcal{Y}_{\al+1}^{(0)}$.
Otherwise, let 
 $\mathcal{Y}_{\al+1}^{(1)}$ be the filter generated by 
$\{D_{\al}^c\}\cup \mathcal{Y}_{\al+1}^{(0)}$.

Next we consider  $\theta(\al)$, which is a pair $
\lgl \mathcal{U}_{\gamma_{\al}},f_{\beta_{\al}}\rgl$ for some $\gamma_{\al}<\kappa$ and $\beta_{\al}<\mathfrak{c}$.
If $f_{\beta_{\al}}''\mathcal{U}_{\gamma_{\al}}\sse\mathcal{Y}_{\al+1}^{(1)}$,
then $f_{\beta_{\al}}\re\mathcal{U}_{\gamma_{\al}}$ will not be cofinal into any ultrafilter extending $\mathcal{Y}_{\al+1}^{(1)}$.
In this case, let $\mathcal{Y}_{\al+1}^{(2)}=\mathcal{Y}_{\al+1}^{(1)}$.
If $f_{\beta_{\al}}''\mathcal{U}_{\gamma_{\al}}\not\sse\mathcal{Y}_{\al+1}^{(1)}$,
then take some $U\in\mathcal{U}_{\gamma_{\al}}$ such that $f_{\beta_{\al}}(U)\not\in\mathcal{Y}_{\al+1}^{(1)}$ and 
let $\mathcal{Y}_{\al+1}^{(2)}$ be the filter generated by $\mathcal{Y}_{\al+1}^{(1)}\cup\{f_{\beta_{\al}}(U)^c\}$.
Note that $f_{\beta_{\al}}\re\mathcal{U}_{\gamma_{\al}}$ cannot be cofinal into any ultrafilter extending $\mathcal{Y}_{\al+1}^{(2)}$.
By Lemma \ref{lem.fix},
there is a 
$Y\in(\mathcal{Y}_{\al+1}^{(2)})^+$ such that for any ultrafilter $\mathcal{Z}$ which extends $\mathcal{Y}_{\al+1}^{(2)}\cup\{Y\}$,
$f_{\beta_{\al}}\re\mathcal{Z}$ is not a cofinal map from $\mathcal{Z}$ into $\mathcal{U}_{\gamma_{\al}}$.
Let $\mathcal{Y}_{\al+1}$ be the filter generated by $\mathcal{Y}_{\al+1}^{(2)}\cup\{Y\}$.

For limit ordinals $\lambda<\mathfrak{c}$,
let $\mathcal{Y}_{\lambda}=\bigcup_{\al<\lambda}\mathcal{Y}_{\al}$.

Let $\mathcal{Y}=\bigcup_{\al<\mathfrak{c}}\mathcal{Y}_{\al}$.
Then $\mathcal{Y}$ is a selective ultrafilter, by (1) - (4).
Moreover, $\mathcal{Y}$ is Tukey incomparable with each $\mathcal{U}_{\gamma}$, $\gamma<\kappa$, by (5).

Since for each collection of selective ultrafilters of cardinality less than or equal to $\mathfrak{c}$ we can build another selective ultrafilter which is Tukey inequivalent to each of them, 
it follows that there are $\mathfrak{c}^+$ Tukey inequivalent selective ultrafilters.
\vskip.1in

The proof of  (2) of the Theorem follows exactly the same steps as for  (1)
with only the following modification.
Before starting the construction, 
let $\mathcal{U}_{\gamma}$, $\gamma<\kappa$, where $\kappa\le\mathfrak{c}$, be a collection of p-points.
Fix an enumeration
$\lgl \vec{A}_{\al}:\al<\mathfrak{c}\rgl$,
where $\vec{A}_{\al}=\lgl A_{\al}^n:n<\om\rgl$,
such that for each countable collection $\vec{B}=\lgl B_n:n<\om\rgl$ of infinite subsets of $\om$,
$\vec{B}=\vec{A}_{\al}$ for cofinally many $\al<\mathfrak{c}$.

Let $\mathcal{Y}_0$ be the Fr\'{e}chet filter.
Given the filter $\mathcal{Y}_{\al}$,
if 
the collection $\{ A_{\al}^n:n<\om\}$ is not contained in $\mathcal{Y}_{\al}$,
then let $\mathcal{Y}_{\al+1}^{(0)}=\mathcal{Y}_{\al}$.
If $\{ A_{\al}^n:n<\om\}$ is  contained in $\mathcal{Y}_{\al}$,
apply Proposition \ref{prop.Ketonen1.3}
to obtain a  set $B$ such that $B\sse^* A_{\al}^n$ for each $n<\om$
and such that $\{B\}\cup\mathcal{Y}_{\al}$  has the finite intersection property.
In this case, let $\mathcal{Y}_{\al+1}^{(0)}$ denote the filter generated by $\{B\}\cup\mathcal{Y}_{\al}$.
The rest of the construction of $\mathcal{Y}_{\al+1}$ proceeds exactly as in part (1).
Letting $\mathcal{Y}=\bigcup_{\al<\mathfrak{c}}\mathcal{Y}_{\al}$,
we see that $\mathcal{Y}$ is a p-point which is Tukey inequivalent to every p-point $\mathcal{U}_{\gamma}$, $\gamma<\kappa$.
The Theorem then follows as in part (1).
\end{proof}

Theorem \ref{thm.selective} follows from Theorem \ref{thm.selective2c>c+}
and Theorem \ref{thm.selective2c=c+}.

\begin{rem}
The stipulation  in (1) in Theorem \ref{thm.selective}
that cov$(\mathcal{M})=\mathfrak{c}$ is optimal, at least for this construction.
For by results of Fremlin and Canjar, (see Theorem 4.6.6 of  \cite{Bartoszynski/JudahBK}),
cov$(\mathcal{M})=\mathfrak{c}$ iff every filter with base of cardinality less than $\mathfrak{c}$ can be extended to a selective ultrafilter.
The stipulation in (2) of Theorem \ref{thm.selective} that $\mathfrak{u}=\mathfrak{d}=\mathfrak{c}$ is perhaps not  optimal, since p-points exist just under the assumption that $\mathfrak{d}=\mathfrak{c}$.
It remains open whether, just assuming $\mathfrak{d}=\mathfrak{c}$, there are $2^{\kappa}$ Tukey incomparable ultrafilters for any $\kappa$ such that cf$(\kappa)=$ cf$(\mathfrak{c})$ and $2^{<\kappa}=\mathfrak{c}$.
\end{rem}

One way of making Tukey increasing chains of ultrafilters is by using $\kappa$-OK points.
We give the following definition straight from \cite{Kunen78}.

\begin{defn}[Kunen \cite{Kunen78}]\label{defn.OK}
Let $X$ be a topological space and $\kappa$ any cardinal.
If $p\in X$ and $U_n$ $(n<\om)$ are neighborhoods of $p$, a {\em $\kappa$-refinement system} for $\lgl U_n:n<\om\rgl$ is a $\kappa$-sequence of neighborhoods of $p$, $\lgl V_{\al}:\al<\kappa\rgl$ such that for all $n\ge 1$,
$$
\forall \al_1<\al_2<\dots<\al_n<\kappa \ \ (V_{\al_1}\cap\dots \cap V_{\al_n}\sse U_n).
$$
A point $p\in X$ is {\em $\kappa$-OK} iff whenever $U_n$ ($n<\om$) are neighborhoods of $p$, 
$\lgl U_n:n<\om\rgl$ has a $\kappa$-refinement system.
\end{defn}

Translating this into the context of ultrafilters,
we let $X$ be the \v{C}ech-Stone remainder $\beta{\om}\setminus\om$,  the collection of all non-principle ultrafilters on $\om$.
A non-principle ultrafilter $\mathcal{U}$ 
is $\kappa$-OK iff whenever $U_n\in\mathcal{U}$ ($n<\om$),
there is a $\kappa$-sequence $\lgl V_{\al}:\al<\kappa\rgl$ of elements of $\mathcal{U}$ such that for all $n\ge 1$,
for all $\al_1<\dots<\al_n<\kappa$,
$V_{\al_1}\cap\dots\cap V_{\al_n}\sse^* U_n$.

Kunen remarked in \cite{Kunen78} 
that if $\mathcal{U}$ is $\kappa$-OK and $\kappa>$ cof$(\mathcal{U})$, then $\mathcal{U}$ is a p-point.
It is easy to see the following.

\begin{prop}
If $\mathcal{U}$ is  $\kappa$-OK but not a p-point, then $\mathcal{U}\ge_T[\kappa]^{<\om}$.
Hence, if $\mathcal{U}$ is  $\kappa$-OK but not a p-point, then cof$(\mathcal{U})=\kappa$ iff $\mathcal{U}\equiv_T [\kappa]^{<\om}$.
\end{prop}

\begin{proof}
Let $\mathcal{U}$ be $\kappa$-OK but not a p-point. 
Then there are $X_n\in\mathcal{U}$ such that for each $X\in\mathcal{U}$, there is an $n<\om$ such that $X\not\sse^* X_n$.
Let $\{C_{\al}:\al\in[\kappa]^{<\om}\}\sse\mathcal{U}$ witness that $\mathcal{U}$ is $\kappa$-OK for $\lgl X_n\rgl_{n<\om}$.
Let $g:[\kappa]^{<\om}\ra\mathcal{U}$ by
$g(\al)=C_{\al}$ for each $\al\in[\kappa]^{<\om}$.
If $\mathcal{X}\sse[\kappa]^{<\om}$ is unbounded, then $\mathcal{X}$ is infinite.
Hence, $g''\mathcal{X}$ is infinite, since $g$ is 1-1.
Take $\{C_{\al_n}:n<\om\}$ to be any infinite subset of  $g''\mathcal{X}$.
Suppose $\{C_{\al_n}:n<\om\}$ is $\contains^*$ bounded below by $Y\in\mathcal{U}$.
Then for each $k$,
$Y\sse^*\bigcap_{n\le k} C_{\al_n}\sse^* X_k$.
But then for each $n$, $Y\sse^* X_n$, contradicting our choice of $\{X_n:n<\om\}$.
Thus, $g: [\kappa]^{<\om}\ra(\mathcal{U},\contains^*)$ is a Tukey map.
Therefore, $[\kappa]^{<\om}\le_T(\mathcal{U},\contains^*)\le_T(\mathcal{U},\contains)$.

If cof$(\mathcal{U})\ne\kappa$,
then $\mathcal{U}\not\equiv_T[\kappa]^{<\om}$; 
hence, $\mathcal{U}>_T[\kappa]^{<\om}$.
If cof$(\mathcal{U})=\kappa$,
then $\mathcal{U}\le_T[\kappa]^{<\om}$.
\end{proof}

It follows that if there are $\kappa$-OK non p-points with cofinality $\kappa$ for each uncountable $\kappa<\mathfrak{c}$, then there is a strictly increasing chain of ultrafilters of length $\al$, where $\al$ is such that $\aleph_{\al}=\mathfrak{c}$.
We would like to point out that Milovich showed in \cite{Milovich08}
 that and ultrafilter $\mathcal{U}$ is  $\mathfrak{c}$-OK and not of top degree iff $\mathcal{U}$ is a p-point.

We now give a general method for building Tukey increasing chains of p-points.

\begin{thm}\label{thm.p-point.chains}
Assuming CH,
for each p-point $D$ there is a p-point $E$ such that 
  $E>_{RK} D$ and moreover, $E>_T D$.
\end{thm}

\begin{proof}
We use the notation from \cite{Blass73}.
In [Theorem 6, \cite{Blass73}],
Blass proved assuming MA that given a 
 p-point $D$ one can construct a p-point $E>_{RK} D$.
Hence, $E\ge_T D$.
His construction can be slightly modified to kill all possible cofinal maps from $D$ into $E$ so that we construct a p-point $E$ which is both Rudin-Keisler and Tukey strictly above $D$.

Let $D$ be a given p-point. 
Fix a bijective pairing $J:\om\times\om\ra\om$ with inverse $(\pi_1,\pi_2)$, and identify $\om$ with $\om\times\om$ via $J$.
A subset $Y\sse\om\times\om$ is called {\em small} iff the function $c_Y(i):=|\{y\in\om:(i,y)\in Y\}|$ is bounded by some $n<\om$ for all $i$ in some $X\in D$.  
Otherwise $Y$ is called {\em large}.
It is useful to note that  from [Lemma 1, p152, \cite{Blass73}], it follows that 
$\om\times\om$ is large, the union of any two small sets is small, the complement of a small set is large, and any superset of a large set is large.
We give the following characterization of large sets.

\begin{claim1}
Let $Y\sse\om\times\om$.
$Y$ is large iff there is a $W\in D$ such that $c_Y\re W$ is bounded below by a non-decreasing, unbounded function on $W$.
\end{claim1}

\begin{proof}
First note that for any $Y\sse\om\times\om$,
$Y$ is large iff for each $n<\om$, $\{i<\om:c_Y(i)\le n\}\not\in D$
iff for each $n<\om$,
$\{i<\om:c_Y(i)>n\}\in D$.
Let $Y\sse\om\times\om$ be large.
For each $n<\om$, define $W_n=\{i<\om:c_Y(i)>n\}$.
Then each $W_n\in D$ and $W_n\contains W_{n+1}$.
Since $D$ is a p-point, there is a $W\in D$ such that for each $n<\om$, $W\sse^* W_n$.
Let $k_n$ be a strictly increasing sequence such that for each $n<\om$,
$W\setminus k_n\sse W_n$.
Note that for each $i\in W\setminus k_n$,  $c_Y(i)>n$.
Therefore, for each $n<\om$, for each $i\in W\cap(k_n,k_{n+1}]$, $c_Y(i)>n$.
Hence, $c_Y$ is bounded below on $W$ by the function $g:W\ra\om$,  where for each $n$, for each $i\in W\cap(k_n,k_{n+1}]$,
$g(i)=n$.

For the reverse direction, if $Y\sse\om\times\om$, $W\in D$ and $c_Y\re W$ is bounded below by a non-decreasing unbounded  function, then for each $n<\om$, $\{i\in W:c_Y(i)\le n\}$ is finite, hence  $\{i<\om:c_Y(i)\le n\}\not\in D$. 
 Therefore, $Y$ is large.
\end{proof}

For the sake of readability, we repeat an argument of Blass [pp 151-152, \cite{Blass73}] in this paragraph.
We are going to construct a p-point $E$ on $\om\times\om$ such that $\pi_1(E)=D$.
To ensure that $E\not\equiv_{RK} D$, it  will suffice that $\pi_1$ is not one-to-one on any set of $E$.
This means that $E$ must contain the complement of the graph of every function from $\om$ to $\om$.
Hence, $E$ must also contain the complement of every finite union of such graphs. 
If $Y$ is the graph of a function, then $Y$ is small, for $c_Y$ is bounded by $1$ on all of $\om$.
Also, if $A\in D$ and $Y=(\om\times\om)-\pi_1^{-1}(A)$,
then $Y$ is small, for $c_Y$ is bounded by $0$ on $A$.
Therefore, if $E$ is an ultrafilter on $\om\times\om$ containing no small set, then $E>_{RK} D$.

We now construct an ultrafilter $E$ in $\om_1$ stages.
Let $\lgl f_{\al}:\al<\om_1\rgl$ enumerate all functions from $\om\times\om$ into $\om$,
and let $\lgl h_{\al}:\al<\om_1\rgl$ enumerate
all continuous monotone maps from $\mathcal{P}(\om)$ into $\mathcal{P}(\om)$.
We build filter bases $\mathcal{Y}_{\al}$, $\al<\om_1$, with the following properties.
\begin{enumerate}
\item
Every set in $\mathcal{Y}_{\al}$ is large.
\item
If $\beta<\al<\om_1$,
then $\mathcal{Y}_{\beta}\sse\mathcal{Y}_{\al}$.
\item
$\mathcal{Y}_{\al}$ is countable.
\item
$f_{\al}$ is finite-to-one or bounded on some set of $\mathcal{Y}_{\al+1}$.
\item
$h_{\al}\re D$ is not a cofinal map from $D$ into any ultrafilter extending $\mathcal{Y}_{\al+1}$.
\end{enumerate}

Let $\mathcal{Y}_0=\{\om\times\om\}$.
If $\al<\om_1$ is a limit ordinal and $\mathcal{Y}_{\beta}$ has been constructed for all $\beta<\al$,
then let $\mathcal{Y}_{\al}=\bigcup_{\beta<\al}\mathcal{Y}_{\al}$.

If $\mathcal{Y}_{\al}$ is given,  do the following.
By [Lemma 3, p 153, \cite{Blass73}],
there is a set $T\sse\om\times\om$ on which $f_{\al}$ is finite-to-one or bounded, and such that $T\cap Y$ is large for each $Y\in\mathcal{Y}_{\al}$.
Let $\mathcal{Y}_{\al}'$ be the filter base obtained by adjoining $T$ to $\mathcal{Y}_{\al}$ and closing under finite intersections.

Next, consider the continuous monotone map $h_{\al}$.
If $h_{\al}"D$ does not generate an ultrafilter, there is nothing to do; let $\mathcal{Y}_{\al+1}=\mathcal{Y}_{\al}'$.
Suppose now that  $h_{\al}"D$ generates an ultrafilter.

\begin{claim2}\label{claim.>T}
There is a set $Z$ such that $Z\cap Y$ and $Z^c\cap Y$ are large for each $Y\in \mathcal{Y}'_{\al}$.
\end{claim2}

\begin{proof}
By the inductive construction,  $\mathcal{Y}'_{\al}$ is countable and every element of $\mathcal{Y}'_{\al}$ is large.
Let $X_n$ ($n<\om$) be a base for $\mathcal{Y}'_{\al}$ such that each $X_n\contains X_{n+1}$.
Since each $X_n$ is large, by Claim 1, there is a $W_n\in D$ and a non-decreasing unbounded  function $g_n:W_n\ra\om$  such that for each $i\in W_n$,  $c_{X_n}(i)\ge g_n(i)$.
Without loss of generality, we can assume that each $W_n\contains W_{n+1}$.
Since $D$ is a p-point, let $W\in D$ satisfy for each $n<\om$, $W\sse^* W_n$.

We shall build disjoint $Z_0,Z_1\sse\om\times\om$ and a strictly increasing sequence $\lgl k_n:n<\om\rgl$ as follows.
Let $k_0$ be least such that for each $i\in[k_0,\om)\cap W_0$, 
$g_0(i)\ge 2$ and $W\setminus k_0\sse W_0$.
In general, choose $k_{m}>k_{m-1}$ satisfying
\begin{enumerate}
\item
for each $j\le m$ and each 
$i\in [k_m,\om)\cap W_j$, $g_j(i)\ge 2(m+1)^2$;
\item
$W\setminus k_m\sse W_m$ (and hence for each $j<m$, $W\setminus k_m\sse W_j$).
\end{enumerate}
Given $m<\om$ and $i\in W\cap[k_m,k_{m+1})$,
for each $j\le m$,
choose $x_{i,j,l},y_{i,j,l}$, $l\le m$, distinct in $\{z\in\om: (i,z)\in X_j\}\setminus\{x_{i,q,l},y_{i,q,l}:l\le m,\ q<j\}$.
(This is possible since for each $i\in W\cap[k_m,k_{m+1})$, for each $j\le m$, 
$c_{X_j}(i)\ge g_j(i)\ge 2(m+1)^2$.)
For each $i\in W$, define $m_i$ to be the
integer $m$ for which $i\in[k_m,k_{m+1})$.
Define $Z_0=\{(i,x_{i,j,l}):i\in W$, $j\le m_i$, $l\le m_i\}$;
$Z_1=\{(i,y_{i,j,l}):i\in W$, $j\le m_i$, $l\le m_i\}$.
Note that $Z_0,Z_1$ are large, disjoint, and have large intersection with each $X_n$.
Letting $Z=Z_0$, then both $Z$ and $Z^c$ have 
the desired properties.  
\end{proof}

Take $Z$ as in Claim 2.
If $Z\in h_{\al}"D$, let $\mathcal{Y}_{\al+1}$ be the filter base  obtained by closing $\mathcal{Y}'_{\al}\cup\{Z^c\}$ under finite intersections; 
and if $Z^c\in h_{\al}"D$ then
let $\mathcal{Y}_{\al+1}$ be the filter base obtained by closing $\mathcal{Y}'_{\al}\cup\{Z\}$ under finite intersections.
Then $h_{\al}\re D$ cannot be a cofinal map from $D$ into any ultrafilter extending $\mathcal{Y}_{\al+1}$.

As in the final argument of [Theorem 6, \cite{Blass73}],
let $\mathcal{Y}=\bigcup_{\al<\om_1}\mathcal{Y}_{\al}$,
and let
$\mathcal{B}$ be the filter of all sets whose complements are small.
Every set of $\mathcal{Y}$, being large, has infinite intersection with  every set of $\mathcal{B}$, so there is an ultrafilter
$E$  extending $\mathcal{Y}\cup\mathcal{B}$. 
Then $E>_{RK} D$, and $E$ is a p-point since requirement (4) is met for all  $\al<\om_1$.
Moreover,
 $E>_T D$, since for every continuous monotone map  $h:\mathcal{P}(\om)\ra\mathcal{P}(\om)$, $h\re D$ is not a cofinal map from $D$ into $E$.
\end{proof}

\begin{rem}
Dilip Raghavan has independently observed Theorem \ref{thm.p-point.chains}.
\end{rem}

\begin{rem}
If one is only interested in building an ultrafilter $E$ Tukey strictly above $D$, then one does not have to use large sets in the previous construction, but one  only needs to ensure that $E$ is a p-point and that all continuous monotone maps are prevented from being cofinal maps from $D$ into $E$.  
In the above proof, we used large sets to ensure that $E$  also be Rudin-Keisler strictly above $D$ in order to obtain the following Corollary.
\end{rem}

\begin{cor}\label{cor.cheap?}
Assuming CH, there is a Tukey strictly increasing chain of p-points of order type $\mathfrak{c}$.
\end{cor}

\begin{proof}
In
[Theorem 7, \cite{Blass73}], Blass proved that MA implies that any RK increasing chain of p-points of length $\om$ has an RK upper bound which is a p-point.
The p-point $E$ constructed in the above Theorem \ref{thm.p-point.chains}
is also RK strictly above $D$, so 
for any $\al<\om_1$,
we can construct $\om$-length chains of p-points
$D_{\al+n}$, where each $D_{\al+n+1}>_T D_{\al+n}$ and $D_{\al+n+1}>_{RK} D_{\al+n}$ ($\al<\om_1$) and 
 then use [Theorem 7, \cite{Blass73}] to find a p-point RK above each $D_{\al+n}$, $n<\om$, hence also Tukey above them. 
\end{proof}

The following questions are to be answered assuming that p-points exist or some assumption that guarantees their existence.

\begin{question}
Is there  a Tukey strictly increasing chain of p-points of length $\mathfrak{c}^+$?  
\end{question}

The Tukey increasing chain of p-points constructed in the proof of Theorem \ref{thm.p-point.chains}
is also Rudin-Keisler increasing.
This leads to the next question.

\begin{question}\label{q.31}
Given any strictly Tukey increasing sequence of p-points of length $\om$, is there always a p-point Tukey above all of them?
\end{question}

In particular,

\begin{question}\label{q.30}
Given any p-point $\mathcal{V}$, is there a p-point $\mathcal{U}$ such that $\mathcal{U}>_T\mathcal{V}$, but
$\mathcal{U}$ and $\mathcal{V}$ are RK-incomparable?
\end{question}

If the answer to Question \ref{q.30} is no, then the answer to Question \ref{q.31} is yes.

We now show that, assuming Martin's Axiom, there are incomparable p-points with a common upper bound and a common lower bound which are also  p-points.

\begin{thm}\label{thm.TnoncomparablePpoints}
Assume Martin's Axiom.
There is a p-point $D$ with two Tukey-incomparable  Tukey predecessors $\pi_1(D)$ and $\pi_2(D)$ which are also p-points, which in turn have a common Tukey lower bound $E$ which is also a p-point.
(In the following diagram, arrows represent strict Tukey reducibility.)
$$
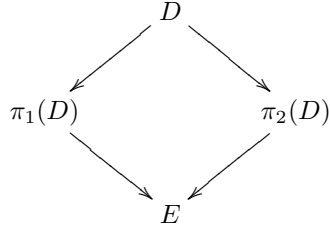
\xymatrix{
&D \ar[ld] \ar[rd] &\\
\pi_1(D) \ar[rd] & & \pi_2(D)\ar[ld] &\\
& E&}
$$
\end{thm}

\begin{proof}
In [Theorem 9, \cite{Blass73}], Blass proved that  assuming Martin's Axiom, there is a p-point with two RK-incomparable predecessors.
He used  the following notions which we shall also use.
A subset of $\om\times\om$ of the form $P\times Q$, where $P$ and $Q$ are subsets of $\om$ of cardinality $n<\om$, is called an {\em $n$-square}.
A subset of $\om\times\om$ is called {\em large} if it includes an $n$-square for every $n$, and {\em small} otherwise.
Blass' construction builds a p-point $D\sse\om\times\om$ consisting of large sets such that $\pi_1(D)$ and $\pi_2(D)$ are RK-incomparable. 
For $i=1,2$, $\pi_i(D)\le_{RK} D$, hence $\pi_i(D)$ are also p-points and are  $\le_T D$.
The fact that every member of $D$ is large ensures that $\pi_1(D)$ and $\pi_2(D)$ are non-principal.

The following Lemma will be useful for constructing the desired $D$.

\begin{lem}\label{lem.T}
Given $\mathcal{Y}$ a filter base on $\om\times\om$ of size $<\mathfrak{c}$ and a  monotone function  $h:\mathcal{P}(\om)\ra\mathcal{P}(\om)$,
there is a large set $U$
such that $U\sse^* Y$ for each $Y\in\mathcal{Y}$,
 and
 for any ultrafilter $D'\contains \mathcal{Y}\cup\{U\}$ consisting only of large sets, $h\re\pi_1(D')$ is not a cofinal map from $\pi_i(D')\ra\pi_{j}(D')$, for $i\ne j$.
\end{lem}

\begin{proof}
By [Lemma 2, Section 6, \cite{Blass73}] (which uses MA),
there is a large set $X$ such that $X\sse^* Y$ for each $Y\in\mathcal{Y}$.
Since $X$ is large, we can choose  $L_k\sse X$, $k<\om$, such that $L_k$ is a  $(2k)$-square and  $\lgl \pi_1(L_k):k<\om\rgl$, $\lgl \pi_2(L_k):k<\om\rgl$ form block sequences;
that is, for each $k<\om$ and $i=1,2$, each element in $\pi_i(L_k)$ is less than each element in $\pi_i(L_{k+1})$.
Let $I=\bigcup_{k<\om}\pi_1(L_k)$ and $J=\bigcup_{k<\om}\pi_2(L_k)$.

\it Case 1. \rm
There is an infinite $I'\sse I$ such that letting $J'=(\om\setminus h(I'))\cap J$ and $m_k=\min\{|I'\cap\pi_1(L_k)|,|J'\cap\pi_2(L_k)|\}$, 
the sequence $\lgl m_k:k<\om\rgl$ is unbounded.
Then there is a strictly increasing subsequence $\lgl m_{k_n}:n<\om\rgl$.
Let $W=\bigcup_{n<\om}(I'\cap\pi_1(L_{k_n}))\times (J'\cap\pi_2(L_{k_n}))$.
Then $W\sse  X$ and $W$ is large.
Note that if $D'$ is any ultrafilter extending $\mathcal{Y}\cup\{W\}$, then
$I'=\pi_1(W)$ is in $\pi_1(D')$ and $h(I')$ is disjoint from $J'=\pi_2(W)$ which is in $\pi_2(D')$.
Therefore, $f(I')\not\in\pi_2(D')$.

\it Case 2.  \rm
Not Case 1.  Then
for each infinite $I'\sse I$,  letting $J'=(\om\setminus h(I'))\cap J$, there is an $m<\om$ such that  $\min\{|I'\cap\pi_1(L_k)|,|J'\cap\pi_2(L_k)|\}\le m$ for each $k<\om$.
Let $W'=\bigcup_{k<\om}L_k$.  
Then $W'\sse X$ and $W'$ is large.

\begin{claim}
For any 
$I'\sse I$ such that $I'=\pi_1(V')$ for some large $V'\sse W'$,
there is a strictly increasing sequence $\lgl k_n:n<\om\rgl$ and an $m<\om$ 
such that  for each $n$, $|h(I')\cap\pi_2(L_{k_n})|\ge 2k_n -m$.
\end{claim}

\begin{proof}
Let $I'\sse I$ be  such that $I'=\pi_1(V')$ for some large $V'\sse W'$,
and let $J'=(\om\setminus h(I'))\cap J$.
Since we are in Case 2, there is an $m<\om$  satisfying
$\min\{|I'\cap\pi_1(L_k)|,|J'\cap\pi_2(L_k)|\}\le m$ for each $k<\om$.
Since $V'$ is large and $V'\sse W'$,
there is a subsequence $\lgl k_n:n<\om\rgl$ such that 
$\lgl |I'\cap\pi_1(L_{k_n})|: n<\om\rgl$ is a strictly increasing sequence 
of numbers greater than $m$.
Then for each $n<\om$,
it must be the case that $|J'\cap\pi_2(L_{k_n})|\le m$.
Note that for each $n$,
$(\om\setminus h(I'))\cap\pi_2(L_{k_n})
= (\om\setminus h(I'))\cap J\cap\pi_2(L_{k_n})
=J'\cap\pi_2(L_{k_n})$,
since $\pi_2(L_{k_n})=J\cap\pi_2(L_{k_n})$.
Thus, for each $n$, $|(\om\setminus h(I'))\cap \pi_2(L_{k_n})|=
|J'\cap\pi_2(L_{k_n})|\le m$.
Since $|\pi_2(L_{k_n})|=2k_n$,
it follows that $|h(I')\cap\pi_2(L_{k_n})|\ge 2k_n-m$.
\end{proof}

Divide each $\pi_2(L_k)$ into two disjoint sets each of size $k$, labeling one of them $M_k$.
Let $J^*=\bigcup_{k<\om}M_k$.
Let $W=W'\cap (\om\times J^*)$.
Then $W\sse X$ and $W$ is large.
Let $D'$ be any ultrafilter extending 
$\mathcal{Y}\cup\{W\}$
consisting only of large sets.
Since $W\in D'$, we have that $J^*\in\pi_2(D')$.
We claim that for all $I'\in\pi_1(D')$, $h(I')\not\sse J^*$.

Let $I'$ be any member of $\pi_1(D')$.
Then there is a $V''\in D'$ such that $V''\sse W$ and $I'':=\pi_1(V'')\sse I'$.
By the Claim,
there is a strictly increasing sequence $\lgl k_n:n<\om\rgl$ and an $m$ such that
for each $n$,
$|h(I'')\cap\pi_2(L_{k_n})|\ge 2k_n-m$.
However, for each $n$,
$|J^*\cap \pi_2(L_{k_n})|=|M_{k_n}|=k_n$,
 which is less than $2k_n-m$ for all large enough $n$.
Thus, $h(I'')\not\sse J^*$.
Since $h$ is monotone, $h(I')$ also cannot be contained in $J^*$.
Thus, $h\re \pi_1(D')$ is not a cofinal map from $\pi_1(D')$ into $\pi_2(D')$.
This ends Case 2.

Thus, in both Cases 1 and 2, we have found a large $W$ such that $W\sse^* Y$ for all $Y\in\mathcal{Y}$ and such that for any ultrafilter $D'$ extending $\mathcal{Y}\cup\{W\}$ consisting only of large sets,
$h\re\pi_1(D')$ is not a cofinal map from $\pi_1(D')$ into $\pi_2(D')$. 
Now repeat the entire above argument starting with $W$ in place of $X$ and reversing the roles of $\pi_1$ and $\pi_2$  to obtain a large $U\sse W$ such that 
 for any ultrafilter $D'\contains \mathcal{Y}\cup\{U\}$ consisting only of large sets, $h\re\pi_2(D')$ is not a cofinal map from $\pi_2(D')$ into $\pi_1(D')$.
This finishes the proof of the Lemma.
\end{proof}

Now we construct the desired p-point $D$ on $\om\times\om$.
Enumerate $\mathcal{P}(\om\times\om)$ as $A_{\al}$, $\al<\mathfrak{c}$, 
and 
enumerate  all  continuous monotone  maps from  $\mathcal{P}(\om)$ into $\mathcal{P}(\om)$ as  $h_{\al}$, $\al<\mathfrak{c}$.
We construct filter bases $\mathcal{Y}_{\al}$, $\al<\mathfrak{c}$, which satisfy the following.
\begin{enumerate}
\item
$\mathcal{Y}_{\al}$ is a filter base of size less than $\mathfrak{c}$.
\item
Every set in $\mathcal{Y}_{\al}$ is large.
\item
If $\beta<\al<\mathfrak{c}$,
then $\mathcal{Y}_{\beta}\sse\mathcal{Y}_{\al}$.
\item
Either $A_{\al}$ or $\om\times\om\setminus A_{\al}$ is in $\mathcal{Y}_{\al+1}$.
\item
There is a $U\in\mathcal{Y}_{\al+1}$ such that $U\sse^* Y$ for each $Y\in\mathcal{Y}_{\al}$.
\item
For any ultrafilter $D'$ extending $\mathcal{Y}_{\al+1}$ consisting only of large sets,
$f_{\al}\re\pi_1(D')$ is not a cofinal map from $\pi_1(D')$ into $\pi_2(D')$,
 and 
$f_{\al}\re\pi_2(D')$ is not a cofinal map from $\pi_2(D')$ into $\pi_1(D')$.
\end{enumerate}

Let $\mathcal{Y}_0=\{\om\times\om\}$.
If $\al$ is a limit ordinal and $\mathcal{Y}_{\beta}$
has been defined for all $\beta<\al$,
then let $\mathcal{Y}_{\al}=\bigcup_{\beta<\al}\mathcal{Y}_{\beta}$.

In the case that  $\mathcal{Y}_{\al}$
has been constructed,
construct $\mathcal{Y}_{\al+1}$ as follows.
By [Lemma 2, p 162, \cite{Blass73}],
there is a large $T$ such that $T\sse^* Y$ for each $Y\in\mathcal{Y}_{\al}$.
If $A_{\al}\cap T$ is large, then let $\mathcal{Y}_{\al}'=\mathcal{Y}_{\al}\cup\{A_{\al}\cap T\}$.
Otherwise, $A_{\al}\cap T$ is small.
Since $T$ is large, then $T\setminus A_{\al}$ is large, by [Lemma 1, p 162, \cite{Blass73}];
so let $\mathcal{Y}_{\al}'=\mathcal{Y}_{\al}\cup\{T\setminus A_{\al}\}$.

Next
 use  Lemma \ref{lem.T} for $\mathcal{Y}'_{\al}$ and $h_{\al}$
to obtain a large $U_{\al}$ such that 
such that $U_{\al}\sse^* Y$ for each $Y\in\mathcal{Y}'_{\al}$,
 and
 for any ultrafilter $D'\contains \mathcal{Y}'_{\al}\cup\{U_{\al}\}$ consisting only of large sets, $h_{\al}\re\pi_i(D')$ is not a cofinal map from $\pi_i(D')$ into $\pi_{j}(D')$, for $i\le 1$ and $j=1-i$. 
Let $\mathcal{Y}_{\al+1}=\mathcal{Y}'_{\al}\cup\{U_{\al}\}$.

Let $\mathcal{D}=\bigcup_{\al<\mathfrak{c}}\mathcal{Y}_{\al}$.
By (2), $\pi_1(D)$ and $\pi_2(D)$ are non-principal;
by (4), $D$ is an ultrafilter;
by (5), $D$ is a p-point;
and by (6), $\pi_1(D)$ and $\pi_2(D)$ are Tukey-incomparable.
Since $\pi_1(D)$ and $\pi_2(D)$ are Rudin-Keisler below $D$, they are also p-points.
Moreover, since the p-point $D$ is  Rudin-Keisler above both $\pi_1(D)$ and $\pi_2(D)$,
it
 follows from [Theorem 5,  \cite{Blass73}] that there is a p-point which is Rudin-Keisler (hence Tukey) below both $\pi_1(D)$ and $\pi_2(D)$.  
Thus, assuming MA, the diamond lattice embeds into the Tukey degrees of p-points.
\end{proof}

[Theorem 5,  \cite{Blass73}] states that if countably many p-points have an RK upper bound which is a p-point, then they have an RK lower bound (which is necessarily a p-point).

\begin{question}
If countably many p-points have a Tukey upper bound which is a p-point, do they necessarily have a Tukey lower bound which is a p-point?
\end{question}

\begin{question}
Does every Tukey strictly decreasing sequence of p-points have a Tukey lower bound which is a p-point?
\end{question}

\begin{rem}
Laflamme showed in \cite{Laflamme90} that in the NCF model of \cite{Blass/Shelah89}, the RK ordering of p-points is upwards directed, and hence also downwards directed.
Thus, in the NCF model, the Tukey degrees of p-points are both upwards and downwards directed.
(We know by Theorem \ref{thm.2} that
the class of basically generated ultrafilters with bases closed under finite intersections is upwards directed.)
Recall that the cardinal inequality $\mathfrak{u}<\mathfrak{g}$ implies NCF (see \cite{BlassLaflamme89}),
so it is natural to ask the following.
\end{rem}

\begin{question}
Does $\mathfrak{u}<\mathfrak{g}$ imply there is a minimal  Tukey degree  in the class of p-points?
\end{question}


\section{Block-basic ultrafilters on $\FIN$}\label{sec.FIN}

In this section we study the Tukey ordering between idempotent ultrafilters $\mathcal{U}$ on the index set $\FIN$
and their  Rudin-Keisler predecessors
$\mathcal{U}_{\min,\max}$, $\mathcal{U}_{\min}$, and $\mathcal{U}_{\max}$.
We begin by giving the relevant definitions for this investigation.

The following definitions may all be found in \cite{Argyros/TodorcevicBK}.
We let $\FIN$ denote the collection of nonempty finite subsets of $\om$.
Note that $\FIN$ is countable and can serve as a base set for ultrafilters.
Because of the natural structure on $\FIN$, which we shall give shortly, the ultrafilters on $\FIN$ may have some extra structure which can be utilized in the study of their Tukey types.
The set $\FIN$ carries the semigroup operation $\cup$,
where for $x,y\in\FIN$ such that $\max(x)<\min(y)$, $x\cup y$ is defined to be $\{i\in\om:i\in x$ or $i\in y\}$, the usual union.
(If $\max(x)\not\le\min(y)$, then  $x\cup y$ is undefined.)
This operation naturally extends to a semigroup operation on the collection $\beta\FIN$ of ultrafilters on $\FIN$, that is,  the \v{C}ech-Stone compactification of $\FIN$, 
as follows.
For $\mathcal{U}$ and $\mathcal{V}$ ultrafilters on $\FIN$,
$\mathcal{U}\cup\mathcal{V}$ is defined to be the collection of all $A\sse\FIN$ such that $\{x\in\FIN:\{y\in\FIN:x\cup y\in A\}\in\mathcal{U}\}\in\mathcal{V}$.
An {\em idempotent ultrafilter} on the semigroup $(\FIN,\cup)$ is an ultrafilter $\mathcal{U}$ on $\FIN$ such that $\mathcal{U}\cup\mathcal{U}=\mathcal{U}$.
The existence of idempotent ultrafilters on  $\FIN$ was established by S.\ Glazer (see \cite{ComfortBK74}).


At this point, we define some standard maps.
The map
$\min:\FIN\ra\om$ is given by $\min(x)$ is the least element of $x$, for any $x\in\FIN$.
Likewise, $\max:\FIN\ra\om$ is defined by letting $\max(x)$ be the largest element of $x$.
The map $(\min,\max):\FIN\ra\om\times\om$  is defined by $(\min,\max)(x)=(\min(x),\max(x))$.
Note that whenever $\mathcal{U}$ is an ultrafilter on $\FIN$,
then the following are ultrafilters:
$\mathcal{U}_{\min}$ is the ultrafilter on $\om$ generated by the collection of sets $\{\min(x):x\in U\}$,  $U\in\mathcal{U}$.
$\mathcal{U}_{\max}$ is the ultrafilter on $\om$ generated by the collection of sets $\{\max(x):x\in U\}$, $U\in\mathcal{U}$.
$\mathcal{U}_{\min,\max}$ is the ultrafilter on $\om\times\om$ generated by the collection of sets 
$\{(\min(x),\max(x)):x\in U\}$, $U\in\mathcal{U}$.
Note that these are all ultrafilters, since they are images of $\mathcal{U}$ under the Rudin-Keisler maps $\min$, $\max$, and $(\min,\max)$, respectively.
Thus, it also follows that 
$\mathcal{U}\ge_{RK}\mathcal{U}_{\min,\max}$,
$\mathcal{U}_{\min,\max}\ge_{RK}\mathcal{U}_{\min}$, and
$\mathcal{U}_{\min,\max}\ge_{RK}\mathcal{U}_{\max}$.
Thus, the same Tukey reductions between these ultrafilters hold.

In \cite{Blass87}, Blass showed that Glazer's proof easily adapts to show the following.

\begin{thm}[Blass, Theorem 2.1, \cite{Blass87}]\label{thm.BlassGlazer}
Let $\mathcal{V}_0$ and $\mathcal{V}_1$ be a pair of nonprincipal ultrafilters on $\om$.
Then there is an idempotent ultrafilter $\mathcal{U}$ on $\FIN$ such that $\mathcal{U}_{\min}=\mathcal{V}_0$ and $\mathcal{U}_{\max}=\mathcal{V}_1$.
\end{thm}

\begin{cor}\label
{cor.BlassGlazer}
There exist idempotent ultrafilters on $\FIN$ realizing the maximal Tukey type $\mathcal{U}_{\mathrm{top}}$.
\end{cor}

\begin{proof}
Let $\mathcal{V}_0=\mathcal{V}_1$ be a nonprincipal ultrafilter on $\om$  such that $\mathcal{V}_0\equiv_T[\mathfrak{c}]^{<\om}$.
Then by Theorem \ref{thm.BlassGlazer},
$\mathcal{U}_{\min}=\mathcal{U}_{\max}=\mathcal{V}_0$.
Since $\mathcal{U}\ge_{RK}\mathcal{U}_{\min}$,
we have that $\mathcal{U}\ge_T\mathcal{V}_0$, which implies that $\mathcal{U}$ has the top Tukey type.
\end{proof}

Thus, one is naturally led to consider the conditions on idempotent ultrafilters $\mathcal{U}$ on $\FIN$ that would prevent $\mathcal{U}$ from having the maximal Tukey type.

\begin{defn}\label{def.block-gen}
A {\em block-sequence} of $\FIN$ is an infinite sequence $X=(x_n)_{n<\om}$ of elements of  $\FIN$ such that for each $n<\om$,
$\max(x_n)<\min(x_{n+1})$.
For a block-sequence $X$,
we let $[X]$ denote  $\{x_{n_1}\cup\dots\cup x_{n_k}:k<\om$ and $n_1<\dots <n_k\}$, the set of finite unions of elements of $X$.
For any $m<\om$,
let $X/m$ denote $(x_n)_{n\ge k}$ where $k$ is least such that $\min(x_k)\ge m$.

The collection of block-sequences carry the following partial ordering $\le$.
For two infinite block-sequences $X=(x_n)_{n<\om}$ and $Y=(y_n)_{n<\om}$, define
$Y\le X$ iff each member of $Y$ is a finite union of elements of $X$; i.e.\ $y_n\in [X]$ for each $n$.
We write $Y\le^* X$ to mean that $Y/m\le X$ for some $m<\om$.
That is, $Y\le^* X$ iff  there is some $k$ such that for all $n\ge k$, $y_n\in[X]$.

An idempotent  ultrafilter $\mathcal{U}$ on $\FIN$ is called {\em block-generated} if it is generated by
sets of the form $[X]$ where $X$ is an infinite block-sequence.
(Block-generated ultrafilters are called 
\it ordered-union ultrafilters
\rm in \cite{Blass87}.)
\end{defn}

We now state some relevant information about block-generated ultrafilters, much of which was proved by Blass in \cite{Blass87}.

\begin{fact}\label{facts.beginning}
Let $\mathcal{U}$ be any nonprincipal block-generated ultrafilter on $\FIN$.
\begin{enumerate}
\item
(Proposition 3.3, \cite{Blass87}) $\mathcal{U}$ is idempotent.
\item
(Corollary 3.6, \cite{Blass87})
$\mathcal{U}$ is not a p-point.
\item
$\mathcal{U}$ is not a q-point.
\item
(Corollary 3.7, \cite{Blass87})
$\mathcal{U}_{\min,\max}$ is isomorphic (i.e.\ Rudin-Keisler equivalent) to $\mathcal{U}_{\min}\cdot\mathcal{U}_{\max}$.
\item
(Proposition 3.9 \cite{Blass87})
$\mathcal{U}_{\min}$ and $\mathcal{U}_{\max}$
are q-points.
\item
$\mathcal{U}_{\min,\max}$ is neither a p-point nor a q-point.
\item
If $\mathcal{U}_{\min}$ is selective, then $\mathcal{U}_{\min,\max}$ is rapid.
\end{enumerate}
\end{fact}

\begin{proof}
(3) We provide a proof of (3) since it does not seem to yet be in the literature, though most likely it has been noticed before.
Let $\mathcal{U}$ be a nonprincipal block-generated ultrafilter on $\FIN$ and let $P_n=\{x\in\FIN:\max(x)=n\}$.
Then $(P_n)_{n<\om}$ forms a partition of $\FIN$ into finite sets.
Let $U$ be any element of $\mathcal{U}$.
Since $\mathcal{U}$ is block-generated,
 there is some block sequence $X$ such that $[X]\in\mathcal{U}$ and $[X]\sse U$.
Let $y$ be any member of $X$ except the first member of $X$, and let $n=\max(y)$.
Then there is an $x\in X$ such that $\max(x)<\min(y)$.
Thus, both $x\cup y$ and $y$ are in $[X]\cap P_n$.
Hence, for each $U\in\mathcal{U}$, there is some $n$ such that $|U\cap P_n|\ge 2$.
Therefore, $\mathcal{U}$ is not a q-point.

(6) Let $M_n=\{x_{\min,\max}:\min(x)=n\}$.
Then $\{M_n:n<\om\}$ is a partition of $\om$.
If $X$ is any block sequence, then $|[X]_{\min,\max}\cap M_n|=\om$ for infinitely many $n$.
So $\mathcal{U}_{\min,\max}$ is not a p-point.
Let $P_n=\{\iota(\{k,n\}):k<n\}$, where $\iota$ is some fixed pairing function.
Then for each $n\ge 1$, $P_n$ is finite, and $\{P_n:n\ge 1\}$ is a partition of $\om$.
If $X$ is a  block sequence,
then $|[X]_{\min,\max}\cap P_n|>1$
for infinitely many $n$.
Hence, $\mathcal{U}_{\min,\max}$ is not a q-point.

(7) Given a strictly increasing function $g:\om\ra\om$, without loss of generality assuming the coding function $\iota:[\om]^2\ra\om$ has the property that $\iota(\{m,n\})\ge n$ for each $m<n$,
let $k_{l}= 2^{l+1}$ for all $l<\om$.
Since $\mathcal{U}_{\min}$ is selective, there is an infinite block-sequence $X$  such that $[X]\in\mathcal{U}$, $|X_{\min}\cap[0,g(k_2)]|=0$, and for each $l\ge 2$, $|X_{\min}\cap(g(k_l),g(k_{l+1})]|\le 1$.
Then $|[X]_{\min,\max}\cap g(n)|< n$
for each $n<\om$.
\end{proof}

By (5),  the existence of block-generated ultrafilters on $\FIN$ cannot be proved on the basis of the usual ZFC axioms of set theory,
though using Hindman's Theorem one can easily establish the existence of such ultrafilters using CH or MA.

As noted above, no nontrivial  idempotent ultrafilter on $\FIN$ is basic, since such an ultrafilter is never a p-point,
so we are naturally led to the following relaxation of this notion.

\begin{defn}\label{def.block-basic}
For infinite block sequences $X_n=(x^n_k)_{k<\om}$ and $X=(x_k)_{k<\om}$, the sequence
$(X_n)_{n<\om}$ {\em converges} to $X$ (written $X_n\ra X$ as $n\ra\infty$)
if for each $l<\om$ there is an $m<\om$ such that for all $n\ge m$ and all $k\le l$,
$x_k^n=x_k$.
A block-generated ultrafilter $\mathcal{U}$ is {\em block-basic} if
whenever we are given a sequence $(X_n)_{n<\om}$ of infinite block sequences of elements of $\FIN$ such that each $[X_n]\in\mathcal{U}$ 
and $(X_n)_{n<\om}$ converges to some infinite block sequence $X$ such that $[X]\in\mathcal{U}$,
then there is an infinite subsequence $(X_{n_k})_{k<\om}$ such that
$\bigcap_{k<\om}[X_{n_k}]\in\mathcal{U}$.
\end{defn}

\begin{defn}\label{defn.45.5}
Let $\FIN^{[n]}$ denote the collection of all block sequences of elements of $\FIN$ of length $n$.
A block-generated ultrafilter $\mathcal{U}$ on $\FIN$ has the {\em 2-dimensional Ramsey Property}
if
for each finite coloring of $\FIN^{[2]}$, 
there is an infinite block sequence $X$ such that $[X]\in\mathcal{U}$ and
$[X]^{[2]}$ is monochromatic.
A block-generated ultrafilter $\mathcal{U}$ on $\FIN$ has the {\em Ramsey Property}
if
for each $n<\om$ and each finite coloring of $\FIN^{[n]}$, 
there is an infinite block sequence $X$ such that $[X]\in\mathcal{U}$ and
$[X]^{[n]}$ is monochromatic.
Let $\FIN^{[\infty]}$ denote the collection of all infinite block sequences of elements of $\FIN$.
A block-generated ultrafilter
$\mathcal{U}$ on $\FIN$ has the {\em $\infty$-dimensional Ramsey Property} if
for every analytic subset $\mathcal{A}$ of $\FIN^{[\infty]}$ there is an
infinite
block sequence $X$ such that  $[X]\in\mathcal{U}$ and  $[X]^{[\infty]}$ is either included in or
disjoint from $\mathcal{A}$.
(For more information about $\infty$-dimensional Ramsey Theory, see \cite{TodorcevicBK10}.)
\end{defn}

The following theorem shows how the notion of block-basic ultrafilters fits with several equivalences shown by Blass in \cite{Blass87}.

\begin{thm}\label{thm.blockbasicequiv}
The following are equivalent for a block-generated ultrafilter $\mathcal{U}$ on $\FIN$.
\begin{enumerate}
\item
$\mathcal{U}$ is block-basic.
\item
For every sequence $(X_n)$ of infinite block sequences of $\FIN$ such that $[X_n]\in\mathcal{U}$ and $X_{n+1}\le^* X_n$ for each $n$,
there is an infinite block sequence $X$ such that $[X]\in\mathcal{U}$ and $X\le^* X_n$ for each $n$.
\item
$\mathcal{U}$ has the 2-dimensional Ramsey property.
\item
$\mathcal{U}$ has the  Ramsey property.
\item
$\mathcal{U}$ has the $\infty$-dimensional Ramsey property.
\end{enumerate}
\end{thm}

\begin{rem}
 (2) is called a \it stable ordered-union ultrafilter \rm in \cite{Blass87}.
\end{rem}

\begin{proof}
The equivalence of (2), (3), (4) and (5) were established in [Theorem 4.2, \cite{Blass87}].

(1) implies (2).
Suppose $\mathcal{U}$ is block-basic.
Let $(X_n)_{n<\om}$ be a sequence of block-sequences of $\FIN$ such that $[X_n]\in\mathcal{U}$ and $X_{n+1}\le^* X_n$ for each $n$.
Let $(m_n)_{n<\om}$ be a strictly increasing sequence such that $X_0\ge X_1/m_1\ge X_2/m_2\ge\dots$.
Let $Y_n={(\{l\})_{l\le m_n}}^{\frown}(X_n/m_n)$.
Then each $Y_n=^* X_n$ and $Y_n\ra (\{l\})_{l<\om}$.
By (1) there is a subsequence $(n_k)_{k<\om}$ such that $\bigcap_{k<\om}[Y_{n_k}]\in\mathcal{U}$.
Since $\mathcal{U}$ is block-generated,
there is a $Z$ such that $[Z]\in\mathcal{U}$ and $[Z]\sse\bigcap_{k<\om}[Y_{n_k}]$.
Then for each $n<\om$,
taking $k$ such that $n_k>n$, we have that 
$X_n=^* Y_n\ge^* Y_{n_k}\ge Z$.
Thus, (2) holds.

Now suppose that (2) holds.  Since $\mathcal{U}$ is block-generated,
 (2) is equivalent to the statement $(2)'$: 
For every sequence $(X_n)_{n<\om}$ of infinite block sequences of $\FIN$ such that each $[X_n]\in\mathcal{U}$,
there is an infinite block sequence $X$ such that  $[X]\in\mathcal{U}$ and $X\le^* X_n$ for each $n$.
Let $(X_n)_{n<\om}$ be a sequence of block sequences such that each $[X_n]\in\mathcal{U}$ and $(X_n)_{n<\om}\ra X$.
By $(2)'$, there is a $Z\le X_0$ such that $[Z]\in\mathcal{U}$ and for each $n<\om$, $Z\le^* X_n$.
Thus, there is a strictly increasing sequence $(m_k)_{k<\om}$ such that each $m_k=\min(z)$ for some $z\in Z$ and
\begin{enumerate}
\item[(a)]
$n\ge m_{k+1}$ implies  
$X_n\cap m_k=X\cap m_k$;
\item[(b)]
$n\le m_k$ implies $X_n/ m_{k+1}\ge Z$.
\end{enumerate}
Let $Z_0=\{z\in Z:\exists k(m_{4k}\le\min(z)<m_{4k+2})\}$.
If $[Z_0]\in\mathcal{U}$,
then take some $Y\le Z_0,X$ such that $[Y]\in\mathcal{U}$.
For each $k<\om$,
$X_{4k+3}\cap m_{4k+2}=X\cap m_{4k+2}\ge Y\cap m_{4k+2}$.
For each $y\in Y$, $y\cap[m_{4k+2}, m_{4k+4})=\emptyset$.
$X_{4k+3}/m_{4k+4}\ge Z\ge Y$.
Therefore, 
 $\bigcap_{k<\om}[X_{4k+3}]\contains [Y]$.
If $[Z_0]\not\in\mathcal{U}$,
then since $\mathcal{U}$ is block-generated, 
there is a $Z_1$ such that $[Z_1]\in\mathcal{U}$ and
$[Z_1]\sse [Z]\setminus[Z_0]$.
Since $Z_1\le Z$ and $[Z_1]\cap[Z_0]=\emptyset$,
for each $z\in Z$, if $\min(z)\in[m_{4k},m_{4k+2})$ then $z\not\in Z_1$.
Therefore, $Z_1\cap Z_0=\emptyset$.
Hence, for each $z\in Z_1$, $\min(z)\in[m_{4k+2},m_{4k+4})$.
Letting $Y\le Z_1,X$ such that $[Y]\in\mathcal{U}$, $\bigcap_{k<\om}[X_{4k+1}]\contains [Y]$.  
Hence, (1) holds.
\end{proof}

\begin{rem}
Blass showed in \cite{Blass87},  that for every  stable ordered-union ultrafilter $\mathcal{U}$ on $\FIN$, 
both $\mathcal{U}_{\min}$ and $\mathcal{U}_{\max}$ are non-isomorphic selective ultrafilters.
Thus, we have the following corollary.
\end{rem}

\begin{cor}\label{cor.blocksel}
If $\mathcal{U}$ is a block-basic ultrafilter on $\FIN$, 
then $\mathcal{U}_{\min}$ and $\mathcal{U}_{\max}$ are 
Rudin-Keisler incomparable selective ultrafilters on $\om$.
\end{cor}

\begin{rem}
It follows by [Theorem 10, \cite{Raghavan/Todorcevic11}] that for any block-basic ultrafilter $\mathcal{U}$ on $\FIN$,
$\mathcal{U}_{\min}$ and $\mathcal{U}_{\max}$ are Tukey-incomparable.
\end{rem}

Applying [Theorem 2.4, \cite{Blass87}] of Blass,
we get some sort of converse to the previous corollary.

\begin{cor}\label{cor.V0V1tominmax}
Assuming CH,
for every pair $\mathcal{V}_0$ and $\mathcal{V}_1$ of non-isomorphic selective ultrafilters on $\om$,
there is a block-basic ultrafilter $\mathcal{U}$ on $\FIN$ such that $\mathcal{U}_{\min}=\mathcal{V}_0$ and $\mathcal{U}_{\max}=\mathcal{V}_1$.
\end{cor}

Our interest in block-basic ultrafilters on $\FIN$ is based on the following fact whose proof is analogous to that of Theorem \ref{thm.5}.

\begin{thm}\label{thm.5'}
Suppose $\mathcal{U}$ is a block-basic ultrafilter on $\FIN$ and that $\mathcal{U}\ge_T\mathcal{V}$ for some ultrafilter $\mathcal{V}$ on any countable index set $I$.
Then there is a monotone continuous map $f:\mathcal{P}(\FIN)\ra\mathcal{P}(I)$ such that
$f"\mathcal{U}$ is a cofinal subset of $\mathcal{V}$.
\end{thm}


Though the proof the next theorem follows the general outline of that of Theorem \ref{thm.5}, 
we include the proof here since it does use some extra arguments.

\begin{thm}\label{thm.5.2}
Suppose $\mathcal{U}$ is a block-basic ultrafilter on $\FIN$
and $\mathcal{V}$ is any ultrafilter on a countable index set $I$.
If $\mathcal{U}_{\min,\max}\ge_T\mathcal{V}$,
then there are an infinite block sequence $\tilde{X}$ such that $[\tilde{X}]\in\mathcal{U}$ and
a monotone continuous function $f$ from $\{[X]_{\min,\max}:X\le \tilde{X}\}$ into $\mathcal{P}(I)$
whose restriction to $\{[X]_{\min,\max}:X\le \tilde{X},\ [X]\in\mathcal{U}\}$
has cofinal range in $\mathcal{V}$.
\end{thm}

\begin{proof}
Let $\mathcal{B}$ be the collection of block sequences $X$ such that $[X]\in\mathcal{U}$.
Then $\{[X]:X\in\mathcal{B}\}$ is a base for $\mathcal{U}$.
Let $\mathcal{C}=\{[X]_{\min,\max}:X\in\mathcal{B}\}$.
Then $\mathcal{C}$ is a base for $\mathcal{U}_{\min,\max}$.
For the sake of notation, let $\mathcal{W}$ denote $\mathcal{U}_{\min,\max}$.
Let $\mathcal{V}$ be any ultrafilter on some countable base set $I$ such that $\mathcal{W}\ge_T\mathcal{V}$ and let
$f:\mathcal{W}\ra\mathcal{V}$ be 
 a monotone cofinal map witnessing that $\mathcal{W}\ge_T\mathcal{V}$.
Then  $f\re\mathcal{C}$ is also a monotone cofinal map from $\mathcal{C}$ into $\mathcal{V}$.

In a similar manner as in the proof of Theorem \ref{thm.5}, we construct an $\tilde{X}\in\mathcal{B}$ such that the map $f$ is continuous on
$\{[W]_{\min,\max}:W\in\mathcal{B},\ W\le\tilde{X}\}$.
Let $\lgl  i_n:n<\om\rgl$ be an enumeration of $I$.
Let $X_0= (\{0\},\{1\},\{2\},\dots)$.
Given $X_n$,
take $X_{n+1} \le X_n$ such that,
letting $(x_i^{n+1})_{i<\om}$ denote $X_{n+1}$,
\begin{enumerate}
\item
$\min(x^{n+1}_0)\ge n+1$;
\item
For each finite block  sequence $s\sse\mathcal{P}(n+1)$, for each $k\le n$, 
if there is a $Z\in\mathcal{B}$ such that $\min(Z)\ge n+1$ and $i_k\not\in f([s\cup Z]_{\min,\max})$,
then $i_k\not\in f([s\cup X_n]_{\min,\max})$.
\end{enumerate}

Since $\mathcal{U}$ is block-basic, 
there is a  $Y\in\mathcal{B}$ such that for each $n<\om$,
$Y\le^* X_n$.
Let $l_0=0$ and for each $n<\om$, let $l_{n+1}>l_n$ satisfy $l_{n+1}=\min(y)$ for some $y\in Y$ and $Y/ l_{n+1}\le X_{l_n}$.

Color $[Y]^{[2]}$ as follows:  Let $h((y_0,y_1))=0$ if there is an $n<\om$ such that $\max(y_0)<l_{n}$ and $l_{n+2}\le\min(y_1)$; $1$ otherwise. 
Since $\mathcal{U}$ has the Ramsey property for pairs, 
there is a block-sequence $\tilde{X}\le Y$ such that $h$ is  constant on $[\tilde{X}]^{[2]}$.
$h$ cannot be constantly $1$ on $[W]^{[2]}$ for any block-sequence $W$, since for any block-sequence $(z_k)$,
there will be some $n$ and some $k<k'$ such that $\max(z_k)<l_n$ and $l_{n+2}\le \min(z_{k'})$;
and such a pair will have color $0$.  
Thus, $h$ is constantly $0$ on $[\tilde{X}]^{[2]}$.

Let $(l_{n_j})$  be
 a subsequence of $(l_n)$ such that 
for each $x$ in $\tilde{X}$, 
either $\max(x)<l_{n_{2j+1}}$ or $l_{n_{2j+2}}\le\min(x)$. 
Suppose $W=(w_0,w_1,w_2,\dots)\le\tilde{X}$ and is in $\mathcal{B}$.
Let $C=[W]_{\min,\max}$ and let $i\in I$.
Let $k$ be such that $i=i_k$.
Take $j$ large enough that $k<l_{n_{2j+1}}$ 
and there is an $m$ such that 
$\max(w_m)<l_{n_{2j+1}}$ and $\min(w_{m+1})\ge l_{n_{2j+2}}$.
Note that $W/ l_{n_{2j+1}}
\le\tilde{X}/ l_{n_{2j+1}}
\sse Y/ l_{n_{2j+1}}
=Y/ l_{n_{2j+2}}
\le X_{l_{n_{2j+1}}}$.
Thus,
$i\not\in f(C)$
iff
$i\not\in f([t\cup(W/ l_{n_{2j+1}})]_{\min,\max})$, 
where $t=(w_0,\dots,w_m)$,
iff
$i\not\in f([t\cup X_{l_{n_{2j+1}}}]_{\min,\max})$
iff
$i\not\in f([t\cup(\tilde{X}/ l_{n_{2j+1}})]_{\min,\max})$.
Thus, $f$ is continuous on $\{[W]_{\min,\max}:W\in\mathcal{B},\ W\le\tilde{X}\}$.

In the following natural way, $f\re \{[W]_{\min,\max}:W\in\mathcal{B},\ W\le\tilde{X}\}$ can be extended to a continuous monotone map from $\{[X]_{\min,\max}:X\le\tilde{X}\}$ into $\mathcal{P}(I)$.
For any $X\le\tilde{X}$,
define $f'([X]_{\min,\max})$ to be $\bigcap\{f([W]_{\min,\max}):W\in\mathcal{B},\ W\le\tilde{X},$ and $X\le W\}$.
It follows from the definition of $f'$ and the fact that $f$ is monotone on $\{[W]_{\min,\max}:W\in\mathcal{B},\ W\le\tilde{X}\}$
that $f'$ is monotone on $\{[X]_{\min,\max}:X\le\tilde{X}\}$.
Note also that when restricted to $\{[W]_{\min,\max}:W\in\mathcal{B},\ W\le\tilde{X}\}$
$f'$ is the same as $f$.
Finally, $f'$ is continuous, since 
for any $X\le\tilde{X}$ and any $k<\om$,
$k\in f'([X]_{\min,\max})$ iff
for all $W\in\mathcal{B},\ W\le\tilde{X},$ and $X\le W$,
$k\in f([W]_{\min,\max})$,
and each of these is determined by the initial segment of $W$ lying strictly below some particular $l_{n_{2_j+2}}$ depending  only on $k$.
So a finite amount of information which depends only on $k$ and $X$ determines whether or not $k$ is in $f'([X]_{\min,\max})$.
Hence, $f'$ is continuous.
\end{proof}

Recall the following theorem of Hindman from \cite{Hindman74}, which is useful for constructing block-basic ultrafilters.

\begin{thm}[Hindman's Theorem]\label{thm.Hindman}
For every finite coloring of $\FIN$,
there is an infinite block sequence $X=(x_n)$  of members of $\FIN$ such that the set $[X]$ of all finite unions of members of $X$ is monochromatic.
\end{thm}

\begin{thm}\label{thm.5Tukey}
Assuming CH, there is a block-basic ultrafilter $\mathcal{U}$ on $\FIN$ such that $\mathcal{U}_{\min,\max}<_T\mathcal{U}$ and  
$\mathcal{U}_{\min}$ and $\mathcal{U}_{\max}$ are Tukey incomparable.  (In the following diagram, arrows represent strict Tukey reducibility.)
$$
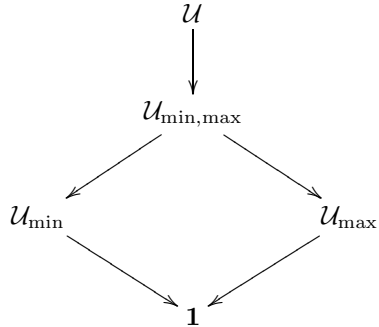
\xymatrix{
&\mathcal{U} \ar[d] &\\
&\mathcal{U}_{\min,\max} \ar[ld] \ar[rd] &\\
\mathcal{U}_{\min}\ar[rd] & & \mathcal{U}_{\max}\ar[ld]\\
&\mathbf{1}}
$$


\end{thm}

\begin{proof}
Recall that for every block-generated ultrafilter $\mathcal{U}$ on $\FIN$, $\mathcal{U}_{\min,\max}\equiv_{RK}\mathcal{U}_{\min}\cdot\mathcal{U}_{\max}$, and by 
Fact \ref{fact.rapid>omom}
 and Corollary \ref{cor.dot},
$\mathcal{U}_{\min}\cdot\mathcal{U}_{\max}\equiv_T\mathcal{U}_{\min}\times\mathcal{U}_{\max}$.
Recall that $\mathcal{U}_{\min}$ and $\mathcal{U}_{\max}$ are Tukey incomparable, since they are non-isomorphic selective ultrafilters.
Thus, it suffices to construct a block-basic ultrafilter
$\mathcal{U}$ on $\FIN$ such that  $\mathcal{U}_{\min,\max}<_T\mathcal{U}$.
Assuming CH, one can  construct a block-basic ultrafilter on $\FIN$ in the standard way (see \cite{Blass87}).

Fix a well-ordering  $\lgl A_{\beta}:\beta<\om_1\rgl$ of $\mathcal{P}(\FIN)$.
By Theorem \ref{thm.5.2}, we can
enumerate   as $\lgl (f_{\beta},\tilde{X}_{\beta}):\beta<\om_1\rgl$,
all
pairs $(f,\tilde{X})$
such that $\tilde{X}\in\FIN^{[\infty]}$ and $f:\{[Z]_{\min,\max}:Z\le\tilde{X}\}\ra\mathcal{P}(\FIN)$
is a monotone continuous function. 
We build a sequence $\lgl S_{\al}:\al<\om_1\rgl$ of elements of $\FIN^{[\infty]}$ such that for each $\al<\om_1$,
\begin{enumerate}
\item[(i)]
For all $\beta<\al$, $S_{\al}\le^* S_{\beta}$;
\item[(ii)]
Either $[S_{\al}]\sse A_{\al}$ or else $[S_{\al}]\cap A_{\al}=\emptyset$;
\item[(iii)] One of the following hold:\\
(a) $[S_{\al}]\cap[\tilde{X}_{\al}]=\emptyset$; or\\
(b) for each $W'\le S_{\al}$, $f_{\al}([W']_{\min,\max})\not\sse [S_{\al}]$;
or\\ 
(c)
$f_{\al}([S_{\al}]_{\min,\max})\cap [S_{\al}]= \emptyset$.
\end{enumerate}

Let $S_0$ be any block sequence such that either $[S_0]\sse A_0$ or else $[S_0]\cap A_0=\emptyset$.
Such an $S_0$ exists by Hindman's Theorem.
At stage $\al$ in the construction, 
let $Y$ be a block sequence  such that
\begin{enumerate}
\item[(i)]
 for all $\beta<\al$, $Y\le^* S_{\beta}$, and
\item[(ii)]
either $[Y]\sse A_{\al}$ or else $[Y]\cap A_{\al}=\emptyset$.
\end{enumerate}
(The standard argument using Hindman's Theorem to find such a $Y$ can be found on p.\ 93 of \cite{Blass87}.)

Now we show there is an $S_{\al}\le Y$ satisfying (iii).
If there is no block sequence $Z\le Y,\tilde{X}_{\al}$, then the domain of $f_{\al}$ is not contained in $\mathcal{U}_{\min,\max}$ for any block-generated ultrafilter $\mathcal{U}$ extending $\{[S_{\beta}]:\beta<\al\}$.
In this case, use Hindman's Theorem to find an $S_{\al}\le Y$ such that $[S_{\al}]\cap[\tilde{X}_{\al}]=\emptyset$.

Now suppose there is a $Z\le Y,\tilde{X}_{\al}$.
If there is a $W\le Z$ such that for each $W'\le W$, $f_{\al}([W']_{\min,\max})\not\sse [W]$, then 
let $S_{\al}=W$.
This ensures that $f_{\al}$ cannot be cofinal 
into any block-generated ultrafilter  extending 
the filter generated by $\{[S_{\beta}]:\beta\le\al\}$, since $f_{\al}$ is monotone.

Otherwise, for each $W\le Z$, there is a $W'\le W$ such that $f_{\al}([W']_{\min,\max})\sse[W]$.
Let $(z_i)$ denote $Z$.
Fix $W=(w_i)$ to be the block sequence where each 
$w_i=z_{3i}\cup z_{3i+1}\cup z_{3i+2}$.
Thus, $W\le Z$.
Fix some $W'\le W$ such that $f_{\al}([W']_{\min,\max})\sse[W]$.
$W'\le W$ means that $W'=(w'_j)$, where each $w'_j=
\bigcup_{i\in I_j}w_i$, where each $I_j$ is some finite set.
Let $m_j=\min(I_j)$ and $k_j=\max(I_j)$.
Let $S_{\al}=(s_j)$, where each
$s_j=z_{3m_j}\cup z_{3k_j+2}$.
Then $\min(s_j)=\min(w'_j)$ and $\max(s_j)=\max(w'_j)$ for all $j<\om$;
so $[W']_{\min,\max}=[S_{\al}]_{\min,\max}$.
Note that $[W]\cap[S_{\al}]=\emptyset$, and $S_{\al}\le Z$.
Note that
for any ultrafilter $\mathcal{U}$ extending $\{[S_{\beta}]:\beta\le\al\}$, $[S_{\al}]_{\min,\max}\in\mathcal{U}_{\min,\max}$.
Hence, $f_{\al}([S_{\al}]_{\min,\max})=f_{\al}([W']_{\min,\max})\sse[W]$ which is disjoint from $[S_{\al}]$.
Thus, the range of $f_{\al}$ will not be contained in $\mathcal{U}$.
By this and the previous two paragraphs, we have satisfied  (iii).

Let $\mathcal{U}$ be the filter generated by $\{[S_{\al}]:\al<\om_1\}$.
Condition (ii) ensures that $\mathcal{U}$ is an ultrafilter which is block-generated.
Condition (iii) ensures that
 $\mathcal{U}_{\min,\max}\not\ge_T\mathcal{U}$, and thus  $\mathcal{U}>_T\mathcal{U}_{\min,\max}$.
\end{proof}

\begin{question}\label{q.UUmin}
If $\mathcal{U}$ is any block-basic ultrafilter, does it follow that $\mathcal{U}>_T\mathcal{U}_{\min,\max}$?
\end{question}

\begin{rem}
Note that the proof of Theorem \ref{thm.5Tukey}
 shows that the generic filter for the forcing notion $(\FIN^{[\infty]},\le^*)$ adjoins a block-basic ultrafilter $\mathcal{U}$ on $\FIN$ with the properties stated in Theorem \ref{thm.blockbasicequiv}.
On the other hand, an argument analogous with the case of selective ultrafilters  on $\om$ (see Theorem 4.9 of Todorcevic appearing in  \cite{Farah98})
shows that if there is a supercompact cardinal,
then every block-basic ultrafilter $\mathcal{U}$ on $\FIN$ is generic over $L(\mathbb{R})$ for the forcing notion
$(\FIN^{[\infty]},\le^*)$.
Thus, the conclusion of Theorem 4.9 in \cite{Farah98}
is true for any block-basic ultrafilter $\mathcal{U}$ on $\FIN$ assuming the existence of a supercompact cardinal.
This leads us also to the following related problem.
\end{rem}

\begin{problem}\label{problem.LR5}
Assume the existence of a supercompact cardinal.
Let $\mathcal{U}$ be an arbitrary block-basic ultrafilter on $\FIN$.
Show that the inner model $L(\mathbb{R})[\mathcal{U}]$
has exactly five Tukey types of ultrafilters on a countable index set.
\end{problem}

This problem is based on the $\mathcal{U}$-version of Taylor's canonical Ramsey Theorem for $\FIN$
stating that for each map $f:\FIN\ra\om$,
there is an $[X]\in\mathcal{U}$ such that $f\re[X]$ is equivalent to one of the five mappings:
constant, identity, min, max, (min,max) (see \cite{Argyros/TodorcevicBK}, \cite{Taylor70}).
If the answer to this problem is positive, 
then one can look at ultrafilters $\mathcal{U}$ on the index set $\FIN_k$ $(k=1,2,3,\dots)$ with analogous Ramsey-theoretic properties
whose corresponding inner models $L(\mathbb{R})[\mathcal{U}]$ have different finite numbers of Tukey types.
This will of course be based on Gower's Theorem for $\FIN_k$ and Lopez-Abad's canonical Ramsey Theorem for $\FIN_k$
(see \cite{Argyros/TodorcevicBK}, \cite{Lopez-Abad07}, \cite{Gowers02}).
For example, for a block-basic ultrafilter $\mathcal{U}$ on $\FIN_2$,
one could expect exactly 43 Tukey types of ultrafilters in $L(\mathbb{R})[\mathcal{U}]$.

The following is a subproblem of Problem \ref{problem.LR5}.

\begin{question}\label{q.5?}
Is it true that for each block-basic $\mathcal{U}$, there are no Tukey types (a) strictly between $\mathcal{U}$ and $\mathcal{U}_{\min,\max}$,
(b) strictly between $\mathcal{U}_{\min,\max}$ and $\mathcal{U}_{\min}$,
and (c) strictly between $\mathcal{U}_{\min,\max}$ and $\mathcal{U}_{\max}$?
\end{question}

\begin{question}
Are there block-basic  ultrafilters $\mathcal{U},\mathcal{V}$ on $\FIN$ which are Tukey equivalent but RK incomparable?
\end{question}


\section{A characterization of ultrafilters which are not of Tukey top degree}

In this section we investigate Isbell's question of whether ZFC implies that there is always an ultrafilter which does not have top Tukey degree. 
It will be useful here to consider the directed partial ordering $\contains^*$ on ultrafilters as well as the one we have been considering all along, namely $\contains$.
We note that always $(\mathcal{U},\contains)\le_T(\mathcal{U},\contains^*)$; for any subset $\mathcal{X}\sse\mathcal{U}$ which is unbounded in $(\mathcal{U},\contains^*)$ is also unbounded in $(\mathcal{U},\contains)$, so the identity map
$id_{\mathcal{U}}:(\mathcal{U},\contains)\ra(\mathcal{U},\contains^*)$ is a Tukey map.
Hence, if $(\mathcal{U},\contains^*)<_T[\mathfrak{c}]^{\om}$,
then also $(\mathcal{U},\contains)<_T[\mathfrak{c}]^{\om}$.
Milovich showed in \cite{Milovich08}
that for any ultrafilter $\mathcal{U}$, there is an ultrafilter $\mathcal{W}$ such that $(\mathcal{W},\contains^*)\le_T(\mathcal{U},\contains)$.
Thus, there is an ultrafilter $\mathcal{U}$ such that $(\mathcal{U},\contains)<_T([\mathfrak{c}]^{<\om},\sse)$ if and only if there is an ultrafilter $\mathcal{W}$ such that 
$(\mathcal{W},\contains^*)<_T([\mathfrak{c}]^{<\om},\sse)$.

CH implies the existence of p-points, which solves Isbell's problem, since p-points  have Tukey type strictly below the top, by Corollary \ref{cor.pptnottop}.
Thus,  we now assume $\neg$CH throughout this section.
Assuming $\neg$CH, the following combinatorial principle holds.

\begin{defn}[Todorcevic]\label{defn.diamond}
$\lozenge_{[\mathfrak{c}]^{\om}}$ is the statement: There exist sets $S_A\sse A$, $A\in[\mathfrak{c}]^{\om}$, such that for each $X\sse\mathfrak{c}$, 
$\{A\in[\mathfrak{c}]^{\om}:X\cap A=S_A\}$ is stationary in $[\mathfrak{c}]^{\om}$.
\end{defn}

This implies the next combinatorial principle  in the same way that the standard $\lozenge$ implies $\lozenge^-$.

\begin{defn}[Todorcevic]\label{defn.diamondminus}
$\lozenge^-_{[[\om]^{\om}]^{\om}}$ is the statement:
There exist ordered pairs $(\mathcal{U}_A,\mathcal{X}_A)$,
where $A\in[[\om]^{\om}]^{\om}$ and $\mathcal{X}_A\sse\mathcal{U}_A\sse A$,
such that for each pair $(\mathcal{U},\mathcal{X})$ with $\mathcal{X}\sse\mathcal{U}$ and $\mathcal{X},\mathcal{U}\in[[\om]^{\om}]^{\mathfrak{c}}$,
$\{A\in[[\om]^{\om}]^{\om}:\mathcal{U}_A=\mathcal{U}\cap A$, $\mathcal{X}_A=\mathcal{X}\cap A\}$ is stationary in $[[\om]^{\om}]^{\om}$.
\end{defn}


We now proceed to define some dense subsets of $[\om]^{\om}$, denoted $D_A$ and $D_A'$  which can be used to give conditions under which an ultrafilter on $\om$ is a p-point, and other conditions under which  it has Tukey type less than $[\mathfrak{c}]^{<\om}$.

For the rest of this section, fix a $\lozenge^-_{[[\om]^{\om}]^{\om}}$ sequence
$(\mathcal{U}_A,\mathcal{X}_A)$,
where $A\in[[\om]^{\om}]^{\om}$.

\begin{defn}\label{defn.DA}
Let $P_A=\{W\in[\om]^{\om}:\exists X\in\mathcal{U}_A(W\cap X=\emptyset)\}$, and let
$Q_A=\{W\in[\om]^{\om}:\exists(B_n)_{n<\om}\sse\mathcal{X}_A(\forall n<\om,\ W\sse^* B_n)\}$.
Let $D_A=P_A\cup Q_A$.
\end{defn}

\begin{fact}\label{fact.DAdense}
For each $A\in[[\om]^{\om}]^{\om}$,
$D_A$ is dense open in the partial ordering $([\om]^{\om},\contains)$.
\end{fact}

\begin{proof}
Let $Y\in[\om]^{\om}$.
Suppose 
$\mathcal{U}_A$ does not generate a nonprincipal filter.
Then
there are $U,V\in\mathcal{U}_A$ such that $|U\cap V|<\om$.
Then either $|Y\setminus U|=\om$ or $|Y\setminus V|=\om$.
Thus, there is a $W\in[Y]^{\om}$ such that for some $X\in\mathcal{U}_A$,
$W\cap X=\emptyset$.  Hence $W\in P_A$, and moreover, any $W'\in[W]^{\om}$ is also in $P_A$.
Suppose that $\mathcal{U}_A$  generates a nonprincipal filter.
Then
for any $U,V\in\mathcal{U}_A$,
$U$ and $V$ have infinite intersection.
If $Y\not\in\lgl\mathcal{U}_A\rgl^+$, then there is an $X\in\mathcal{U}_A$ such that $|Y\cap X|<\om$.
So $W=Y\setminus  X\in P_A$, and any $W'\in[W]^{\om}$ is also in $P_A$.
Otherwise, $Y\in\lgl\mathcal{U}_A\rgl^+$.  Then there is a $W\in[Y]^{\om}$ such that for each $B\in\mathcal{U}_A$, 
$W\sse^* B$, since $|\mathcal{U}_A|\le\om$.
Thus, $W\in Q_A$.  Moreover, any $W'\in[W]^{\om}$ is also in $Q_A$.
Therefore, $D_A$ is dense open in $[\om]^{\om}$.
\end{proof}

\begin{fact}\label{fact.necessary}
For any nonprincipal ultrafilter $\mathcal{U}$, 
$\{A\in[[\om]^{\om}]^{\om}:\mathcal{U}\cap D_A\ne\emptyset\}$ is stationary.
\end{fact}

\begin{proof}
Let $\mathcal{U}$ be an ultrafilter and suppose that $\{A\in[[\om]^{\om}]^{\om}:\mathcal{U}\cap D_A\ne\emptyset\}$ is not stationary.
Then $\{A\in[[\om]^{\om}]^{\om}:\mathcal{U}\cap D_A =\emptyset\}$ contains a club set, call it $\mathcal{C}$.
Let $\mathcal{X}=\bigcup\{\mathcal{X}_A:A\in\mathcal{C}\}$.
Let $X\in\mathcal{U}$.  
There are club many $A\in[[\om]^{\om}]^{\om}$ with $X\in A$.
Thus, there are club many $A$ with $(A,\mathcal{U}\cap A)\prec([\om]^{\om},\mathcal{U})$.
By $\lozenge^-_{[[\om]^{\om}]^{\om}}$, 
there is an $A\in[\om]^{\om}$ with $X\in A$ such that $\mathcal{U}\cap A=\mathcal{U}_A$ and $\mathcal{U}\cap A=\mathcal{X}_A$.
Therefore, $\mathcal{U}\sse\mathcal{X}$.

We claim that for each $\mathcal{Y}\in[\mathcal{U}]^{\om}$,
there is no $X\in\mathcal{U}$ such that $X\sse^* Y$ for each $Y\in\mathcal{Y}$.
Take an $A\in\mathcal{C}$ containing the Fr\'{e}chet filter with $\mathcal{Y}\sse A$ such that $\mathcal{U}\cap D_A=\emptyset$, $\mathcal{U}_A=\mathcal{U}\cap A$, and $\mathcal{X}_A=\mathcal{U}\cap A$.
Then $\mathcal{Y}\sse\mathcal{U}\cap A=\mathcal{X}_A$ and for
each infinite subset $\mathcal{Z}$ of $\mathcal{Y}$ in $\mathcal{U}$,
there is no $U\in\mathcal{U}$ such that $U\sse^* Z$ for all $Z\in\mathcal{Z}$, 
since $Q_A\cap\mathcal{U}=\emptyset$.
Contradiction,
since $\mathcal{U}$ contains the Fr\'{e}chet filter and $\om\in Q_A\cap\mathcal{U}\cap A$.
\end{proof}

\begin{fact}\label{fact.clubnontop}
If $\mathcal{U}$ is an ultrafilter and $\mathcal{U}\cap D_A\ne\emptyset$ for club many $A\in[[\om]^{\om}]^{\om}$,
then $\mathcal{U}<_T\mathcal{U}_{\mathrm{top}}$.
\end{fact}

\begin{proof}
Let $\mathcal{X}\in[\mathcal{U}]^{\mathfrak{c}}$.
$\{A\in[[\om]^{\om}]^{\om}:(A,\mathcal{U}\cap A,\mathcal{X}\cap A)\prec([\om]^{\om},\mathcal{U},\mathcal{X})\}$ is club in $[[\om]^{\om}]^{\om}$.
$\{A\in [[\om]^{\om}]^{\om}: \mathcal{U}\cap A=\mathcal{U}_A,\ \mathcal{X}\cap A=\mathcal{X}_A\}$ is stationary.
If $\mathcal{U}\cap D_A\ne\emptyset$ for club many $A$,
then there are stationary many $A$ such that $\mathcal{U}\cap A=\mathcal{U}_A$, $\mathcal{X}\cap A=\mathcal{X}_A$;
and either there is a $U\in\mathcal{U}_A$ and a $W\in\mathcal{U}$ such that $U\cap W=\emptyset$,
which is impossible,
or else there is a $W\in\mathcal{U}$ and $(B_n)_{n<\om}\sse\mathcal{X}_A=A\cap \mathcal{X}$ such that for each $n<\om$,
$W\sse^* B_n$.
Therefore, $\mathcal{U}$ is not of Tukey top degree. 
\end{proof}

\begin{fact}\label{fact.ppointhitsall}
If $\mathcal{U}$ is a p-point, then $\mathcal{U}\cap D_A\ne\emptyset$ for all $A\in[[\om]^{\om}]^{\om}$.
\end{fact}

\begin{proof}
Let $A\in[[\om]^{\om}]^{\om}$ be given.
If $\mathcal{U}_A\not\sse \mathcal{U}$,
then  taking an $X\in\mathcal{U}_A\setminus\mathcal{U}$,
we have
$\om\setminus X\in\mathcal{U}\cap P_A$.
If $\mathcal{U}_A\sse\mathcal{U}$,
then since $\mathcal{X}_A$ is countable, there is a $W\in\mathcal{U}$ which is
almost contained in every member of
 $\mathcal{X}_A$.
Hence, $W\in\mathcal{U}\cap Q_A$.
\end{proof}

Let $Q'_A=\{W\in[\om]^{\om}: \forall X\in\mathcal{X}_A(W\sse^*X)\}$.
Let $D'_A=P_A\cup Q'_A$.
By the same proof as for $D_A$, we see that $D'_A$ is dense open in $[\om]^{\om}$.

\begin{fact}\label{fact.ppointclub}
If $\mathcal{U}\cap D'_A\ne\emptyset$ for club many $A$,
then $\mathcal{U}$ is a p-point.
\end{fact}

\begin{proof}
Suppose that $\mathcal{C}$ is club in $[[\om]^{\om}]^{\om}$ and for each $A\in\mathcal{C}$,
$\mathcal{U}\cap D'_A\ne\emptyset$.
Let $\mathcal{Y}\in[\mathcal{U}]^{\om}$.
Take $A$ such that $\mathcal{Y}\sse A$, $(A,\mathcal{U}\cap A,\mathcal{U}\cap A)\prec ([\om]^{\om},\mathcal{U},\mathcal{U})$,
 $\mathcal{U}_A=\mathcal{U}\cap A$,
 $\mathcal{X}_A=\mathcal{U}\cap A$,
 and $\mathcal{U}\cap D'_A\ne\emptyset$.
Then $\mathcal{Y}\sse\mathcal{U}\cap A=\mathcal{X}_A$.
So since there is a $W\in\mathcal{U}$ such that for each $X\in \mathcal{X}_A$, $W\sse^* X$,
there is a 
$U\in\mathcal{U}$ such that $U\sse^* Y$ for each $Y\in
\mathcal{Y}$.
\end{proof}

\begin{rem}\label{rem.buildnontop}
Assuming $\neg$CH and that there are no p-points (the remaining open case for Isbell's Problem), 
if Isbell's Problem is solved in the affirmative with an ultrafilter $\mathcal{U}<_T[\mathfrak{c}]^{<\om}$, 
then $\mathcal{U}$ {\em must} have the following properties.
\begin{enumerate}
\item
$\mathcal{U}\cap D_A\ne\emptyset$ for club many $A\in[[\om]^{\om}]^{\om}$.
\item
The collection of $A\in[[\om]^{\om}]^{\om}$ such that $\mathcal{U}\cap D'_A\ne\emptyset$ does not contain a club set.
\end{enumerate}
To solve Isbell's Problem under the assumptions $\neg$CH and there are no p-points, it suffices to find an ultrafilter $\mathcal{U}$ such that (1) holds and 
\begin{enumerate}
\item[(3)]
There is some $A\in[[\om]^{\om}]^{\om}$ such that $\mathcal{U}\cap D_A=\emptyset$.
\end{enumerate}
\end{rem}

\begin{question}
Assume $\neg$CH and there are no p-points.
Can we use these dense sets, or similar ones, 
to obtain 
\begin{enumerate}
\item
an ultrafilter which is not Tukey top?
\item
an ultrafilter which is not Tukey top but also is not basically generated?
\end{enumerate}
\end{question}


\section{Concluding remarks and problems}\label{sec.problems}

Recall that the properties of p-point and rapid are preserved under Rudin-Keisler reducibility.  

\begin{question}
Which properties of ultrafilters are  preserved under 
Tukey reducibility?
\end{question}

By Theorem \ref{thm.7},
if a p-point $\mathcal{U}\ge_T\om^{\om}$,
then $\mathcal{U}\equiv_T\mathcal{U}\cdot\mathcal{U}$, which is not a p-point, so the property of being a p-point is not preserved by Tukey reducibility.
However, we may ask the following.

\begin{question}
If $\mathcal{U}$ is a p-point and $\mathcal{U}\ge_T\mathcal{V}$, then is there a p-point $\mathcal{W}$ such that $\mathcal{W}\equiv_T\mathcal{V}$?
\end{question}

\begin{question}
Which lattices can be embedded into the Tukey types of p-points?
In particular, are there two Tukey incomparable p-points which have no p-point as a common Tukey upper bound?
\end{question}

\begin{question}
Are there two Tukey non-comparable ultrafilters whose least upper bound is the top Tukey type?
\end{question}



\begin{question}
Does every Tukey minimal type contain a selective ultrafilter?
\end{question}




\begin{question}
What is the structure of the Rudin-Keisler types within the top Tukey type?
\end{question}

\bibliographystyle{amsplain}
\bibliography{references1}

\end{document}